\documentclass{article}
\usepackage{a4}
\usepackage{paralist}
\usepackage{graphicx, psfrag}
\usepackage{pgfplotstable}
\usepackage{subcaption,tikz}

\usepackage{amsmath,amssymb}
\usepackage{amsthm}
\usepackage{amsfonts}
\usepackage{booktabs}
\usepackage{graphicx}
\usepackage{diagbox}

\usepackage{multirow}
\usepackage{hyperref}

\usepackage{algorithmic}

\usepackage[ulem=normalem,draft]{changes}
\usepackage[ruled, longend, linesnumbered, noline]{algorithm2e}
\SetEndCharOfAlgoLine{}

\usepackage{algorithmic}

\title{State-based nested iteration solution \\ 
         of optimal control problems with PDE constraints\footnote{This paper
           is dedicated to our friend and colleague Arnd R\"osch in
         occasion of his 60th birthday.}}
        
\author{Ulrich~Langer\footnote{Institute of Computational Mathematics,
    Johannes Kepler University Linz, Altenberger Stra{\ss}e 69, 4040 Linz,
    Austria, Email: ulanger@numa.uni-linz.ac.at},
  \;
  Richard~L\"oscher\footnote{Institut f\"{u}r Angewandte Mathematik,
    Technische Universit\"{a}t Graz, Steyrergasse 30, 8010 Graz, Austria,
    Email: loescher@math.tugraz.at}, 
  \; Olaf~Steinbach\footnote{Institut f\"{u}r Angewandte Mathematik,
    Technische Universit\"{a}t Graz, Steyrergasse 30, 8010 Graz, Austria,
    Email: o.steinbach@tugraz.at}, 
  \; Huidong~Yang\footnote{Faculty of Mathematics, University of Vienna,
  and 
  Doppler Laboratory for Mathematical Modeling and Simulation of Next
  Generations of Ultrasound Devices (MaMSi),
  Oskar--Morgenstern--Platz 1, A-1090 Wien, Austria, Email: huidong.yang@univie.ac.at}
}  

\date{}

\newcommand{\norm}[1]{\|#1\|}
\newcommand{\skpr}[1]{\langle#1\rangle}
\pgfplotsset{compat=1.17}

\newtheorem{theorem}{Theorem}
\newtheorem{lemma}{Lemma}
\newtheorem{cor}{Corollary}

\newtheorem{remark}{Remark}

\numberwithin{equation}{section} 

\begin{document}

\maketitle

\begin{abstract}
  We consider an abstract framework for the numerical solution of optimal
  control problems (OCPs) subject to partial differential equations (PDEs). 
  Examples include not only the distributed control of elliptic PDEs such
  as the Poisson equation discussed in this paper in detail but also
  parabolic and hyperbolic equations. The approach covers the standard
  $L^2$ setting as well as the more recent energy regularization, also
  including state and control constraints. We discretize OCPs subject to
  parabolic or hyperbolic PDEs by means of space-time finite elements
  similar as in the elliptic case. We discuss regularization and finite
  element error estimates, and derive an optimal relation between the
  regularization parameter and the finite element mesh size in order to
  balance the accuracy, and the energy costs for the corresponding control.
  Finally, we also discuss the efficient solution of the resulting systems
  of algebraic equations, and their use in a state-based nested iteration
  procedure that allows us to compute finite element approximations to the
  state and the control in asymptotically optimal complexity.
  The numerical results illustrate the theoretical findings quantitatively.
\end{abstract} 

\noindent
\begin{keywords} 
PDE constrained optimal control problems, 
finite element method, 
error estimates, solvers, nested iteration
\end{keywords}

\medskip

\noindent
\begin{msc}
49J20,  
49M05,  
35J05,  
65M60,  
65N22,   
65F10 
\end{msc}

\section{Introduction, motivation, and preliminaries}
\label{Section:Introduction}
Since Lions' pioneering monograph \cite{Lions:1968Book} on the optimal
control of systems described by partial differential equations (PDEs) of
elliptic, parabolic, or hyperbolic types, the investigation of optimal
control problems (OCPs) for PDEs and their numerical solution have
developed into a well-established research field in Applied Mathematics
with many applications in different areas in science and engineering.
Since then the development of the mathematical analysis on OCPs for PDEs
is documented by a huge number of publications. We here only refer the
reader to the books \cite{BorziSchulz2011Book,DeLosReyes:2015Book,
  HinzePinnauUlbrichUlbrich:2009Book,Troeltzsch:2010Book}, 
the collections \cite{HarbirKouriLacassRidzal2018Collection,
  HerzogHeinkenschlossKaliseStadlerTrelat:2022Proceedings,
  LeugeringEngellGriewankEtAl2012Collection}, 
and the survey paper \cite{BorziSchulz:2009SIAMreview}. Tracking-type
OCPs for PDEs can be posed as follows: Find the optimal control
$u_\varrho \in U$ and the corresponding state $y_\varrho \in Y$
minimizing the cost functional 
\begin{equation}\label{Eqn:Introduction:AbstractOCP1}
  {\mathcal{J}}(y_\varrho,u_\varrho) \, = \, \frac{1}{2} \,
  \| y_\varrho - \overline{y} \|_{H_Y}^2 + \frac{1}{2} \, \varrho \,
  \| u_\varrho \|^2_U
\end{equation}
subject to (s.t.) the state equation
\begin{equation}\label{Eqn:Introduction:StateEquation1}
B y_\varrho = u_\varrho \quad \mbox{in} \; \; U,
\end{equation}
where we are thinking about elliptic (e.g., Poisson's equation: $B=-\Delta$), 
parabolic (e.g., heat equation: $B=\partial_t-\Delta_x$), and hyperbolic
(e.g., wave equation: $B=\partial_{tt}-\Delta_x$) PDEs or systems of such PDEs 
with appropriate boundary and initial conditions. Here $\overline{y} \in H_Y$
denotes the given desired state (target) that we want to track as close as
possible, and $\varrho > 0$ is a suitable chosen regularization parameter 
that also defines the energy cost $\| u_\varrho \|^2_U$ of the control
$u_\varrho \in U$ appearing as right-hand side of the state equation
\eqref{Eqn:Introduction:StateEquation1}.
In this paper, the spaces $Y$, $U$, and $H_Y$ are Hilbert spaces 
with $Y \subset H_Y \subset Y^*$ being a Gelfand triple, and $B: Y \to U$ 
is always assumed to be an isomorphism. Of course, one can consider also
nonlinear PDEs or systems of PDEs represented then by a nonlinear operator
$B$ acting between Banach spaces in general; 
see, e.g., \cite{Troeltzsch:2010Book}. 
The practical realization of the control $u_\varrho$ sometimes requires
additional, so-called box-constraints imposed on the control, i.e., we look
for some optimal control $u_\varrho \in U_\text{\tiny ad} = K_c \subset U$ in
a non-empty, closed, and convex subset $K_c$ of $U$. Similarly, we can also
request box-constraints imposed on the state, i.e. we look for  
$y_\varrho \in Y_\text{\tiny ad} = K_s \subset Y$ in a non-empty, closed, 
and convex subset $K_s$ of $Y$. These are the main assumptions ensuring  
existence and uniqueness of an optimal solution
$(y_\varrho,u_\varrho) \in Y \times U$ of the optimal control problem
\eqref{Eqn:Introduction:AbstractOCP1}-\eqref{Eqn:Introduction:StateEquation1}
or the coresponding box-constrained problems where $U$ or $Y$ is replaced
by $U_\text{\tiny ad}$ or $Y_\text{\tiny ad}$; see, e.g.,
\cite{DeLosReyes:2015Book,Lions:1968Book,Troeltzsch:2010Book}.

One of the most simple examples of such kind of tracking-type OCPs is the 
distributed optimal control of the Poisson equation with $L^2$ regularization: 
Find  $u_\varrho \in U = L^2(\Omega)$ and  $y_\varrho \in Y = H^1_0(\Omega)$
minimizing the cost functional 
\begin{equation}\label{Eqn:Introduction:OCPwithL^2Regularization}
  {\mathcal{J}}(y_\varrho,u_\varrho) = \frac{1}{2} \,
  \| y_\varrho - \overline{y} \|_{H_Y=L^2(\Omega)}^2 + \frac{1}{2} \, \varrho \,
  \| u_\varrho \|_{U=L^2(\Omega)}^2 
\end{equation}
s.t. the Dirichlet boundary value problem for the Poisson equation
\begin{equation}\label{Eqn:Introduction:DBVP4Poisson*}
  - \Delta y_\varrho = u_\varrho \quad \mbox{in} \; \Omega, \quad
  y_\varrho = 0 \quad \mbox{on} \; \partial \Omega,
\end{equation}
where $\Omega \subset \mathbb{R}^d$ is a bounded Lipschitz domain 
with boundary $\partial \Omega$, $d=1,2,3$.
It is well known that the solution of the OCP
\eqref{Eqn:Introduction:OCPwithL^2Regularization}--\eqref{Eqn:Introduction:DBVP4Poisson*} is characterized by the gradient equation
\begin{equation}\label{Eqn:Introduction:L2:GradientEquation}
  p_\varrho + \varrho \, u_\varrho = 0 \quad \mbox{in} \; \Omega,
\end{equation}
where $p_\varrho$ solves the adjoint Dirichlet boundary value problem
\begin{equation}\label{Eqn:Introduction:L2:AdjointProblem}
  - \Delta p_\varrho = y_\varrho - \overline{y} \quad
  \mbox{in} \; \Omega, \quad
  p_\varrho = 0 \quad \mbox{on} \; \partial \Omega .
\end{equation}
When inserting $u_\varrho = - \Delta y_\varrho$ into the gradient equation
\eqref{Eqn:Introduction:L2:GradientEquation}, we get
$p_\varrho = \varrho \, \Delta y_\varrho$. Hence, we have to solve the
BiLaplace equation
\begin{equation}\label{Eqn:Introduction:L2:BiLaplace}
  \varrho \, \Delta^2 y_\varrho + y_\varrho = \overline{y}
  \quad \mbox{in} \; \Omega, \quad
  y_\varrho = \Delta y_\varrho =0 \quad
  \mbox{on} \; \partial \Omega .
\end{equation}
Since the operator $B = - \Delta$ is an isomorphism from 
$H^1_0(\Omega,\Delta) = \{y \in H_0^1(\Omega): \Delta u \in L^2(\Omega)\}$
onto $U=L^2(\Omega)$, the natural choice for the state space would be 
$Y=H^1_0(\Omega,\Delta)$ rather than $Y=H_0^1(\Omega)$ in the case of
$L^2$ regularization. On the other hand, if we choose $Y=H_0^1(\Omega)$
as state space, then we can permit controls from $U = H^{-1}(\Omega)$,
and now $B: H_0^1(\Omega) \rightarrow H^{-1}(\Omega)$ is an isomorphism as well.
Therefore, instead of using the $L^2$ regularization
$\| u_\varrho \|_{L^2(\Omega)}^2$ in
\eqref{Eqn:Introduction:OCPwithL^2Regularization},
we may also consider the energy
regularization $\| u_\varrho \|^2_{H^{-1}(\Omega)}$ to minimize
\begin{equation}\label{Eqn:Introduction:EnergyCostFunctional}
  {\mathcal{J}}(y_\varrho, u_\varrho) =
  \frac{1}{2} \int_\Omega [y_\varrho(x)-\overline{y}(x)]^2 \, dx +
  \frac{1}{2} \, \varrho \, \| u_\varrho \|^2_{H^{-1}(\Omega)}
\end{equation}
subject to the Poisson equation \eqref{Eqn:Introduction:DBVP4Poisson*}.
By duality we have
$ \| u_\varrho \|_{H^{-1}(\Omega)} = \| \nabla y_\varrho \|_{L^2(\Omega)}$, and
hence we can write \eqref{Eqn:Introduction:EnergyCostFunctional} as the
reduced cost functional
\begin{equation}\label{Eqn:Introduction:ReducedEnergyCostFunctional}
  \widetilde{\mathcal{J}}(y_\varrho) =
  \frac{1}{2} \int_\Omega [y_\varrho(x)-\overline{y}(x)]^2 \, dx +
  \frac{1}{2} \, \varrho \, \int_\Omega |\nabla y_\varrho(x)|^2 \, dx,
\end{equation}
whose minimizer is given as the unique solution of the gradient equation
\begin{equation}\label{Eqn:Introduction:EnergyGradient}
  - \varrho \, \Delta y_\varrho + y_\varrho = \overline{y}
  \quad \mbox{in} \; \Omega, \quad y_\varrho = 0 \quad
  \mbox{on} \; \partial \Omega .
\end{equation}
To underline the differences in considering the optimal control problems 
\eqref{Eqn:Introduction:OCPwithL^2Regularization} and
\eqref{Eqn:Introduction:EnergyCostFunctional} subject to
\eqref{Eqn:Introduction:DBVP4Poisson*}, i.e.,
when measuring the control $u_\varrho$ either in $L^2(\Omega)$ or in
$H^{-1}(\Omega)$, let us consider three simple examples. Therefore, we will
study three different target functions $\overline{y}$ on the
one-dimensional ($d=1$) domain $\Omega = (0,1)$, 
namely the regular target
\begin{equation}\label{Eqn:Introduction:target-u1}
  \overline{y}_1(x) = 4x(1-x) \quad \mbox{for} \; x \in (0,1)
\end{equation}
with $\overline{y}_1 \in H^1_0(\Omega) \cap H^2(0,1)$, 
the piecewise linear target
\begin{equation}\label{Eqn:Introduction:target-u2}
  \overline{y}_2(x) =\begin{cases}
    1,& x=0.5,\\
    0,& x\in (0,0.25)\cup (0.75,1),\\
    \text{piecewise linear}, & \text{else}, 
  \end{cases}
\end{equation}
belonging to $H^1_0(\Omega)\cap H^s(\Omega)$ for all $s<1.5$,
and the discontinuous target
\begin{equation}\label{Eqn:Introduction:target-u3}
  \overline{y}_3(x) = \begin{cases}
    1,& x\in (0.25,0.75),\\
    0,& \text{else},
  \end{cases}
\end{equation}
with $\overline{y}_3 \in H^s(\Omega)$, $s<0.5$. In all of these cases, we can
solve both the gradient equation \eqref{Eqn:Introduction:EnergyGradient} and
the BiLaplace equation \eqref{Eqn:Introduction:L2:BiLaplace} analytically in
order to compute the state functions $y_{i,\varrho}$, $i=1,2,3$
for different values of the relaxation or cost parameter $\varrho$, see
Fig.~\ref{Fig:Introduction:states-1D}. There we also plot the
errors $\| y_{i,\varrho} - \overline{y}_i \|_{L^2(\Omega)}$ for both
regularization norms as a function of the regularization parameter
$\varrho$. We observe that in all cases the error
$\| y_{i,\varrho} - \overline{y}_i \|_{L^2(\Omega)}$ is smaller when using
energy regularization in $H^{-1}(\Omega)$ instead of using the
regularization in $L^2(\Omega)$. In fact, the errors coincide when
considering $\varrho_{L^2} = \varrho_{H^{-1}}^2$. Note that the related
regularization error estimates were already shown in
\cite{NeumuellerSteinbach:2021M2AS}.

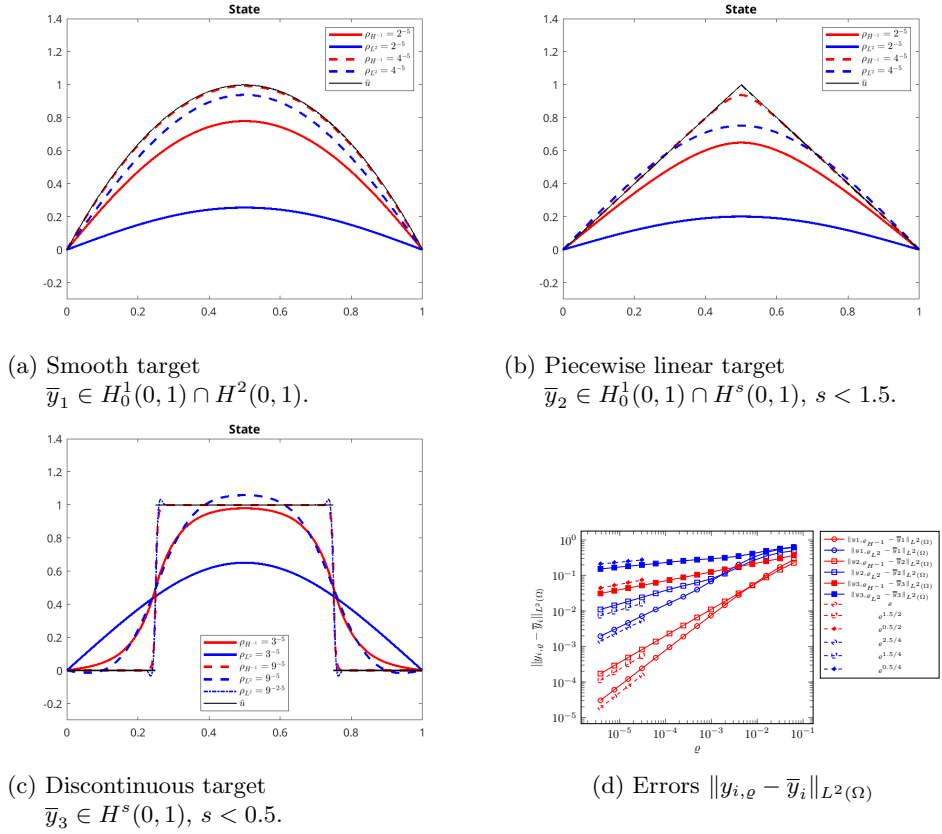
\begin{figure}[h]
  \centering
  \begin{subfigure}{0.48\textwidth}
    \centering
    \includegraphics[width=\textwidth]{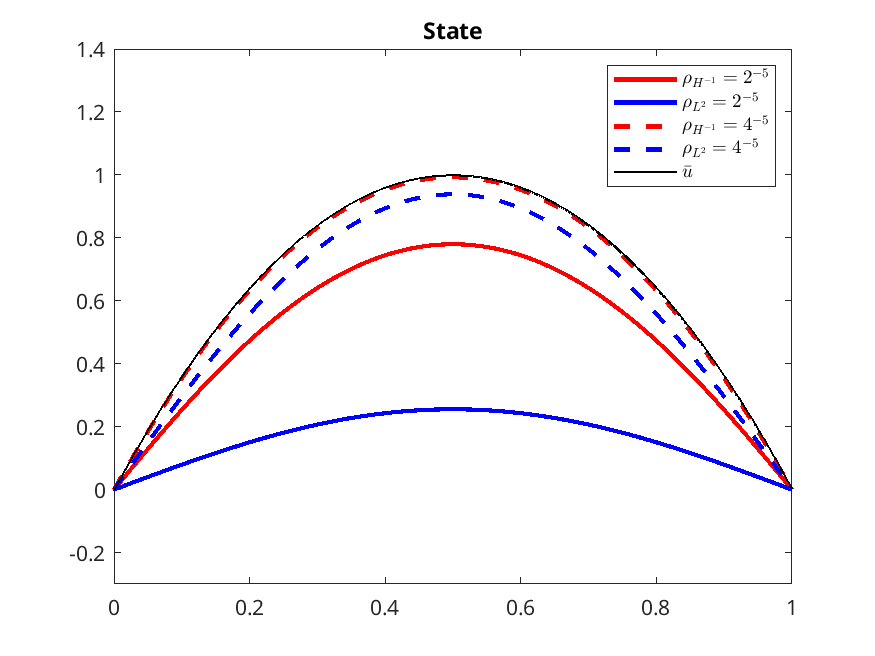}
    \caption{Smooth target \\
      \phantom{(a)} $\overline{y}_1 \in H^1_0(0,1) \cap H^2(0,1)$.}
  \end{subfigure}
  \hfill
  \begin{subfigure}{0.48\textwidth}
    \centering
    \includegraphics[width=\textwidth]{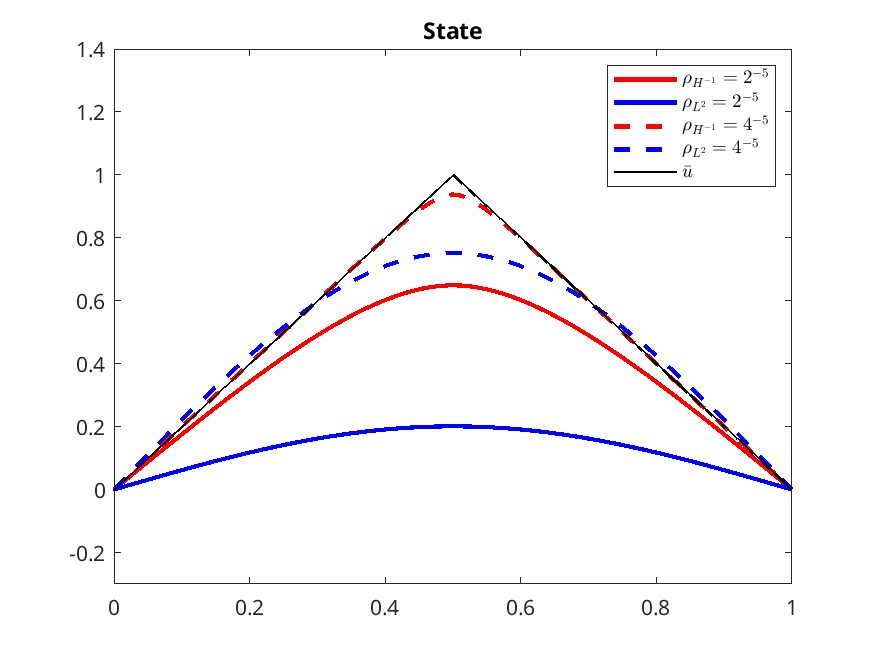}
    \caption{Piecewise linear target \\ 
      \phantom{(b)} $\overline{y}_2  \in H^1_0(0,1) \cap
      H^s(0,1)$, $s < 1.5$.}
  \end{subfigure}
  \\
  \begin{subfigure}[t]{0.48\textwidth}
    \centering
    \includegraphics[width=\textwidth]{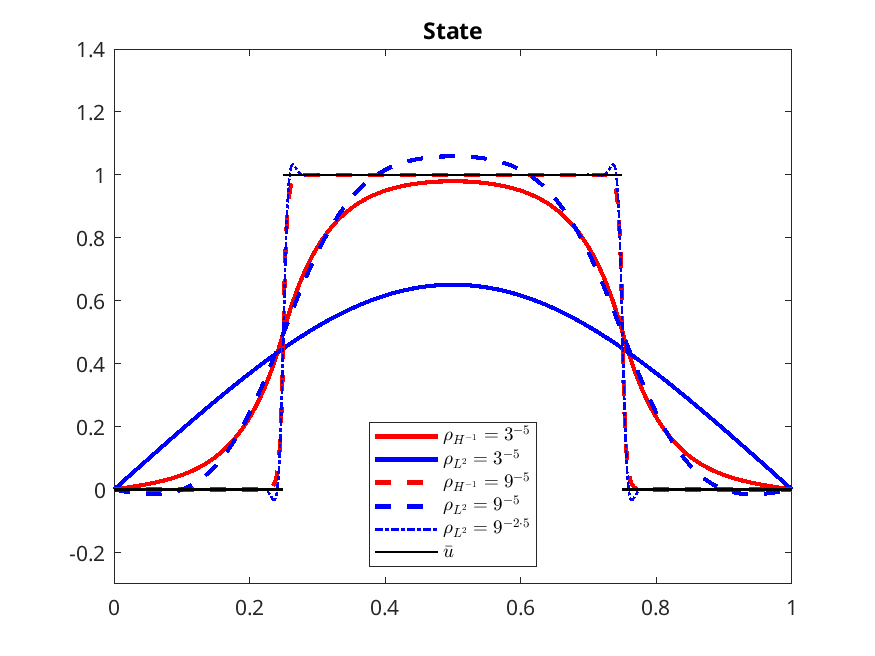}
    \caption{Discontinuous target \\
      \phantom{(c)} $\overline{y}_3 \in H^s(0,1)$, $s<0.5$.}
  \end{subfigure}
  \hfill
  \begin{subfigure}[t]{0.48\textwidth}
    \centering
    \begin{tikzpicture}[scale=0.45]
      \begin{axis}[
        xmode = log,
        ymode = log,
        xlabel=$\varrho$,
        ylabel=$\| y_{i,\varrho}-\overline{y}_i \|_{L^2(\Omega)}$,
        legend style={font=\tiny}, legend pos = outer north east]
        \addplot[mark = o,red] table [col sep=
        &, y=errorex_e, x=rho]{tables/tab_u1L2ande_rho.dat};
        \addlegendentry{$\|y_{1,\varrho_{H^{-1}}}-\overline y_1 \|_{L^2(\Omega)}$}
        \addplot[mark = o,blue] table [col sep=
        &, y=errorex_L2, x=rho]{tables/tab_u1L2ande_rho.dat};
        \addlegendentry{$\|y_{1,\varrho_{L^2}}-\overline{y}_1 \|_{L^2(\Omega)}$}
        \addplot[mark = square,red] table [col sep=
        &, y=errorex_e, x=rho]{tables/tab_u2L2ande_rho.dat};
        \addlegendentry{$\|y_{2,\varrho_{H^{-1}}}-\overline{y}_2 \|_{L^2(\Omega)}$}
        \addplot[mark = square,blue] table [col sep=
        &, y=errorex_L2, x=rho]{tables/tab_u2L2ande_rho.dat};
        \addlegendentry{$\|y_{2,\varrho_{L^2}}-\overline{y}_2 \|_{L^2(\Omega)}$}
        \addplot[mark = square*,red] table [col sep=
        &, y=errorex_e, x=rho]{tables/tab_u3L2ande_rho.dat};
        \addlegendentry{$\|y_{3,\varrho_{H^{-1}}}-\overline{y}_3 \|_{L^2(\Omega)}$}
        \addplot[mark = square*,blue] table [col sep=
        &, y=errorex_L2, x=rho]{tables/tab_u3L2ande_rho.dat};
        \addlegendentry{$\|y_{3,\varrho_{L^2}}-\overline{y}_3 \|_{L^2(\Omega)}$}
        \addplot[
        domain = 2^(-18):2^(-15),
        samples = 4,
        dashed,
        thin,
        red,
        mark = o,] {5*x};
        \addlegendentry{$\varrho$}
        \addplot[
        domain = 2^(-18):2^(-15),
        samples = 4,
        dashed,
        thin,
        red,
        mark = square,] {1.3*x^(0.75)};
        \addlegendentry{$\varrho^{1.5/2}$}
        \addplot[
        domain = 2^(-18):2^(-15),
        samples = 4,
        dashed,
        thin,
        red,
        mark = diamond*,] {1*x^(0.25)};
        \addlegendentry{$\varrho^{0.5/2}$}
        \addplot[
        domain = 2^(-18):2^(-15),
        samples = 4,
        dashed,
        thin,
        blue,
        mark = o,] {3.5*x^(2.5/4)};
        \addlegendentry{$\varrho^{2.5/4}$}
        \addplot[
        domain = 2^(-18):2^(-15),
        samples = 4,
        dashed,
        thin,
        blue,
        mark = square,] {0.8*x^(0.75/2)};
        \addlegendentry{$\varrho^{1.5/4}$}
        \addplot[
        domain = 2^(-18):2^(-15),
        samples = 4,
        dashed,
        thin,
        blue,
        mark = diamond*,] {1*x^(0.25/2)};
        \addlegendentry{$\varrho^{0.5/4}$}
      \end{axis}
    \end{tikzpicture}

    \caption{Errors $\|y_{i,\varrho}-\overline{y}_i\|_{L^2(\Omega)}$}
  \end{subfigure}
  \caption{Targets $\overline{y}_i$, state solutions $y_{i,\varrho_{L^2}}$
    and $y_{i,\varrho_{H^{-1}}}$, and errors
    $\| y_{i,\varrho} - \overline{y}_i \|_{L^2(\Omega)}$
    for different choices of regularization
    parameters and regularization norms.}
  \label{Fig:Introduction:states-1D}
\end{figure}

When the state $y_{i,\varrho}$ is known we can compute the control
$u_{i,\varrho}(x) = - y_{i,\varrho}''(x)$ as well as the costs
$\| u_{i,\varrho} \|_{L^2(\Omega)}$ 
and $\| u_{i,\varrho} \|_{H^{-1}(\Omega)} = \| y_{i,\varrho}' \|_{L^2(\Omega)}$
for both regularization norms. We observe that in the case of the smooth
target $\overline{y}_1$ all costs remain bounded for $\varrho \to 0$,
while for the discontinuous target $\overline{y}_3$ all costs tend
to infinity as $\varrho \to 0$. The situation is different for the
piecewise linear target $\overline{y}_2$ where the costs remain bounded
in the case of energy regularization, but tend to infinity as $\varrho \to 0$
when using $L^2$ regularization. As discussed later in this paper, these
observations are in complete agreement with our theoretical results.
Moreover, they show that our theoretical results are sharp.

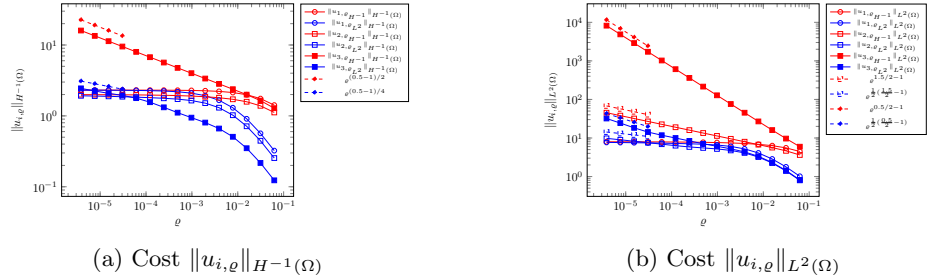
\begin{figure}[htbp!]
  \centering
  \begin{subfigure}[t]{0.45\textwidth}
    \centering
    \begin{tikzpicture}[scale=0.45]
      \begin{axis}[
        xmode = log,
        ymode = log,
        xlabel=$\varrho$,
        ylabel=$\| u_{i,\varrho}\|_{H^{-1}(\Omega)}$,
        legend style={font=\tiny}, legend pos = outer north east]
        \addplot[mark = o,red] table [col sep=
        &, y=costH1ex_e, x=rho]{tables/tab_u1L2ande_rho.dat};
        \addlegendentry{$\|u_{1,\varrho_{H^{-1}}}\|_{H^{-1}(\Omega)}$}
        \addplot[mark = o,blue] table [col sep=
        &, y=costH1ex_L2, x=rho]{tables/tab_u1L2ande_rho.dat};
        \addlegendentry{$\|u_{1,\varrho_{L^{2}}}\|_{H^{-1}(\Omega)}$}
        \addplot[mark = square,red] table [col sep=
        &, y=costH1ex_e, x=rho]{tables/tab_u2L2ande_rho.dat};
        \addlegendentry{$\|u_{2,\varrho_{H^{-1}}}\|_{H^{-1}(\Omega)}$}
        \addplot[mark = square,blue] table [col sep=
        &, y=costH1ex_L2, x=rho]{tables/tab_u2L2ande_rho.dat};
        \addlegendentry{$\|u_{2,\varrho_{L^{2}}}\|_{H^{-1}(\Omega)}$}
        \addplot[mark = square*,red] table [col sep=
        &, y=costH1ex_e, x=rho]{tables/tab_u3L2ande_rho.dat};
        \addlegendentry{$\|u_{3,\varrho_{H^{-1}}}\|_{H^{-1}(\Omega)}$}
        \addplot[mark = square*,blue] table [col sep=
        &, y=costH1ex_L2, x=rho]{tables/tab_u3L2ande_rho.dat};
        \addlegendentry{$\|u_{3,\varrho_{L^{2}}}\|_{H^{-1}(\Omega)}$}
        \addplot[
        domain = 2^(-18):2^(-15),
        samples = 4,
        dashed,
        thin,
        red,
        mark = diamond*,] {x^(-0.25)};
        \addlegendentry{$\varrho^{(0.5-1)/2}$}
        \addplot[
        domain = 2^(-18):2^(-15),
        samples = 4,
        dashed,
        thin,
        blue,
        mark = diamond*,] {0.65*x^(-0.125)};
        \addlegendentry{$\varrho^{(0.5-1)/4}$}
      \end{axis}
    \end{tikzpicture}
    \caption{Cost $\|u_{i,\varrho}\|_{H^{-1}(\Omega)}$}
  \end{subfigure}
  \hfill
  \begin{subfigure}[t]{0.45\textwidth}
    \centering
    \begin{tikzpicture}[scale=0.45]
      \begin{axis}[
        xmode = log,
        ymode = log,
        xlabel=$\varrho$,
        ylabel=$\| u_{i,\varrho}\|_{L^2(\Omega)}$,
        legend style={font=\tiny}, legend pos = outer north east]
        \addplot[mark = o,red] table [col sep=
        &, y=costL2ex_e, x=rho]{tables/tab_u1L2ande_rho.dat};
        \addlegendentry{$\|u_{1,\varrho_{H^{-1}}}\|_{L^2(\Omega)}$}
        \addplot[mark = o,blue] table [col sep=
        &, y=costL2ex_L2, x=rho]{tables/tab_u1L2ande_rho.dat};
        \addlegendentry{$\|u_{1,\varrho_{L^{2}}}\|_{L^2(\Omega)}$}
        \addplot[mark = square,red] table [col sep=
        &, y=costL2ex_e, x=rho]{tables/tab_u2L2ande_rho.dat};
        \addlegendentry{$\|u_{2,\varrho_{H^{-1}}}\|_{L^2(\Omega)}$}
        \addplot[mark = square,blue] table [col sep=
        &, y=costL2ex_L2, x=rho]{tables/tab_u2L2ande_rho.dat};
        \addlegendentry{$\|u_{2,\varrho_{L^{2}}}\|_{L^2(\Omega)}$}
        \addplot[mark = square*,red] table [col sep=
        &, y=costL2ex_e, x=rho]{tables/tab_u3L2ande_rho.dat};
        \addlegendentry{$\|u_{3,\varrho_{H^{-1}}}\|_{L^2(\Omega)}$}
        \addplot[mark = square*,blue] table [col sep=
        &, y=costL2ex_L2, x=rho]{tables/tab_u3L2ande_rho.dat};
        \addlegendentry{$\|u_{3,\varrho_{L^{2}}}\|_{L^2(\Omega)}$}	
        \addplot[
        domain = 2^(-18):2^(-15),
        samples = 4,
        dashed,
        thin,
        red,
        mark = square,] {3*x^(-0.25)};
        \addlegendentry{$\varrho^{1.5/2-1}$}
        \addplot[
        domain = 2^(-18):2^(-15),
        samples = 4,
        dashed,
        thin,
        blue,
        mark = square,] {3*x^(-0.125)};
        \addlegendentry{$\varrho^{\tfrac{1}{2}(\tfrac{1.5}{2}-1)}$}
        \addplot[
        domain = 2^(-18):2^(-15),
        samples = 4,
        dashed,
        thin,
        red,
        mark = diamond*,] {x^(-0.75)};
        \addlegendentry{$\varrho^{0.5/2-1}$}
        \addplot[
        domain = 2^(-18):2^(-15),
        samples = 4,
        dashed,
        thin,
        blue,
        mark = diamond*,] {0.4*x^(-0.375)};
        \addlegendentry{$\varrho^{\tfrac{1}{2}(\tfrac{0.5}{2}-1)}$}
      \end{axis}
    \end{tikzpicture}
    \caption{Cost $\|u_{i,\varrho}\|_{L^2(\Omega)}$}
  \end{subfigure}
  \caption{Error and different types of 
          regularization with $\varrho = \varrho_{H^{-1}}=\varrho_{L^2}$. }
        \label{Fig:Introduction:error-and-cost}
\end{figure}

The finite element discretization of OPCs such as
\eqref{Eqn:Introduction:AbstractOCP1}--\eqref{Eqn:Introduction:StateEquation1}
in general, or 
\eqref{Eqn:Introduction:OCPwithL^2Regularization}--\eqref{Eqn:Introduction:DBVP4Poisson*}
and \eqref{Eqn:Introduction:EnergyCostFunctional}--\eqref{Eqn:Introduction:DBVP4Poisson*} in particular,
is usually investigated for fixed $\varrho > 0$, and discretization error
estimates are provided for the finite element errors
$\| y_\varrho - y_{\varrho h}\|$ and $\|u_\varrho - u_{\varrho h}\|$ in
different norms; see, e.g.,
\cite{  LangerSteinbachTroeltzschYang:2021SINUM,
  LangerSteinbachTroeltzschYang:2021SISC,
  MeyerRoesch2004SICON,
  Roesch2004ZAA,
  RoeschSimon2007NFAO}
However, as shown above, the distance of the computed finite element state
$y_{\varrho h}$ from the desired state $\overline{y}$ is basically given by the 
distance $\|y_\varrho - \overline{y}\|$ that is fixed for fixed $\varrho$.
Similarly, the computed costs $\|u_{\varrho h}\|$ approximate the 
cost $\|u_\varrho\|$ that is also fixed for fixed $\varrho$. If we want to
improve the distance $\|y_{\varrho h} - \overline{y}\|$, then we have to
diminish $\varrho$ that leads to higher cost $\|u_{\varrho h}\|$. On the other
side, if we want to reduce the cost, then we have to enlarge $\varrho$ that
yields a larger distance $\|y_{\varrho h} - \overline{y}\|$. A careful
analysis of the discretization error $\|y_{\varrho h} - \overline{y}\|$ 
in terms of $\varrho$ and $h$ shows that we have to relate the
regularization parameter $\varrho$ to the mesh-size $h$ in order to get an
asymptotically optimal convergence of $y_{\varrho h}$ to $\overline{y}$ in
balance with the energy cost for the control. It is shown in   
\cite{LangerLoescherSteinbachYang:2023CMAM,LangerLoescherSteinbachYang:2024NLA}
that we should choose $\varrho = h^2$ and $\varrho = h^4$
for OCPs
\eqref{Eqn:Introduction:OCPwithL^2Regularization}--\eqref{Eqn:Introduction:DBVP4Poisson*}
with $L^2$ regularization and 
\eqref{Eqn:Introduction:EnergyCostFunctional}--\eqref{Eqn:Introduction:DBVP4Poisson*} 
with $H^{-1}$ regularization, respectively.
The same choices as in elliptic OPCs hold for parabolic OPCs
\cite{LangerSteinbachYang:2024ACOM} and hyperbolic OPCs
\cite{LoescherSteinbach:2022SINUM,LangerLoescherSteinbachYang:2025CAMWA} 
when using space-time finite element discretizations. If we permit functions
as regularization, then $\varrho$ can be adapted to the 
local mesh-size that can heavily vary over the computational domain 
in the case of an adaptive mesh refinement; see
\cite{LangerLoescherSteinbachYang:2024CAMWA}. In this connection, we also
refer to the very recent paper \cite{PowerPryer2025arXiv2503.11386} 
where  both the regularization parameter and the mesh-size are dynamically
adjusted locally on the basis of a posteriori error estimates.
The incorporation of box constraints imposed on the state or the control
finally leads to state-based variational inequalities of first kind which
can be again discretised by (space-time) finite elements; see
\cite{GanglLoescherSteinbach2025CAMWA} for the elliptic OPC
\eqref{Eqn:Introduction:EnergyCostFunctional}--\eqref{Eqn:Introduction:DBVP4Poisson*} with state or control constraints. As in the unconstrained case, 
the choice $\varrho = h^2$ again gives asymptotically optimal convergence.
For constant $\varrho > 0$, we refer to the recent papers
\cite{BrennerGedickeSung:2023CMAM,GongTan2025JSC,GudiShakya2025JCAM}
and the references therein.

An approximate  solution to OCPs such as
\eqref{Eqn:Introduction:AbstractOCP1}--\eqref{Eqn:Introduction:StateEquation1}
can be found via the finite element discretization of the optimality system
consisting of the state equation, the adjoint equation, and the gradient
equation for defining the optimal state $y_\varrho$, the optimal adjoint
(co-state) $p_\varrho$, and the optimal control $u_\varrho$; 
cf. \eqref{Eqn:Introduction:DBVP4Poisson*}--\eqref{Eqn:Introduction:L2:AdjointProblem} and \eqref{Eqn:Introduction:L2:GradientEquation}
for the simple elliptic OCP
\eqref{Eqn:Introduction:EnergyCostFunctional}--\eqref{Eqn:Introduction:DBVP4Poisson*} 
with $L^2$ regularization. The finite element discretization of the OCP 
\eqref{Eqn:Introduction:AbstractOCP1}--\eqref{Eqn:Introduction:StateEquation1}
finally leads to a linear symmetric, but indefinite linear system of
algebraic equations (saddle point system) for defining finite element vectors 
$\mathbf{y}_{\varrho h}, \mathbf{p}_{\varrho h}, \mathbf{u}_{\varrho h}$ 
related to the finite element approximations 
$y_{\varrho h}, p_{\varrho h}, u_{\varrho h}$ to $y_\varrho, p_\varrho, u_\varrho$
via the finite element isomorphism. This $3 \times 3$ block system can
be reduced to the $2 \times 2$ block system
        \begin{equation}
        \label{Eqn:Introduction:sid-2x2-block-system}
        \begin{bmatrix}
            \varrho^{-1} \mathbf{A}_h & \mathbf{B}_h\\
            \mathbf{B}^\top_h             & - \mathbf{M}_h
        \end{bmatrix}
        \begin{bmatrix}
            \mathbf{p}_{\varrho h}\\
            \mathbf{y}_{\varrho h}
        \end{bmatrix}
        =
        \begin{bmatrix}
            \mathbf{0}_h\\
            -\mathbf{\overline{y}}_h
        \end{bmatrix}
        \end{equation}
by eliminating the control $\mathbf{u}_{\varrho h}$. Here,
the matrix $\mathbf{B}_h$ arises from the finite element discretization of $B$,
$\mathbf{M}_h$ denotes the mass matrix, and $\mathbf{A}_h$ represents 
the regularisation term. Since we have to vary $\varrho$ in order to adapt
the accuracy of the approximation of the target $\overline{y}$ and the
energy cost for the control $u_\varrho$ to the practical requirements and
to our budget, we would like to have not only a solver for the symmetric
and indefinite system  \eqref{Eqn:Introduction:sid-2x2-block-system} 
or for the original $3 \times 3$ block system that runs in asymptotically
optimal complexity in terms of $h$ but also a solver that is robust in
$\varrho$. Such kind of robust iterative solvers were proposed and analysed in,
e.g.,
\cite{SchoeberlSimonZulehner2011SINUM,SchulzWittum2008CVS,Zulehner:2011SIMAX}.
Eliminating the adjoint $\mathbf{p}_{\varrho h}$ from
\eqref{Eqn:Introduction:sid-2x2-block-system}, we arrive at the symmetric
and positive definite (spd), state-based Schur-complement system 
\begin{equation}\label{Eqn:Introduction:Sy=rhs}
 \mathbf{S}_{\varrho h} \mathbf{y}_{\varrho h} = \mathbf{\overline{y}}_h,
\end{equation}
where $\mathbf{S}_{\varrho h} = \mathbf{M}_h + \varrho \mathbf{D}_h$ with
$\mathbf{D}_h = \mathbf{B}_h^T\mathbf{A}_h^{-1}\mathbf{B}_h$.
The spd Schur-complement system can efficiently be solved by means of
the preconditioned conjugate gradient (pcg) method 
provided that a robust preconditioner $\mathbf{C}_h$ is available, and 
the matrix-by-vector multiplication 
$\mathbf{D}_h * \mathbf{y}_{\varrho h}^k$,  which contains the application of
$\mathbf{A}_h^{-1}$, can be performed efficiently.
Surprisingly, for the optimal choice of the regularisation $\varrho$, the 
Schur-complement $\mathbf{S}_{\varrho h}$ is always spectrally equivalent 
to the mass matrix $\mathbf{M}_h$ and, therefore, to some diagonal 
approximation of $\mathbf{M}_h$ such as the lumped mass matrix
$\mbox{lump}(\mathbf{M}_h)$. So, the lumped mass matrix
$\mbox{lump}(\mathbf{M}_h)$ can serve as robust preconditioner
$\mathbf{C}_h$. This result is not only true for the elliptic case
\cite{LangerLoescherSteinbachYang:2023CMAM,LangerLoescherSteinbachYang:2024NLA}
but also for parabolic \cite{LangerSteinbachYang:2024ACOM} 
and hyperbolic \cite{LoescherSteinbach:2022SINUM} OCPs as well as for 
variable regularizations $\varrho$ locally adapted to the mesh-size.
It turns out that, for the elliptic OCP
\eqref{Eqn:Introduction:EnergyCostFunctional}--\eqref{Eqn:Introduction:DBVP4Poisson*} with $H^{-1}$ regularization,
$\mathbf{B}_h = \mathbf{A}_h = \mathbf{D}_h = \mathbf{K}_h$, and, therefore, 
$\mathbf{S}_{\varrho h} = \mathbf{M}_h + \varrho \mathbf{K}_h$,
where $\mathbf{K}_h$ is the spd finite element Laplacian stiffness matrix.
So, a fast multiplication is ensured. Similarly, for the elliptic OCP
\eqref{Eqn:Introduction:OCPwithL^2Regularization}--\eqref{Eqn:Introduction:DBVP4Poisson*}
with $L^2$ regularization, we have
$\mathbf{S}_{\varrho h} = \mathbf{M}_h + \varrho \mathbf{K}_h
\mathbf{M}_h^{-1} \mathbf{K}_h$. Here we can replace
$\mathbf{M}_h$ by $\mbox{lump}(\mathbf{M}_h)$ without affecting the
asymptotic behavior of the discretization error 
\cite{LangerLoescherSteinbachYang:2024NLA}, and a fast multiplication is
again  ensured. In general, we have to use inner iterations for
approximating the application of $\mathbf{A}_h^{-1}$.
Now, the pcg with the preconditioner $\mathbf{C}_h = \mbox{lump}(\mathbf{M}_h)$ 
can be used as nested solver in a nested iteration process on a sequence of 
uniformly or adaptively refined meshes starting with some coarse mesh 
and stopping the nested iteration as soon as a prescribed accuracy 
for the approximation of the given desired state $\overline{y}$ is
achieved without exceeding a given budget for the energy cost of the control.
The reconstruction of the control from the computed state is an integral part 
of the nested iteration process. This allows us to solve OCPs such as 
\eqref{Eqn:Introduction:AbstractOCP1}--\eqref{Eqn:Introduction:StateEquation1}
always in optimal, or, at least, almost optimal complexity.
In the case of OCPs with state or control constraints, we have to solve 
variational inequalities of first kind in the state-based formulation.
After the finite element discretization, these variational inequalities 
are living in the finite element state space, and can be reformulated 
as non-differentiable non-linear systems of algebraic equations for determining 
the nodal solution vector $\mathbf{y}_{\varrho h}$ corresponding to 
the finite element state solution $y_{\varrho h}$.
This non-linear system can by solved by the semi-smooth Newton method
that is nothing but the primal-dual active set method
\cite{HintermuellerItoKunisch2003SJO}. Alternatively, we can use
multigrid methods for variational inequalities arising, e.g.,
from obstacle problems \cite{HackbuschMittelmann1983NumerMath};
see also the overview article  \cite{GraeserKornhuber2009JCM}.
For $\varrho$ robust solvers of control or state constraint OCPs,
we also refer to 
\cite{AxelssonNeytchevaStroem2018JNM,DravinsNeytcheva2023Calcolo,
StollWathen:2012NLA} and the references therein.

The rest of the paper is organized as follows.
Section~\ref{Section:AbstractOCP} presents a theoretical framework 
for the analysis and numerical analysis of OCPs of the form
\eqref{Eqn:Introduction:AbstractOCP1}--\eqref{Eqn:Introduction:StateEquation1}
including regularization error estimates
(Subsection~\ref{Subsection:AbstractOCP:Setting}),
Galerkin discretization and error estimates
(Subsection~\ref{Subsection:AbstractOCP:Disrectization}),
recovering of the control from the computed Galerkin approximation to the
state in a simple postprocessing procedure
(Subsection~\ref{Subsection:AbstractOCP:ControlRecovering}), solvers and
their use in nested iteration with accuracy and cost control
(Subsection~\ref{Subsection:AbstractOCP:Solver}), 
the handling of additional functional box-constraints for the state and
the control (Subsection~\ref{Subsection:AbstractOCP:Constraints}).
The application of this abstract framework to the distributed control of 
Poisson's equation \eqref{Eqn:Introduction:OCPwithL^2Regularization}--\eqref{Eqn:Introduction:DBVP4Poisson*}
is presented in Section~\ref{Section:Application}, 
where we also discuss our numerical results.
The application to the distributed control of Poisson's equation deliver
the blueprint  for other applications such as discussed in
Section~\ref{Section:OtherApplications}.
In Section~\ref{Section:ConclusionOutlook}, we draw some conclusions,
and give an outlook on further research directions in connection
with our approach.

\section{Abstract optimal control problems}
\label{Section:AbstractOCP}

\subsection{Abstract setting and regularization error estimates}
\label{Subsection:AbstractOCP:Setting}
Let $X \subset H_X \subset X^*$ and $Y \subset H_Y \subset Y^*$ be Gelfand
triples of Hilbert spaces, where $X^*$ and $Y^*$ are the duals of $X$ and
$Y$ with respect to $H_X$ and $H_Y$, respectively. We assume that
$H_X$ and $H_Y$ are Hilbert spaces with the inner products
$\langle \cdot , \cdot \rangle_{H_X}$ and $\langle \cdot, \cdot \rangle_{H_Y}$,
respectively. Moreover, the duality pairings
$\langle \cdot , \cdot \rangle_{X^*,X}$ and $\langle \cdot , \cdot
\rangle_{Y^*,Y}$ are defined as extension of the inner products in $H_X$, and
in $H_Y$, respectively.

Let $ B : Y \to X^*$ be a bounded, linear operator which is assumed to
satisfy an inf-sup condition, i.e., for all $y \in Y$, we have
\[
  \| B y \|_{X^*} \leq c_2^B \, \| y \|_Y, \quad
  \sup\limits_{0 \neq x \in X}
  \frac{\langle B y , x \rangle_{X^*,X}}{\| x \|_X} \geq c_1^B \, \| y \|_Y
  \, .
\]
In addition we assume that $B$ is surjective.
Hence,
$B : Y \to X^*$ defines an isomorphism. By
\[
  \langle By , x \rangle_{X^*,X} =: \langle y , B^* x \rangle_{Y,Y^*}
  \quad \mbox{for all} \; (y,x) \in Y \times X
\]
we define the adjoint operator $B^* : X \to Y^*$.
For optimal control problems in which we are interested, $B$ 
results from boundary value or initial-boundary value problems 
for PDEs or systems of PDEs.
$Y$ is the state space with the norm $\| \cdot \|_Y$, 
while $U = X^*$ denotes the control space with norm $\| \cdot \|_{X^*}$
which describes the cost of the control. In order to define an equivalent
norm in $X^*$ we consider a linear self-adjoint and elliptic
operator $A : X \to X^*$ satisfying
\[
  \| A x \|_{X^*} \leq c_2^A \, \| x \|_X, \quad
  \langle A x , x \rangle_{X^*,X} \geq c_1^A \, \| x \|_X^2
  \quad \mbox{for all} \; x \in X .
\]
With this we define
$\| x \|_A = \sqrt{\langle A x,x \rangle_{X^*,X}}$ and
$\| u \|_{A^{-1}} = \sqrt{\langle A^{-1} u , u \rangle_{X,X^*}}$
which are equivalent norms in $X$ and $X^*$, respectively.

We first consider an abstract tracking type optimal control problem with
neither state nor control constraints to find the minimizer
$(y_\varrho, u_\varrho) \in Y \times U$ of the functional
\begin{equation}\label{Eqn:AbstractOCP:AbstractOCP}
  {\mathcal{J}}(y_\varrho,u_\varrho) = \frac{1}{2} \,
  \| y_\varrho - \overline{y} \|_{H_Y}^2 + \frac{1}{2} \, \varrho \,
  \| u_\varrho \|^2_{A^{-1}} \quad
  \mbox{s.t.} 
  \; B y_\varrho = u_\varrho \; \mbox{in} \; U=X^* ,
\end{equation}
where $\varrho > 0$ is the cost or regularization parameter
on which the solution depends, and $\overline{y} \in H_Y$ denotes the 
given target or desired state.

In the standard approach we consider the solution of the constraint equation
$B y_\varrho = u_\varrho$ which defines the control-to-state map
$y_\varrho = B^{-1}u_\varrho$. With this we can write the reduced
cost functional as
\[
  \widehat{J}(u_\varrho) =
  \frac{1}{2} \, \| B^{-1} u_\varrho - \overline{y} \|^2_{H_Y}
  + \frac{1}{2} \, \varrho \, \| u_\varrho \|^2_{A^{-1}},
\]
and its minimizer $u_\varrho \in U=X^*$ is given as the unique solution of the
gradient equation
\[
  B^{-1,*} (B^{-1} u_\varrho - \overline{y}) +
  \varrho \, A^{-1} u_\varrho \, = \, 0 \quad \mbox{in} \; X  .
\]
Since $B : Y \to X^*$ is an isomorphism, we can write the reduced cost
functional as
\begin{equation}\label{Eqn:AbstractOCP:Abstract reduced cost functional}
  \widetilde{\mathcal{J}}(y_\varrho) = \frac{1}{2} \,
  \| y_\varrho - \overline{y} \|_{H_Y}^2 + \frac{1}{2} \, \varrho \,
  \| B y_\varrho \|^2_{A^{-1}},
\end{equation}
and its minimizer $y_\varrho \in Y$ is given as the unique solution of
the gradient equation
\begin{equation}\label{Eqn:AbstractOCP:Abstract gradient equation}
  y_\varrho + \varrho \, B^* A^{-1} B y_\varrho = \overline{y} \quad
  \mbox{in} \; Y^* .
\end{equation}
The linear operator $S := B^* A^{-1} B : Y \to Y^*$ is self-adjoint
and elliptic, i.e.,
\[
  \langle S y , y \rangle_{Y^*,Y} \geq c_1^S \, \| y \|_Y^2 \quad
  \text{and} \quad
  \| S y \|_{Y^*} \leq c_2^S \, \| y \|_Y \quad \mbox{for all} \;
  y \in Y, 
\]
with $c_1^S = c_1^A (c_1^B / c_2^A)^2$ and $c_2^S = (c_2^B)^2 / c_1^A$;
see, e.g., \cite[Lemma 1]{LangerSteinbachYang:2024ACOM}.
We note that, for $By_\varrho=u_\varrho$, we have
\begin{equation}\label{Eqn:AbstractOCP:Norm equivalence}
  \| y_\varrho \|_S^2 = \langle S y_\varrho , y_\varrho \rangle_{Y^*,Y} =
  \langle A^{-1} B y_\varrho , B y_\varrho \rangle_{X,X^*} =
  \langle A^{-1} u_\varrho , u_\varrho \rangle_{X,X^*} =
  \| u_\varrho \|_{A^{-1}}^2.
\end{equation}
Moreover, for the solution $y_\varrho \in Y \subset H_Y$ of
\eqref{Eqn:AbstractOCP:Abstract gradient equation}, we get
\begin{equation}\label{Eqn:AbstractOCP:Abstract regularity Su}
S y_\varrho = \frac{1}{\varrho} \, (\overline{y}-y_\varrho) \in H_Y \, .
\end{equation}
Using \eqref{Eqn:AbstractOCP:Norm equivalence}, we can rewrite the reduced
cost functional \eqref{Eqn:AbstractOCP:Abstract reduced cost functional}
in the form
\[
  \widetilde{\mathcal{J}}(y_\varrho) = \frac{1}{2} \,
  \| y_\varrho - \overline{y} \|_{H_Y}^2 + \frac{1}{2} \, \varrho \,
  \| y_\varrho \|^2_S ,
\]
where the realization of $\| y_\varrho \|_S^2$ with $S = B^* A^{-1} B$
involves the inversion of $A$, which in general may complicate the
numerical implementation. Hence we may replace $\| y_\varrho \|_S^2$
by any equivalent but more easier computable norm
$\| y_\varrho \|_D^2 = \langle D y_\varrho , y_\varrho \rangle_{Y^*,Y}$
for some bounded and elliptic operator $D : X \to X^*$, and satisfying
the norm equivalence inequalities
\[
  c_1^D \, \| y \|_S \leq \| y \|_D \leq c_2^D \, \| y \|_S
  \quad \mbox{for all} \; y \in Y .
\]
We now minimize
\begin{equation}\label{Eqn:AbstractOCP:AbstractFunctionalD}
  \check{\mathcal{J}}(y_\varrho) = \frac{1}{2} \,
  \| y_\varrho - \overline{y} \|_{H_Y}^2 + \frac{1}{2} \, \varrho \,
  \| y_\varrho \|^2_D ,
\end{equation}
whose minimizer $y_\varrho \in Y$ is given as the unique solution satisfying
\begin{equation}\label{Eqn:AbstractOCP:Abstract VF}
  \langle y_\varrho , y \rangle_{H_Y} +
  \varrho \, \langle D y_\varrho , y  \rangle_{Y^*,Y} =
  \langle \overline{y} , y \rangle_{H_Y} \quad
  \mbox{for all} \; y \in Y .
\end{equation}
For the unique solution $y_\varrho \in Y$ of \eqref{Eqn:AbstractOCP:Abstract VF},
and depending on the regularity of the target $\overline{y}$,
we can state the following result for the regularization error
$\| y_\varrho - \overline{y} \|_{H_Y}$. 

\begin{lemma}\label{Lemma:AbstractOCP:RegularizationErrorUnconstrained}
  Let $y_\varrho \in Y$ be the unique solution of the variational formulation
  \eqref{Eqn:AbstractOCP:Abstract VF}. For $\overline{y} \in H_Y$ there holds
  \[
    \| y_\varrho \|_{H_Y} \leq \| \overline{y} \|_{H_Y}, \quad
    \| y_\varrho \|_D \leq \varrho^{-1/2} \, \| \overline{y} \|_{H_Y}, \quad
    \| y_\varrho - \overline{y} \|_{H_Y} \leq \| \overline{y} \|_{H_Y}, 
  \]
  while for $\overline{y} \in Y$ we have
  \[
    \| y_\varrho - \overline{y} \|_D \leq \| \overline{y} \|_D , \quad
    \| y_\varrho - \overline{y} \|_{H_Y} \leq \varrho^{1/2} \,
    \| \overline{y} \|_D , \quad
    \| y_\varrho \|_D \leq \| \overline{y} \|_D .
  \]
  If in addition $D \overline{y} \in H_Y$ is satisfied for
  $\overline{y} \in Y$,
  \[
    \| y_\varrho - \overline{y} \|_{H_Y} \leq
    \varrho \, \| D \overline{y} \|_{H_Y}, \quad
    \| y_\varrho - \overline{y} \|_D \leq \varrho^{1/2} \,
    \| D \overline{y} \|_{H_Y} 
  \]
  follow.
\end{lemma}
\begin{proof}
  For the particular test function $y = y_\varrho \in Y$ we first have
  \[
    \langle y_\varrho , y_\varrho \rangle_{H_Y} +
    \varrho \, \langle D y_\varrho , y_\varrho  \rangle_{Y^*,Y} =
    \langle \overline{y} , y_\varrho \rangle_{H_Y} \leq
    \| \overline{y} \|_{H_Y} \| y_\varrho \|_{H_Y},
  \]
  i.e.,
  \[
    \| y_\varrho \|_{H_Y} \leq \| \overline{y} \|_{H_Y} , \quad
    \| y_\varrho \|_D \leq \varrho^{-1/2} \, \| \overline{y} \|_{H_Y} .
  \]
  Moreover, we also obtain
  \[
    \varrho \, \| y_\varrho \|_D^2 =
    \varrho \, \langle D y_\varrho , y_\varrho \rangle_{Y^*,Y} =
    \langle \overline{y} - y_\varrho , y_\varrho \rangle_{H_Y} =
    \langle \overline{y} - y_\varrho , \overline{y} \rangle_{H_Y} -
    \langle \overline{y} - y_\varrho , \overline{y} - y_\varrho \rangle_{H_Y},
  \]
  which gives
  \[
    \| y_\varrho - \overline{y} \|_{H_Y}^2 + \varrho \, \| y_\varrho \|^2_D =
    \langle \overline{y} - y_\varrho , \overline{y} \rangle_{H_Y} \leq
    \| \overline{y} - y_\varrho \|_{H_Y} \| \overline{y} \|_{H_Y} ,
  \]
  i.e.,
  \[
    \| y_\varrho - \overline{y} \|_{H_Y} \leq \| \overline{y} \|_{H_Y}.
  \]
  When assuming $\overline{y} \in Y$, we can choose
  $y = \overline{y} - y_\varrho \in Y$ as test function to conclude
  \begin{eqnarray*}
    \| \overline{y} - y_\varrho \|^2_{H_Y}
    & = & \langle \overline{y} - y_\varrho , \overline{y} - y_\varrho
          \rangle_{H_Y} \, = \,
          \varrho \, \langle D y_\varrho ,
          \overline{y} - y_\varrho \rangle_{Y^*,Y} \\
    & = & \varrho \, \langle D \overline{y} , \overline{y} - y_\varrho
          \rangle_{Y^*,Y} - \varrho \, \langle D (\overline{y} - y_\varrho),
          \overline{y} - y_\varrho \rangle_{Y^*,Y},
  \end{eqnarray*}
  i.e.,
  \[
    \| y_\varrho - \overline{y} \|_{H_Y}^2 + \varrho \,
    \| y_\varrho - \overline{y} \|_D^2 =
    \varrho \, \langle D \overline{y} , \overline{y} - y_\varrho \rangle_{H_Y}
    \leq
    \varrho \, \| \overline{y} \|_S \| y_\varrho - \overline{y} \|_D .
  \]
  Hence,
  \[
    \| y_\varrho - \overline{y} \|_D \leq \| \overline{y} \|_D , \quad
    \| y_\varrho - \overline{y} \|_{H_Y} \leq \varrho^{1/2} \,
    \| \overline{y} \|_D .
  \]
  On the other hand, we can also write
  \[
    \| \overline{y} - y_\varrho \|^2_{H_Y}
    =
    \varrho \, \langle D y_\varrho ,
    \overline{y} - y_\varrho \rangle_{Y^*,Y} =
    \varrho \, \langle D y_\varrho , \overline{y} \rangle_{Y^*,Y} -
    \varrho \, \langle D y_\varrho , y_\varrho \rangle_{Y^*,Y},
  \]
  i.e.,
  \[
    \| \overline{y} - y_\varrho \|^2_{H_Y}
    +
    \varrho \, \| y_\varrho \|^2_D
    =
    \varrho \, \langle D y_\varrho , \overline{y} \rangle_{Y^*,Y}
    \leq \varrho \, \| y_\varrho \|_D \| \overline{y} \|_D ,
  \]
  and hence,
  \[
    \| y_\varrho \|_D \leq \| \overline{y} \|_D .
  \]
  Finally, if $D \overline{y} \in H_Y$,
  \[
    \| y_\varrho - \overline{y} \|_{H_Y}^2 + \varrho \,
    \| y_\varrho - \overline{y} \|_D^2 =
    \varrho \, \langle D \overline{y} , \overline{y} - y_\varrho \rangle_{Y^*,Y}
    \leq \varrho \, \| D \overline{y} \|_{H_Y}
    \| y_\varrho - \overline{y} \|_{H_Y},
  \]
  i.e.,
  \[
    \| y_\varrho - \overline{y} \|_{H_Y} \leq
    \varrho \, \| D \overline{y} \|_{H_Y}, \quad
    \| y_\varrho - \overline{y} \|_D \leq \varrho^{1/2} \,
    \| D \overline{y} \|_{H_Y} .
  \]
\end{proof}

\noindent
When using the results of
Lemma \ref{Lemma:AbstractOCP:RegularizationErrorUnconstrained}
we obtain a bound for the costs $\| u_\varrho \|_{X^*}$ for the
control $u_\varrho = B y_\varrho$, depending on the regularity of
the target $\overline{y}$.

\begin{cor}
  From the gradient equation
  $\varrho D y_\varrho + y_\varrho = \overline{y}$, see
  \eqref{Eqn:AbstractOCP:Abstract VF}, we obtain
  \[
    \| D y_\varrho \|_{H_Y} = \frac{1}{\varrho} \,
    \| \overline{y} - y_\varrho \|_{H_Y} \leq \left \{
      \begin{array}{ccl}
        \varrho^{-1} \, \| \overline{y} \|_{H_Y}
        & & \mbox{for} \, \overline{y} \in H_Y, \\[1mm]
        \varrho^{-1/2} \, \| \overline{y} \|_D
        & & \mbox{for} \; \overline{y} \in Y, \\[1mm]
        \| D \overline{y} \|_{H_Y}
        & & \mbox{for} \; \overline{y} \in Y, \, D \overline{y} \in H_Y ,
      \end{array} \right.
  \]
  as well as
  \[
    \| u_\varrho \|_{A^{-1}} = \| y_\varrho \|_S \leq 
    \frac{1}{c_1^D} \, \| y_\varrho \|_D
    \leq \frac{1}{c_1^D} \left \{
      \begin{array}{ccl}
        \varrho^{-1/2} \, \| \overline{y} \|_{H_Y}
        && \mbox{for} \; \overline{y} \in H_Y, \\[1mm]
        \| \overline{y} \|_D
        && \mbox{for} \; \overline{y} \in Y.
      \end{array}
      \right.
  \]
  In particular, for $\overline{y} \in Y$ the costs
  $\| u_\varrho \|_{A^{-1}}$ of the control $u_\varrho = B y_\varrho$ are
  uniformly bounded as $\varrho \to 0$. Moreover,
  $ \| y_\varrho - \overline{y} \|_{H_Y} \to 0 $
  as $\varrho \to 0$ implies $ u_\varrho \to \overline{u} = B \overline{y}$
  in $X^*$ in this case. However, in the more interesting case
  $\overline{y} \not \in Y$ we conclude $B \overline{y} \not\in U=X^*$,
  and $\| u_\varrho \|_{A^{-1}} \to \infty$ as $\varrho \to 0$. In this case
  we have to balance the regularization error
  $\| y_\varrho - \overline{y} \|_{H_Y}$ with the costs
  $\| u_\varrho \|_{A^{-1}}$ of the control $u_\varrho = B y_\varrho$ we are
  willing to pay.
\end{cor}

\noindent
In particular for less regular target functions, e.g.,
$\overline{y} \in Y$ but $D \overline{y} \not\in H_Y$, or even
$\overline{y} \not\in Y$, we may include the regularization
parameter $\varrho$ in the definition of the regularization
operator $D_\varrho : Y \to Y^*$. Instead of
\eqref{Eqn:AbstractOCP:Abstract VF} we then consider the
variational formulation to find $y_\varrho \in Y$ such that
\begin{equation}\label{Eqn:AbstractOCP:Abstract VF Drho}
  \langle y_\varrho , y \rangle_{H_Y} +
  \langle D_\varrho y_\varrho , y  \rangle_{Y^*,Y} =
  \langle \overline{y} , y \rangle_{H_Y} \quad
  \mbox{for all} \; y \in Y .
\end{equation}
As in Lemma \ref{Lemma:AbstractOCP:RegularizationErrorUnconstrained}
we then conclude the regularization error estimates
\begin{equation}\label{Eqn:AbstractOCP:Regularization Drho}
  \| y_\varrho - \overline{y} \|_{H_Y} \leq \| \overline{y} \|_{H_Y}
  \quad \mbox{for} \; \overline{y} \in H_Y, \quad
  \| y_\varrho - \overline{y} \|_{H_Y} \leq
  \| \overline{y} \|_{D_\varrho} \quad \mbox{for} \;
  \overline{y} \in Y.
\end{equation}
  
\subsection{Galerkin discretization and error estimates}
\label{Subsection:AbstractOCP:Disrectization}
For the discretization of the variational problem
\eqref{Eqn:AbstractOCP:Abstract VF} we introduce a conforming
finite-dimensional subspace
$Y_h = \mbox{span} \{ \varphi_i \}_{i=1}^M \subset Y$
spanned by the $M=M(h)$ basis functions $\varphi_1,\ldots,\varphi_M$.
Here, $h$ denotes some positive discretization parameter such that
$h$ tends to zero when $M=M(h)$ goes to infinity. For example, one can
think about $h$ being the mesh-size in a finite element 
discretization as considered in Section~\ref{Section:Application}. 

At this time we assume that for any $y \in Y$ there exists a
projection $P_h y \in Y_h$ satisfying
\begin{equation}\label{Eqn:AbstractOCP:Approximation Assumption}
  \| y - P_h y \|_{H_Y} \leq
  c_1 \, h^\alpha \, \| y \|_D, \quad
  \| y - P_h y \|_D \, \leq \,
  c_2 \, \| y \|_D \, ,
\end{equation}
for some $\alpha > 0$, and with positive constants $c_1, c_2$.
Moreover, if $D y \in H_Y$ is satisfied for $y \in Y$, we also assume,
for some positive constants $c_3, c_4$,
\begin{equation}\label{Eqn:AbstractOCP:Approximation Assumption 2}
  \| y - P_h y \|_{H_Y} \leq
  c_3 \, h^{2\alpha} \, \| D y \|_{H_Y}, \quad
  \| y - P_h y \|_D \, \leq \, 
  c_4 \, h^\alpha \, \| D y \|_{H_Y} \, .
\end{equation}
The Galerkin discretization of the variational formulation
\eqref{Eqn:AbstractOCP:Abstract VF} reads as follows:
Find the Galerkin approximation 
$y_{\varrho h} \in Y_h$ to the state $y_{\varrho} \in Y$ such that 
\begin{equation}\label{Eqn:AbstractOCP:GalerkinVF}
  \langle y_{\varrho h} , y_h \rangle_{H_Y} +
  \varrho \, \langle D y_{\varrho h} , y_h  \rangle_{Y^*,Y} =
  \langle \overline{y} , y_h \rangle_{H_Y} \quad
  \mbox{for all} \; y_h \in Y_h .
\end{equation}
When using standard arguments, we conclude unique solvability of
\eqref{Eqn:AbstractOCP:GalerkinVF} as well as Cea's lemma,
\begin{equation}\label{Eqn:AbstractOCP:AbstractCea}
  \| y_\varrho - y_{\varrho h} \|^2_{H_Y} + \varrho \,
  \| y_\varrho - y_{\varrho h} \|^2_D \leq \inf\limits_{y_h \in Y_h} \Big[
  \| y_\varrho - y_h \|^2_{H_Y} + \varrho \, 
  \| y_\varrho - y_h \|^2_D \Big] ,
\end{equation}
and using \eqref{Eqn:AbstractOCP:Approximation Assumption} for
$y = y_\varrho \in Y$, this gives
\begin{equation}\label{Eqn:AbstractOCP:Error yrho}
  \| y_\varrho - y_{\varrho h} \|^2_{H_Y} + \varrho \,
  \| y_\varrho - y_{\varrho h} \|^2_D \leq 
  \Big[ c_1^2 \, h^{2\alpha} + c_2 \, \varrho \Big] \, \| y_\varrho \|^2_D
  \leq c \, h^{2\alpha} \, \| y_\varrho \|_D^2,
\end{equation}
when choosing
\begin{equation}\label{Eqn:AbstractOCP:Optimal Parameter}
\varrho = h^{2\alpha} \, .
\end{equation}
Otherwise, if the regularisation or cost parameter $\varrho$ is fixed,
we can not expect any further convergence for small discretization
parameters $h$ satisfying $h^{2\alpha} < \varrho$.

However, depending on the regularity of the target $\overline{y}$ we can
refine the error estimate \eqref{Eqn:AbstractOCP:AbstractCea} as follows:

\begin{lemma}
  Let $y_{\varrho h} \in Y_h$ be the unique solution of the Galerkin
  variational problem \eqref{Eqn:AbstractOCP:GalerkinVF}. Then there
  holds the error estimate, when choosing $\varrho = h^{2\alpha}$,
  \begin{equation}\label{Eqn:AbstractOCP:Error}
    \| y_\varrho - y_{\varrho h} \|_{H_Y}^2 + h^{2\alpha} \,
    \| y_\varrho - y_{\varrho h} \|_D^2
    \leq c \, \left \{
      \begin{array}{cl}
        \| \overline{y} \|^2_{H_Y}
        & \mbox{for} \; \overline{y} \in H_Y, \\[2mm]
        h^{2\alpha} \, \| \overline{y} \|_D^2
        & \mbox{for} \; \overline{y} \in Y, \\[2mm]
        h^{4\alpha} \, \| D \overline{y} \|_{H_Y}^2
        & \mbox{for} \; \overline{y} \in Y, \, D \overline{y} \in H_Y .
      \end{array} \right.
  \end{equation}
\end{lemma}

\begin{proof}
  For $\overline{y} \in H_Y$, the estimate follows from
  Cea's lemma \eqref{Eqn:AbstractOCP:AbstractCea}
  for the particular test function $v_h=0$, and using
  Lemma~\ref{Lemma:AbstractOCP:RegularizationErrorUnconstrained}
  for $\overline{y} \in H_Y$,
  \[
    \| y_\varrho - y_{\varrho h} \|_{H_Y}^2 + \varrho \,
    \| y_\varrho - y_{\varrho h} \|_D^2
    \leq
    \| y_\varrho \|^2_{H_Y} + \varrho \, \| y_\varrho \|_D^2 \leq
    2 \, \| \overline{y} \|_{H_Y}^2 .
  \]
  For $\overline{y} \in Y$, using \eqref{Eqn:AbstractOCP:AbstractCea},
  the triangle inequality, 
  Lemma~\ref{Lemma:AbstractOCP:RegularizationErrorUnconstrained},
  and \eqref{Eqn:AbstractOCP:Approximation Assumption}, we have
  \begin{eqnarray*}
    \| y_\varrho - y_{\varrho h} \|^2_{H_Y}
    + \varrho \, \| y_\varrho - y_{\varrho h} \|^2_D
    & \leq & \inf\limits_{y_h \in Y_h} \Big[
             \| y_\varrho - y_h \|^2_{H_Y}
             + \varrho \, \| y_\varrho - y_h \|^2_D \Big] \\
    && \hspace*{-5cm}
       \leq \, 2 \, \| y_\varrho - \overline{y} \|^2_{H_Y} +
       2 \, \varrho \, \| y_\varrho - \overline{y} \|_D^2 +
       2 \inf\limits_{y_h \in Y_h} \Big[
       \| \overline{y} - y_h \|^2_{H_Y}
       + \varrho \, \| \overline{y} - y_h \|^2_D \Big] \\
    && \hspace*{-5cm}
       \leq \, 
       4 \, \varrho \, \| \overline{y} \|_D^2 +
       2 \, \Big[
       \| \overline{y} - P_h \overline{y} \|^2_{H_Y}
       + \varrho \, \| \overline{y} - P_h \overline{y} \|^2_D \Big] \\
    && \hspace*{-5cm}
       \leq \, \Big[ 2 \, c_1^2 \, h^{2\alpha} 
       + \varrho \, (4+c_2^2) \Big] \, \| \overline{y} \|_D^2 \\
    && \hspace*{-5cm} = \, c \, h^{2\alpha} \, \| \overline{y} \|_D^2 
  \end{eqnarray*}
  when choosing $\varrho = h^{2\alpha}$.
  If in addition $S \overline{y} \in H_Y$ is satisfied, the proof of the
  third estimate follows the same lines.
\end{proof}

\noindent
Since the Galerkin approximation $y_{\varrho h}$ to the state $y_\varrho$ as 
solution of the minimization problem
\eqref{Eqn:AbstractOCP:AbstractFunctionalD} is an approximation of the
desired target $\overline{y}$, we are interested in estimates for the
related error $y_{\varrho h}-\overline{y}$ in the $H_Y$ norm. 

\begin{theorem}
  Let $y_{\varrho h} \in Y_h$ be the unique solution of 
  \eqref{Eqn:AbstractOCP:GalerkinVF}, and choose
  $\varrho = h^{2\alpha}$. Then there hold the error estimates
  \begin{equation}\label{Eqn:AbstractOCP:Error target}
    \| y_{\varrho h} - \overline{y} \|_{H_Y} \leq c \, \left \{
      \begin{array}{ccl}
        \| \overline{y} \|_{H_Y}
        & & \mbox{for} \; \overline{y} \in H_Y, \\[2mm]
        h^\alpha \, \| \overline{y} \|_D
        & & \mbox{for} \; \overline{y} \in Y, \\[2mm]
        h^{2\alpha} \, \| D \overline{y} \|_{H_Y}
        && \mbox{for} \; \overline{y} \in Y, \, D \overline{y} \in H_Y .
      \end{array} \right.
  \end{equation}
  When using space interpolation arguments, and
  assuming $\overline{y} \in [H_Y,Y]_{|s}$ for some $s \in [0,1]$,
  we finally conclude the error estimate
  \begin{equation}\label{Eqn:AbstractOCP:Error Interpolation}
    \| y_{\varrho h} - \overline{y} \|_{H_Y} \leq c \, h^{\alpha s} \,
    \| \overline{y} \|_{[H_Y,Y]_{|s}} \quad
    \mbox{for} \; \overline{y} \in [H_Y,Y]_{|s} , \; s \in [0,1].
  \end{equation}
\end{theorem}

\subsection{Control recovering}
\label{Subsection:AbstractOCP:ControlRecovering}
When an approximate optimal state $y_{\varrho h}$ is known we can
compute the associated optimal control
$\widetilde{u}_\varrho = B y_{\varrho h} \in U = X^*$ and an approximate
control $\widetilde{u}_{\varrho h} \in U_h$ via post processing, where
$U_h = \mbox{span} \{ \psi_k \}_{k=1}^N \subset U$ is a suitable ansatz space.
For this we consider the variational formulation to find
$\widetilde{u}_{\varrho h} \in U_h$ such that
\begin{equation}\label{Eqn:AbstractOCP:Control VF}
  \langle \widetilde{u}_{\varrho h} , x_h \rangle_{X^*,X}
  =
  \langle B y_{\varrho h} , x_h \rangle_{X^*,X}
  \quad \mbox{for all} \; x_h \in X_h ,
\end{equation}
where $X_h = \mbox{span} \{ \phi_k \}_{k=1}^N \subset X$ is a suitable
test space. As in \eqref{Eqn:AbstractOCP:Approximation Assumption}
we assume that there exists a projection operator
$\Pi_h : X \to X_h$ satisfying the error estimate
\begin{equation}\label{Eqn:AbstractOCP:Assumption Pi}
\| x - \Pi_h x \|_X \leq c \, h^\alpha \, \| B^* x \|_{H_Y} .
\end{equation}
In order to ensure unique solvability of the
Galerkin--Petrov variational formulation
\eqref{Eqn:AbstractOCP:Control VF} we need to assume the discrete
inf-sup stability condition
\begin{equation}\label{Eqn:AbstractOCP:Control inf-sup}
  c_S \, \| u_h \|_{X^*} \leq
  \sup\limits_{0 \neq x_h \in X_h} \frac{\langle u_h,x_h \rangle_{X^*,X}}
  {\| x_h \|_X} \quad \mbox{for all} \; u_h \in U_h .
\end{equation}
In addition to \eqref{Eqn:AbstractOCP:Control VF} we consider
the variational formulation to find $u_{\varrho h} \in U_h$ such that
\[
  \langle u_{\varrho h} , x_h \rangle_{X^*,X} =
  \langle u_\varrho , x_h \rangle_{X^*,X} =
  \langle B y_\varrho , x_h \rangle_{X^*,X}
  \quad \mbox{for all} \; x_h \in X_h ,
\]
and we conclude the perturbed Galerkin orthogonality
\[
  \langle u_{\varrho h} - \widetilde{u}_{\varrho h} , x_h \rangle_{X^*,X}
  =
  \langle B (y_\varrho - y_{\varrho h}), x_h \rangle_{X^*,X}
  \quad \mbox{for all} \; x_h \in X_h .
\]
Moreover, using standard arguments, we conclude the error estimate
\[
  \| u_\varrho - u_{\varrho h} \|_{X^*}
  \leq \frac{1}{c_S} \inf\limits_{u_h \in U_h}
  \| u_\varrho - u_h \|_{X^*} \leq 
  \frac{1}{c_S} \, \| u_\varrho \|_{X^*} =
  \frac{1}{c_S} \, \| B y_\varrho \|_{X^*} \leq 
  \frac{c_2^B}{c_S} \, \| y_\varrho \|_Y .
\]
Now, using the discrete inf-sup stability condition
\eqref{Eqn:AbstractOCP:Control inf-sup}
as well as Cea's lemma \eqref{Eqn:AbstractOCP:AbstractCea} for
$y_h = P_h y_\varrho$ we obtain, choosing $\varrho = h^{2\alpha}$,
\begin{eqnarray*}
  c_S \, \| u_{\varrho h} - \widetilde{u}_{\varrho h} \|_{X^*}
  & \leq & \sup\limits_{0 \neq x_h \in X_h}
           \frac{\langle u_{\varrho h} - \widetilde{u}_{\varrho h},x_h
           \rangle_{X^*,X}}{\| x_h \|_X} \\
  & = & \sup\limits_{0 \neq x_h \in X_h}
        \frac{\langle B(y_\varrho-y_{\varrho h}),x_h
        \rangle_{X^*,X}}{\| x_h \|_X} \\
  & \leq & \| B (y_\varrho - y_{\varrho h}) \|_{X^*} \, \leq \,
           c_2^B \, \| y_\varrho - y_{\varrho h} \|_Y \, \leq \,
           \frac{c_2^B}{c_1^D} \, \| y_\varrho - y_{\varrho h} \|_D \\
  & \leq & \frac{c_2^B}{c_1^D} \sqrt{
           \frac{1}{\varrho} \, \| y_\varrho - P_h y_\varrho \|_{H_Y}^2 +
           \| y_\varrho - P_h y_\varrho \|_D^2 } \\
  & \leq & \frac{c_2^B}{c_1^D} \sqrt{
           \frac{1}{\varrho} \, c_1^2 \, h^{2\alpha} \,
           \| y_\varrho \|_D^2 + c_2^2 \, 
           \| y_\varrho \|^2_D } \, = \, c \, \| y_\varrho \|_D \, .
\end{eqnarray*}
Therefore,
\[
  \| u_\varrho - \widetilde{u}_{\varrho h} \|_{X^*}
  \, \leq \, \| u_\varrho - u_{\varrho h} \|_{X^*} +
  \| u_{\varrho h} - \widetilde{u}_{\varrho h} \|_{X^*}
  \, \leq \, c \, \| y_\varrho \|_D \, .
\]
follows. 
\begin{lemma}
  Let the assumptions \eqref{Eqn:AbstractOCP:Assumption Pi} and
  \eqref{Eqn:AbstractOCP:Control inf-sup} hold true. Further,
  let $y_\varrho \in Y$ be the unique solution of the variational
  formulation \eqref{Eqn:AbstractOCP:Abstract VF}, and let
  $u_\varrho = B y_\varrho \in U = X^*$ be the associated control.
  For $y_{\varrho h} \in Y_h$ being the unique solution of the
  Galerkin variational formulation \eqref{Eqn:AbstractOCP:GalerkinVF}
  we compute $\widetilde{u}_{\varrho h} \in U_h$
  as unique solution of the Galerkin--Petrov variational formulation
  \eqref{Eqn:AbstractOCP:Control VF}. For this
  approximate control $\widetilde{u}_{\varrho h} \in U_h$
  we obtain the associated state
  $\widetilde{y}_\varrho = B^{-1} \widetilde{u}_{\varrho h}$. For
  $\overline{y} \in H_Y$ we then have the error estimate
  \begin{equation}\label{Eqn:AbstractOCP:Final Error 1}
    \| \widetilde{y}_\varrho - \overline{y} \|_{H_Y} \leq c \,
    \| \overline{y} \|_{H_Y},
  \end{equation}
  while for $\overline{y} \in Y$ we have
  \begin{equation}\label{Eqn:AbstractOCP:Final Error 2}
    \| \widetilde{y}_\varrho - \overline{y} \|_{H_Y} \leq c \, h^\alpha \,
    \| \overline{y} \|_Y
  \end{equation}
  when choosing $\varrho = h^{2\alpha}$, and where $\alpha$ is given
  by the approximation property
  \eqref{Eqn:AbstractOCP:Approximation Assumption}.
\end{lemma}
\begin{proof}
  For any $\psi \in H_Y \subset Y^*$ we first define $x_\psi \in X$ as unique
  solution of the operator equation $B^* x_\psi = \psi$. With this we obtain
\begin{eqnarray*}
  \| \widetilde{y}_\varrho - y_\varrho \|_{H_Y}
     & = & \sup\limits_{0 \neq \psi \in H_Y}
     \frac{\langle \widetilde{y}_\varrho - y_\varrho , \psi \rangle_{H_Y}}
     {\| \psi \|_{H_Y}} \, = \, \sup\limits_{0 \neq \psi \in H_Y}
     \frac{\langle \widetilde{y}_\varrho - y_\varrho , B^* x_\psi \rangle_{H_Y}}
     {\| \psi \|_{H_Y}} \\
  & = & \sup\limits_{0 \neq \psi \in H_Y}
        \frac{\langle B \widetilde{y}_\varrho - B y_\varrho , x_\psi
        \rangle_{X^*,X}}
         {\| \psi \|_{H_Y}} \, = \,
         \sup\limits_{0 \neq \psi \in H_Y}
        \frac{\langle \widetilde{u}_{\varrho h} - u_\varrho , x_\psi
        \rangle_{X^*,X}}
         {\| \psi \|_{H_Y}} \\
  & = & \sup\limits_{0 \neq \psi \in H_Y}
        \frac{\langle \widetilde{u}_{\varrho h} - u_\varrho ,
        x_\psi - \Pi_h x_\psi \rangle_{X^*,X}
        +
        \langle \widetilde{u}_{\varrho h} - u_\varrho ,
        \Pi_h x_\psi \rangle_{X^*,X}}
        {\| \psi \|_{X^*,X}} \\
  & = & \sup\limits_{0 \neq \psi \in H_Y}
        \frac{\langle \widetilde{u}_{\varrho h} - u_\varrho ,
        x_\psi - \Pi_h x_\psi \rangle_{X^*,X}
        +
        \langle B (y_{\varrho h} - y_\varrho) , \Pi_h x_\psi \rangle_{X^*,X}}
        {\| \psi \|_{H_Y}} \, .
\end{eqnarray*}
With
\begin{eqnarray*}
  \langle \widetilde{u}_{\varrho h} - u_\varrho ,
  x_\psi - \Pi_h x_\psi \rangle_{X^*,X}
  & \leq &
  \| \widetilde{u}_{\varrho h} - u_\varrho \|_{X^*}
           \| x_\psi - \Pi_h x_\psi \|_X \\
  & \leq & c \, h^\alpha \, \| y_\varrho \|_D \| B^* x_\psi \|_{H_Y} \\
  & = & c \, h^\alpha \, \| y_\varrho \|_D \| \psi \|_{H_Y} ,
\end{eqnarray*}
and
\begin{eqnarray*}
  && \langle B (y_{\varrho h} - y_\varrho) , \Pi_h x_\psi \rangle_{X^*,X} \\
  && \hspace*{1cm}  = \, \langle B (y_{\varrho h} - y_\varrho) ,
     \Pi_h x_\psi - x_\psi \rangle_{X^*,X}
     + \langle B (y_{\varrho h} - y_\varrho) , x_\psi \rangle_{X^*,X} \\
  && \hspace*{1cm} \leq \, \| B (y_{\varrho h} - y_\varrho) \|_{X^*}
           \| \Pi_h x_\psi - x_\psi \|_X +
           \langle y_{\varrho h} - y_\varrho , B^* x_\psi \rangle_{H_Y} \\
  && \hspace*{1cm} \leq \, c_2^B \, \| y_{\varrho h} - y_\varrho \|_Y
     c \, h^\alpha \| B^* x_\psi \|_{H_Y} +
     \| y_{\varrho h} - y_\varrho \|_{H_Y}
     \| \psi \|_{H_Y} \\
  && \hspace*{1cm} = \, c \, h^\alpha \, \| y_{\varrho h} - y_\varrho \|_Y
     \| \psi \|_{H_Y} +
     \| y_{\varrho h} - y_\varrho \|_{H_Y}
     \| \psi \|_{H_Y}
\end{eqnarray*}
we conclude
\[
  \| \widetilde{y}_\varrho - y_\varrho \|_{H_Y} \leq
  \| y_{\varrho h} - y_\varrho \|_{H_Y} +
  c \, h^\alpha \, \Big[
  \| y_\varrho \|_D + \| y_\varrho - y_{\varrho h} \|_D
  \Big] \, ,
\]
and by the triangle inequality we have
\[
  \| \widetilde{y}_\varrho - \overline{y} \|_{H_Y} \leq
  \| y_\varrho - \overline{y} \|_{H_Y} +
  \| \widetilde{y}_\varrho - y_\varrho \|_{H_Y} .
\]
For $\overline{y} \in H_Y$ we finally conclude
\[
  \| \widetilde{y}_\varrho - y_\varrho \|_{H_Y} \leq
  c \, \Big( 1 + h^\alpha \, \varrho^{-1/2} \Big) \, \| \overline{y} \|_{H_Y} ,
\]
while for $\overline{y} \in Y$ we have
\[
  \| \widetilde{y}_\varrho - y_\varrho \|_{H_Y} \leq
  c \, \Big( h^\alpha + h^{2\alpha} \, \varrho^{-1/2} \Big)
  \, \| \overline{y} \|_D .
\]
Now the assertion follows for $\varrho = h^{2\alpha}$.
\end{proof}

\noindent
Using \eqref{Eqn:AbstractOCP:Final Error 1} and
\eqref{Eqn:AbstractOCP:Final Error 2}, and as
in \eqref{Eqn:AbstractOCP:Error Interpolation} we can use
space interpolation arguments to derive the final result of this
subsection.

\begin{theorem}
  Let the discrete state
  $y_{\varrho h} \in Y_h$ be the unique finite element solution of the
  variational formulation \eqref{Eqn:AbstractOCP:GalerkinVF} with
  $\varrho = h^{2\alpha}$, where we assume 
  $\overline{y} \in [H_Y,Y]_{|s}$ for some $s \in [0,1]$.
  Let the discrete control
  $\widetilde{u}_{\varrho h} \in U_h$ be the unique solution
  of \eqref{Eqn:AbstractOCP:Control VF}.
  For the resulting state
  $\widetilde{y}_\varrho = B \widetilde{u}_{\varrho h} \in Y$ we then
  have the error estimate
  \begin{equation}\label{Eqn:AbstractOCP:Final Error 3}
  \| \widetilde{y}_\varrho - \overline{y} \|_{H_Y} \, \leq \,
  c \, h^{\alpha s} \, \| \overline{y} \|_{[H_Y,Y]_{|s}} .
\end{equation}
\end{theorem}

\noindent
The former estimate relates the resulting state to the target, showing
the accuracy of the method. In addition it is important to control the
costs. Therefore we need to have a computable
bound for $\norm{ \widetilde{u}_{\varrho h}}_{X^\ast}$.  

\begin{lemma}\label{Lem:AbstractOCP:bound for the cost via state}
  Let $\widetilde{u}_{\varrho h}\in U_h$ be the unique solution of
  \eqref{Eqn:AbstractOCP:Control VF}, and let
  \eqref{Eqn:AbstractOCP:Control inf-sup} hold true. Then
  \begin{equation*}
    \norm{\widetilde{u}_{\varrho h}}_{X^\ast} \leq c \, \norm{y_{\varrho h}}_{D}. 
  \end{equation*}  
  If, in addition, the discrete inf-sup stability
  \begin{equation}\label{Eqn:AbstractOCP:discrete inf sup Y}
    \widetilde{c}_S \, \norm{y_h}_D \leq
    \sup_{0\neq x_h\in X_h}
    \frac{\skpr{By_{h},x_h}_{X^\ast,X}}{\norm{x_h}_X},\quad \forall y_h\in Y_h, 
  \end{equation} 
  holds, then 
  \begin{equation*}
    \norm{y_{\varrho h}}_D \leq
    \widetilde{c}_S \, \norm{\widetilde{u}_{\varrho h}}_{X^\ast}. 
  \end{equation*} 
\end{lemma}

\begin{proof}
  We compute, using \eqref{Eqn:AbstractOCP:Control inf-sup} and
  \eqref{Eqn:AbstractOCP:Control VF} 
  \begin{align*}
    c_S \, \norm{\widetilde{u}_{\varrho h}}_{X^\ast}
    \leq \sup_{0\neq x_h\in X_h}
    \frac{\skpr{\widetilde{u}_{\varrho h},x_h}_{X^\ast,X}}{\norm{x_h}_{X}}
    = \sup_{0\neq x_h\in X_h}
    \frac{\skpr{By_{\varrho h},x_h}_{X^\ast,X}}{\norm{x_h}_{X}} \leq
    c \, \norm{y_{\varrho h}}_D.  
  \end{align*}
  The second estimate follows the same lines, using
  \eqref{Eqn:AbstractOCP:discrete inf sup Y} and
  \eqref{Eqn:AbstractOCP:Control VF}, i.e., 
  \begin{align*}
    \widetilde{c}_S\norm{y_{\varrho h}}_D \leq \sup_{0\neq x_h\in X_h}
    \frac{\skpr{By_{\varrho h},x_h}_{X^\ast,X}}{\norm{x_h}_X} =
    \sup_{0\neq x_h\in X_h}
    \frac{\skpr{\widetilde{u}_{\varrho h},x_h}_{X^\ast,X}}{\norm{x_h}_X}
    \leq \norm{\widetilde{u}_{\varrho h}}_{X^\ast}. 
  \end{align*} 
\end{proof}

\noindent
The last question to be answered is: How well does
$\norm{\widetilde{u}_{\varrho h}}_{X^\ast}$ approximate the actual cost
$\norm{u_\varrho}_{X^\ast}$? To deliver a satisfactory response, let us
introduce the projection $Q_h:H_{X}\to X_h$ defined as 
\begin{equation*}
  \skpr{Q_hu,x_h}_{H_X}= \skpr{u,x_h}_{H_X},\quad \forall x_h\in X_h 
\end{equation*} 
and let us make the following assumptions:
\begin{enumerate}
\item[{\em i.}] $\text{dim}(U_h) = \text{dim}(X_h)$ and
  $Q_h:U_h\to X_h$ is uniformly bounded from below, i.e., 
  \begin{equation*}
    \norm{Q_h u_h}_{H_X} \geq c_Q \norm{u_h}_{H_X},\quad \forall u_h\in U_h.
  \end{equation*}
  Then, $Q_h$ admits a bounded inverse $Q_h^{-1}:X_h\to U_h$ with
  $\norm{Q_h^{-1}}\leq c_Q^{-1}$.  
\item[{\em ii.}] The projection operator $\Pi_h:X\to X_h$ satisfies 
  \begin{equation*}
    \norm{x-\Pi_h x}_{H_X} \, \leq \,
    c \, h^{2\alpha} \, \norm{x}_D,\quad \forall x\in X
    \quad \mbox{for some} \; \alpha > 0.
  \end{equation*}
\item[{\em iii.}] There holds an inverse inequality 
  \begin{equation}\label{eq:AbstractOCP:inverse inequality}
    \norm{x_h}_D \, \leq \, c_I \, h^{-\alpha} \,
    \norm{x_h}_{H_X},\quad \forall x_h\in X_h \quad
    \mbox{for some} \; \alpha > 0 .
  \end{equation}
\item[{\em iv.}] If $Dy_\varrho\in H_Y$ then $u_\varrho = By_\varrho \in H_X$
  and there holds the estimate 
  \begin{equation}\label{Eqn:AbstractOCP:estimate_Dy-u}
    \norm{Dy_\varrho}_{H_Y} \, \leq \, c \, \norm{u_\varrho}_{H_X}. 
  \end{equation}
\end{enumerate}
Then we can prove the following estimate. 

\begin{theorem}
  For arbitrary but fixed $\varrho >0$ let $u_\varrho = By_\varrho\in H_X$
  and let $\widetilde{u}_{\varrho h}\in U_h$ be the unique solution of
  \eqref{Eqn:AbstractOCP:Control VF}. Then,
  \begin{equation*}
    \norm{u_\varrho  - \widetilde{u}_{\varrho h}}_{X^\ast} \, \leq \,
    c \, h^\alpha \, \norm{u_\varrho}_{H_X}
    \quad \mbox{for some} \; \alpha > 0.
  \end{equation*}
\end{theorem}

\begin{proof}
  By assumption {\em iv.}, $Dy_\varrho \in H_Y$, and using Cea's lemma
  \eqref{Eqn:AbstractOCP:AbstractCea} and
  \eqref{Eqn:AbstractOCP:estimate_Dy-u} we get 
  \begin{equation}\label{Eqn:AbstractOCP:auxiliary Cea}
    \norm{y_\varrho - y_{\varrho h}}_D \leq
    \left(c_4h^{2\alpha} + c_3 \frac{h^{4\alpha}}{\varrho}\right)^{1/2}
    \norm{Dy_\varrho}_{H_Y} \leq
    c \, h^\alpha \left(c_4+c_3\frac{h^{2\alpha}}{\varrho}\right)^{1/2}
    \norm{u_\varrho}_{H_X}. 
  \end{equation}
  Further, it holds that 
  \begin{align*}
    \norm{u_\varrho -\widetilde{u}_{\varrho h}}_{X^\ast}
    & = \sup_{0\neq x\in X}
      \frac{\skpr{u_\varrho-\widetilde{u}_{\varrho h},x}_{X^\ast,X}}{\norm{x}_X}\\
    &= \sup_{0\neq x\in X}\left(
      \frac{\skpr{u_\varrho-\widetilde{u}_{\varrho h},x-\Pi_h x}_{X^\ast,X}}
      {\norm{x}_X} +
      \frac{\skpr{u_\varrho-\widetilde{u}_{\varrho h},\Pi_h x}_{X^\ast,X}}
      {\norm{x}_X}\right).
  \end{align*}
  For the second term we can estimate,
  using \eqref{Eqn:AbstractOCP:auxiliary Cea},
  \begin{align*}
    \skpr{u_\varrho - \widetilde{u}_{\varrho h},\Pi_h x}_{X^\ast, X}
    & = \skpr{B(y_\varrho - y_{\varrho h}),\Pi_h x}_{X^\ast,X} \, \leq \,
      \norm{y_\varrho - y_{\varrho h}}_{D}\norm{\Pi_h x}_D\\
    & \leq c \, \norm{y_\varrho - y_{\varrho h}}_D\norm{x}_D
    \, \leq \, c \, h^\alpha \,
    \left(c_4+c_3\frac{h^{2\alpha}}{\varrho}\right)^{1/2}
    \norm{u_\varrho}_{H_X}\norm{x}_X . 
  \end{align*}
  For the first term we have 
  \[
    \skpr{u_\varrho -\widetilde{u}_{\varrho h},x-\Pi_hx}_{X^\ast,X}
    \leq \norm{u_\varrho - \widetilde{u}_{\varrho h}}_{H_X}\norm{x-\Pi_h x}_{H_X}
    \leq ch^\alpha\norm{u_\varrho -\widetilde{u}_{\varrho h}}_{H_X}\norm{x}_{X},
  \]
  and further 
  \[
    \norm{u_\varrho -\widetilde{u}_{\varrho h}}_{H_X}
    \leq
    c \, \Big(
    \norm{u_\varrho}_{H_X}+ \norm{\widetilde{u}_{\varrho h}}_{H_X}\Big). 
  \]
  So, it remains to bound
  \begin{equation*}
    \norm{\widetilde{u}_{\varrho h}}_{H_X} \leq c \, \norm{u_\varrho}_{H_X}.
  \end{equation*}
  Therefore, we first consider the case $U_h = X_h$ and estimate,
  using \eqref{Eqn:AbstractOCP:Control VF},
  \eqref{Eqn:AbstractOCP:auxiliary Cea} and the inverse inequality
  \eqref{eq:AbstractOCP:inverse inequality},
  \begin{align*}
    \norm{\widetilde{u}_{\varrho h}}_{H_X}^2
    & = \skpr{\widetilde{u}_{\varrho h},\widetilde{u}_{\varrho h}}_{H_X} \, = \,
      \skpr{By_{\varrho h},\widetilde{u}_{\varrho h}}_{H_X} \, = \,
      \skpr{B(y_{\varrho h} - y_\varrho),\widetilde{u}_{\varrho h}}_{H_X} +
      \skpr{u_\varrho,\widetilde{u}_{\varrho h}}_{H_X}\\
    & \leq \norm{y_\varrho - y_{\varrho h}}_{D}
      \norm{\widetilde{u}_{\varrho h}}_{D} +
      \norm{u_\varrho}_{H_X}\norm{\widetilde{u}_{\varrho h}}_{H_X}\\
    & \leq c \, h^\alpha \, \left(
      c_4+c_3\frac{h^{2\alpha}}{\varrho}\right)
      \norm{u_\varrho}_{H_X}c_Ih^{-\alpha}\norm{\widetilde{u}_{\varrho h}}_{H_X}
      + \norm{u_\varrho}_{H_X}\norm{\widetilde{u}_{\varrho h}}_{H_X} \\
    & \leq c \, \norm{u_\varrho}_{H_X}\norm{\widetilde{u}_{\varrho h}}_{H_X}.  
  \end{align*}
  Now, if $U_h\neq X_h$, we we define
  $\widehat{u}_{\varrho h} = Q_h\widetilde{u}_{\varrho h}\in X_h$, which satisfies
  \begin{equation*}
    \skpr{\widehat{u}_{\varrho h},x_h}_{H_X} =
    \skpr{Q_h\widetilde{u}_{\varrho h},x_h}_{H_X} =
    \skpr{\widetilde{u}_{\varrho h},x_h}_{X^\ast,X} =
    \skpr{By_{\varrho h},x_h}_{X^\ast,X},\, \forall x_h\in X_h.   
  \end{equation*}
  Recalling, that $Q_h:U_h\to X_h$ is boundedly invertible and
  $\widetilde{u}_{\varrho h} = Q_h^{-1}\widehat{u}_{\varrho h}$, we have
  \begin{align*}
    \norm{\widetilde{u}_{\varrho h}}_{H_X} =
    \norm{Q_h^{-1}\widehat{u}_{\varrho h}}_{H_X} \leq
    \norm{Q_h^{-1}}\norm{\widehat{u}_{\varrho h}}_{H_X} \leq
    c_Q^{-1} \, \norm{u_\varrho}_{H_X}.
  \end{align*}
  Thus, we can replace $\widetilde{u}_{\varrho h}$ by $\widehat{u}_{\varrho h}$
  in the above derivation, which finishes the proof. 
\end{proof}

%
%
%
%
\subsection{Solvers and their use in nested iteration}
\label{Subsection:AbstractOCP:Solver}
Once the basis is chosen, the Galerkin variational formulation
\eqref{Eqn:AbstractOCP:GalerkinVF} is equivalent to the following 
spd
linear system of algebraic equations: Find
${\bf y}_{\varrho h} = (y_1,\ldots,y_M)^\top \in \mathbb{R}^M$ solving the spd
system
\begin{equation}\label{Eqn:AbstractOCP:AlgebraicSystem4State}
  ({\mathbf M}_h + \varrho \, {\mathbf D}_h) {\mathbf y}_{\varrho h} =
  \overline{\mathbf y}_h,
\end{equation}
where
${\mathbf M}_h = (\langle \varphi_j , \varphi_i \rangle_{H_Y})_{i,j=1,\ldots,M}$
and 
${\mathbf D}_h = (\langle D \varphi_j , \varphi_i \rangle_{H_Y})_{i,j=1,\ldots,M}$ 
are $M \times M$ spd matrices, while the vector 
${\mathbf \overline{y}}_h =(\langle  \overline{y}, \varphi_i
\rangle_{H_Y})_{i=1,\ldots,M} \in \mathbb{R}^M $ is defined by the given target
$\overline{y} \in H_Y$. Thus, the solution
${\mathbf y}_{\varrho h} = (y_1,\ldots,y_M)^\top \in \mathbb{R}^M$ of
\eqref{Eqn:AbstractOCP:AlgebraicSystem4State} provides the coefficients for
the Galerkin solution 
$y_{\varrho h} = \sum_{j=1}^M y_j \varphi_j \in Y_h \subset Y$ of
\eqref{Eqn:AbstractOCP:GalerkinVF} via Galerkin's isomorphism
$y_{\varrho h} \leftrightarrow {\mathbf y}_{\varrho h}$.

Let us choose the optimal regularization parameter $\varrho = h^{2\alpha}$,
and let ${\mathbf C}_h$ be an asymptotically optimal spd preconditioner
for the spd system matrix ${\mathbf M}_h + \varrho \, {\mathbf D}_h$ of
\eqref{Eqn:AbstractOCP:AlgebraicSystem4State}, i.e. there are positive,
$h$ respectively $M_h$ independent spectral constants $c_1$ and $c_2$ such
that the  spectral equivalence inequalities
\begin{equation} \label{Eqn:AbstractOCP:SpectralEquivalenceInequalities}
  c_1 \, {\mathbf C}_h \le {\mathbf S}_{\varrho h} = {\mathbf M}_h +
  \varrho \, {\mathbf D}_h \le  c_2 \, {\mathbf C}_h
\end{equation}
hold, and the action ${\mathbf C}_h^{-1} {\mathbf r}_h$ is of asymptotically
optimal algebraic complexity $O(M_h)$, where the best spectral constants
$c_1$ and $c_2$ can be characterized by the minimal eigenvalue
$\lambda_{\min}({\mathbf C}_h^{-1}{\mathbf S}_{\varrho h})$ and the maximal
eigenvalue $\lambda_{\max}({\mathbf C}_h^{-1}{\mathbf S}_{\varrho h})$ 
of ${\mathbf C}_h^{-1}{\mathbf S}_{\varrho h}$, respectively. Then the
algebraic system \eqref{Eqn:AbstractOCP:AlgebraicSystem4State} 
can be solved by the pcg method 
to a given relative accuracy $\varepsilon \in (0,1)$ in the
${\mathbf S}_{\varrho h}$ energy norm with asymptotically optimal algebraic
complexity ${\mathcal{O}}(M_h)$ provided that the multiplication of the
system matrix ${\mathbf S}_{\varrho h}$ with a vector is of asymptotically
optimal complexity too. More precisely, after $n$ pcg iterations, we get
the iteration error estimate
\begin{equation} \label{Eqn:AbstractOCP:IterationErrorEstimate}
  \|{\mathbf y}_{\varrho h} - {\mathbf y}_{\varrho h}^n \|_{{\mathbf S}_{\varrho h}} 
  \le q_n \,
  \|{\mathbf y}_{\varrho h} - {\mathbf y}_{\varrho h}^0 \|_{{\mathbf S}_{\varrho h}}, 
\end{equation}
in the ${\mathbf S}_{\varrho h}$ energy norm 
$\|\cdot\|_{{\mathbf S}_{\varrho h}} = ({\mathbf S}_{\varrho h}\cdot,\cdot)^{1/2}$,
where ${\mathbf y}_{\varrho h}$, ${\mathbf y}_{\varrho h}^n$, and
${\mathbf y}_{\varrho h}^0$ denote the exact solution of
\eqref{Eqn:AbstractOCP:AlgebraicSystem4State}, the $n$th pcg iterate, and
the initial guess, respectively. The reduction factor $q_n$ after $n$ pcg
iterations is given by $q_n = 2 q^n/(1+q^{2n})$ with
$q = ((\text{cond}_2({\mathbf C}_h^{-1}{\mathbf S}_{\varrho h}))^{1/2}-1)/((\text{cond}_2({\mathbf C}_h^{-1}{\mathbf S}_{\varrho h}))^{1/2} + 1) < 1$
and $\text{cond}_2({\mathbf C}_h^{-1}{\mathbf S}_{\varrho h}) 
= \lambda_{\max}({\mathbf C}_h^{-1}{\mathbf S}_{\varrho h})/\lambda_{\min}({\mathbf C}_h^{-1}{\mathbf S}_{\varrho h})
\le c_2/c_1$. The proofs of these well-known results on pcg can be found 
in the standard literature; 
see, e.g., \cite[Chapter 10]{Hackbusch2016Book}, or \cite[Chapter 13]{Steinbach:2008Book}.
 
The Petrov--Galerkin scheme \eqref{Eqn:AbstractOCP:Control VF} allows us
to recover the control  $\widetilde{u}_{\varrho h} \in U_h$  from the computed
state $y_{\varrho h} \in Y_h$. Determining the solution 
$\widetilde{u}_{\varrho h} = \sum_{j=1}^N u_j \psi_j$ of the
Petrov--Galerkin scheme \eqref{Eqn:AbstractOCP:Control VF} is equivalent to
the solution of the following system of algebraic equations: 
Find the coefficient vector
$\widetilde{\mathbf u}_{\varrho h} = (u_1,\ldots,u_N)^\top \in \mathbb{R}^N$
such that
\begin{equation}\label{Eqn:AbstractOCP:AlgebraicSystem4Control}
  \overline{\mathbf M}_h \widetilde{\mathbf u}_{\varrho h} =
  {\mathbf B}_h{\mathbf y}_{\varrho h}
\end{equation}
where $\overline{\mathbf M}_h =
(\langle \psi_j , \phi_i \rangle_{H_X})_{i,j=1,\ldots,N}$ and 
${\mathbf B}_h = (\langle B \varphi_j , \phi_i \rangle_{H_X}
)_{i=1,\ldots,M;\,j=1,\ldots,N}$ are $N \times N$ and $M \times N$ matrices,
respectively. Due to the discrete inf-sup condition
\eqref{Eqn:AbstractOCP:Control inf-sup}, the $N \times N$ matrix
$\overline{\mathbf M}_h$ is always regular, but in general neither symmetric
nor positive definite. If we would choose
$U_h = X_h = \mbox{span} \{ \phi_k \}_{k=1}^N \subset X$, then
$\overline{\mathbf M}_h$ is spd as Gram matrix, but then the computed
control is in general too smooth. Nonetheless, in our application presented
in Section~\ref{Section:Application}, we can choose different spaces $U_h$
and $X_h$ such that the inf-sup condition
\eqref{Eqn:AbstractOCP:Control inf-sup} is satisfied
and $\overline{\mathbf M}_h$ is spd at the same time.
Then system \eqref{Eqn:AbstractOCP:AlgebraicSystem4Control} can efficiently be 
solved by pcg. 

In practice, these solvers should be used within a nested iteration procedure
on a sequence of refined finite dimensional spaces with growing dimensions
$M_1 <  \dots < M_\ell <  \dots < M_L$ respectively
$N_1 <  \dots < N_\ell <  \dots < N_L$  related to 
shrinking discretization parameters (mesh sizes) 
$h_1 >  \dots > h_\ell >  \dots > h_L$ such that 
$h_L$ goes to zero and  $M_L, N_L$ go to infinity 
as $L$ tends to infinity.
At some $h=h_\ell$, this nested iteration should produce
\begin{itemize}
\item a control $\widetilde{u}_{\varrho h}$ such that 
  $\|\widetilde{u}_{\varrho h}\|_U =
  \|\widetilde{u}_{\varrho h}\|_{A^{-1}} < c_\text{\tiny cost}$, and
 \item  the corresponding state 
   $\widetilde{y}_{\varrho} = B^{-1} \widetilde{u}_{\varrho h}$ 
   satisfying 
   $ \| \widetilde{y}_\varrho - \overline{y} \|_{H_Y} \, \leq
   \, \varepsilon \, \| \overline{y} \|_{[H_Y,Y]_{|s}}$
\end{itemize}
in asymptotically optimal arithmetical complexity ${\mathcal{O}}(M_h)$ 
with a given ``budget'' $c_\text{\tiny cost} > 0$ and given accuracy
$\varepsilon = 10^{-p} < 1$, where $\varrho=h^{2\alpha}$. We note that
the control $\widetilde{u}_{\varrho h}$ will be recovered from the computed
state ${y}_{\varrho h}$, and
$\widetilde{y}_{\varrho} = B^{-1} \widetilde{u}_{\varrho h}$ 
satisfies the estimate \eqref{Eqn:AbstractOCP:Final Error 3}.
So, for sufficiently small $h$, we have $c h^{\alpha s} \le \varepsilon$.
We further note that, in practice,  
${y}_{\varrho h} \leftrightarrow \mathbf{y}_{\varrho h}$ and 
$\widetilde{u}_{\varrho h} \leftrightarrow \widetilde{\mathbf u}_{\varrho h}$
will be computed by solving the algebraic systems
\eqref{Eqn:AbstractOCP:AlgebraicSystem4State} and
\eqref{Eqn:AbstractOCP:AlgebraicSystem4Control}, respectively. 

Algorithm~\ref{Algorithm1} summarizes the accuracy and cost controlled
nested iteration procedure described above. The subindex $\ell$ always
indicates the refinement level, i.e. ${\mathbf M}_\ell$ stands for
${\mathbf M}_{h_\ell}$, ${\mathbf K}_\ell$ for ${\mathbf K}_{\varrho,h_\ell}$
with $\varrho=h_\ell^{2\alpha}$ etc. For $\ell = 1$, systems
\eqref{Eqn:AbstractOCP:AlgebraicSystem4State} and
\eqref{Eqn:AbstractOCP:AlgebraicSystem4Control} can be solved directly 
as indicated in the comments at lines 7 and 8, but their iterative solution 
starting with zero initial guesses is also possible provided that appropriate
preconditioners ${\mathbf C}_1$ and $\overline{\mathbf C}_1$ are available.
Since good initial guesses are available for $\ell = 2,\cdots,L$, the
algebraic systems \eqref{Eqn:AbstractOCP:AlgebraicSystem4State} and
\eqref{Eqn:AbstractOCP:AlgebraicSystem4Control} should be solved by
preconditioned iterative methods with appropriate
preconditioners ${\mathbf C}_\ell$ and $\overline{\mathbf C}_\ell$.
In Subsection~\ref{Subsection:Application:EnergyRegularisation:NestedIteration},
we show that pcg can be used not only for solving
\eqref{Eqn:AbstractOCP:AlgebraicSystem4State} but also
\eqref{Eqn:AbstractOCP:AlgebraicSystem4Control} with simple diagonal 
preconditioners ${\mathbf C}_\ell$ and $\overline{\mathbf C}_\ell$ 
obtained from lumping the corresponding mass matrices 
${\mathbf M}_\ell$ and $\overline{\mathbf M}_\ell$. 

\begin{algorithm}[ht]
{
	
	\Indp
	\For{$\ell=1, \dots, L$}
	{
		Generate $\mathbf{M}_\ell$, $\mathbf{D}_\ell$, ${\overline{\mathbf y}}_\ell$ 
                                      \tcc*{\eqref{Eqn:AbstractOCP:AlgebraicSystem4State}}
        
        $\varrho_\ell \gets h_\ell^{2 \alpha}$ \tcc*{optimal regularization}
        
        $\mathbf{S}_\ell \gets \mathbf{M}_\ell + \varrho \mathbf{D}_\ell$ 
                                      \tcc*{\eqref{Eqn:AbstractOCP:AlgebraicSystem4State}}
		
		Generate $\overline{\mathbf M}_\ell$, $\mathbf{B}_\ell$
                                      \tcc*{\eqref{Eqn:AbstractOCP:AlgebraicSystem4Control}}
		
        \uIf{$\ell = 1$}
        {
		$\mathbf{y}_\ell \gets \mathbf{S}^{-1}_\ell {\overline{\mathbf y}}_\ell$
                                      \tcc*{solve \eqref{Eqn:AbstractOCP:AlgebraicSystem4State} directly}
                                      
         $\widetilde{\mathbf u}_\ell \gets \overline{\mathbf M}_\ell^{-1} {\mathbf B}_\ell \mathbf{y}_\ell$
                                      \tcc*{solve \eqref{Eqn:AbstractOCP:AlgebraicSystem4Control} directly}
         }
         \Else{
         $\mathbf{y}_\ell  \gets \mathbf{I}_{\ell-1}^\ell\mathbf{y}_{\ell-1}$ 
                                      \tcc*{prolongation of the state}
                                      \tcc*{as initial guess for the iteration}
                                      
         $\mathbf{y}_\ell \gets \mathbf{S}^{-1}_\ell {\overline{\mathbf y}}_\ell$
                                      \tcc*{solve \eqref{Eqn:AbstractOCP:AlgebraicSystem4State} iteratively} 
                                      
         $\widetilde{\mathbf u}_\ell  \gets \mathbf{I}_{\ell-1}^\ell \widetilde{\mathbf u}_{\ell-1}$ 
                                      \tcc*{prolongation of the control}
                                      \tcc*{as initial guess for the iteration}
                                      
         $\widetilde{\mathbf u}_\ell \gets \overline{\mathbf M}_\ell^{-1} {\mathbf B}_\ell \mathbf{y}_\ell$
                                      \tcc*{solve \eqref{Eqn:AbstractOCP:AlgebraicSystem4Control}  iteratively}                                                          
		 }
		 
		 $e_\ell \gets \|\mathbf{y}_\ell -{\overline{\mathbf y}}_\ell\|_{{\mathbf M}_\ell} =\|{y}_\ell -{\overline{y}}_\ell\|_{H_Y}$ \tcc*{discretization error}
		 
		 $c_\ell = \|\widetilde{u}_\ell\|_U^2$ \tcc*{energy cost of the control}
        
        \If{$c_\ell >  c_\text{\tiny cost}$}
         {
         {\bf STOP} and \Return{$c_\ell, \widetilde{\mathbf u}_\ell$, $e_\ell, \mathbf{y}_\ell$}
                    \tcc*{cost test}
         }
        
         \If{$e_\ell \le  \varepsilon \, \|{\overline{\mathbf y}}_\ell\|_{H_Y}$}
         {
         {\bf STOP} and \Return{$c_\ell, \widetilde{\mathbf u}_\ell$, $e_\ell, \mathbf{y}_\ell$}
                    \tcc*{accuracy test}
         }

    }
	
	\Return{$c_L, \mathbf{y}_L$, $e_L, \mathbf{y}_L$}
	
	\caption{Accuracy and cost controlled nested iteration.} 
	\label{Algorithm1}
}
\end{algorithm}

\begin{remark}
  The energy cost $\|\widetilde{u}_\ell\|_{U=X^*}^2 =
  \|\widetilde{u}_\ell\|_{A^{-1}}^2  =
  \langle A^{-1} \widetilde{u}_\ell, \widetilde{u}_\ell\rangle_{X,X^*}$ 
  in Line 16 of Algorithm~\ref{Algorithm1} is in general not computable
  exactly, but we can efficiently compute a good upper bound as follows:
  \begin{equation}\label{Eqn:AbstractOCP:ComputableControlCost}
    \|\widetilde{u}_\ell\|_U^2 =
    \langle w, \widetilde{u}_\ell\rangle_{X,X^*} =
    \langle w, \widetilde{u}_\ell\rangle_{H_X} \le
    \|w\|_{H_X}\|\widetilde{u}_\ell\|_{H_X} \le
    c_F^2 \|\widetilde{u}_\ell\|_{H_X}^2,
  \end{equation}
  where we used that $w = A^{-1} \widetilde{u}_\ell \in X$ solves the
  variational equation
  \begin{equation*}
    \langle A w, v \rangle_{X^*,X} =
    \langle \widetilde{u}_\ell, v\rangle_{X^*,X} =
    \langle \widetilde{u}_\ell,v\rangle_{H_X} \; \forall v \in X \subset H_X,  
  \end{equation*}
  $ \|w\|_{H_X} \le c_F  \|w\|_A$ (abstract Friedrichs' inequality),
  and $\|w\|_A \le c_F \|\widetilde{u}_\ell\|_{H_X}$. Thus, we can replace 
  $c_\ell = \|\widetilde{u}_\ell\|_U^2$ by the easily computable cost
  $c_\ell = \|\widetilde{u}_\ell\|_{H_X}^2 =
  (\widehat{\mathbf{M}}_\ell \widetilde{\mathbf{u}}_\ell,
  \widetilde{\mathbf{u}}_\ell )$ in Line 16, and
  $c_\ell > c_\text{\tiny cost}$ by $c_\ell > c_\text{\tiny cost} / c_F^2$
 in Line 17, where $\widehat{\mathbf{M}}_\ell $ denotes the spd mass
  matrix in $H_X$. In particular, for the $H_X$-regularization ($A = I)$ ,
  which corresponds to the $L^2$-regularization in the applications,
  the cost $c_\ell = \|\widetilde{u}_\ell\|_{U=H_X}^2 = (\widehat{\mathbf{M}}_\ell \widetilde{\mathbf{u}}_\ell,
  \widetilde{\mathbf{u}}_\ell)$ of the control can 
  be calculated  directly. 
  We note that we can also use the bound $\norm{y_{\varrho h}}_{D}$
  from Lemma~\ref{Lem:AbstractOCP:bound for the cost via state}
  without recovering the control $\widetilde{u}_{\varrho h}$ at the nested 
  levels.
\end{remark}

\subsection{Constraints}
\label{Subsection:AbstractOCP:Constraints}
To include additional constraints, e.g., on the control $u_\varrho$ or on
the state $y_\varrho$, we now consider the minimization of the reduced cost
functional \eqref{Eqn:AbstractOCP:AbstractFunctionalD} over a non-empty,
convex and closed subset $K \subset Y$, where we assume $0 \in K$ to
be satisfied. The minimizer $y_\varrho \in K$ satisfying
\[
  \check{\mathcal{J}}(y_\varrho) = \min\limits_{y \in K}
  \check{\mathcal{J}}(y) 
\]
is determined as the unique solution $y_\varrho \in K$ of the first
kind variational inequality
\begin{equation}\label{Abstract VI}
  \langle y_\varrho , y - y_\varrho \rangle_{H_Y} +
  \varrho \, \langle D y_\varrho , y - y_\varrho \rangle_{Y^*,Y} \geq
  \langle \overline{y} , y - y_\varrho \rangle_{H_Y} \quad
  \mbox{for all} \; y \in K .
\end{equation}
As in Lemma \ref{Lemma:AbstractOCP:RegularizationErrorUnconstrained}
 we can state the following result on the error
$\| y_\varrho - \overline{y} \|_{H_Y}$. 

\begin{lemma}\label{Lemma regularization error state constraint}
  Let $y_\varrho \in K$ be the unique solution of the variational inequality
  \eqref{Abstract VI}. For $\overline{y} \in H_Y$ there holds
  \[
    \| y_\varrho - \overline{y} \|_{H_Y} \leq
    \| \overline{y} \|_{H_Y}, \quad
    \| y_\varrho \|_D \leq \varrho^{-1/2} \, \| \overline{y} \|_{H_Y} ,
  \]
  while for $\overline{y} \in K$ we have
    \[
    \| y_\varrho - \overline{y} \|_S \leq \| \overline{y} \|_D, \quad
    \| y_\varrho - \overline{y} \|_{H_Y} \leq \varrho^{1/2} \,
    \| \overline{y} \|_D , \quad
    \| y_\varrho \|_D \leq \| \overline{y} \|_D .
  \]
  If in addition $D \overline{y} \in H_Y$ is satisfied for
  $\overline{y} \in K$, then the estimates 
  \[
    \| y_\varrho - \overline{y} \|_{H_Y} \leq \varrho \,
    \| D \overline{y} \|_{H_Y} \quad \mbox{and} \quad
    \| y_\varrho - \overline{y} \|_D \leq
    \varrho^{1/2} \, \| D \overline{y} \|_{H_Y} 
  \]
  follows.
\end{lemma}

\begin{proof}
  From the variational inequality \eqref{Abstract VI} we obviously have
  \[
    \varrho \, \langle D y_\varrho , y - y_\varrho \rangle_{Y^*,Y} \geq
    \langle \overline{y} - y_\varrho , y - y_\varrho \rangle_{H_Y} \quad
    \forall v \in K .
  \]
  In particular for $y=0 \in K$ this gives
  \[
    \varrho \, \langle D y_\varrho , y_\varrho \rangle_{Y^*,Y} +
    \| \overline{y} - y_\varrho \|^2_{H_Y} \leq
    \langle \overline{y} - y_\varrho , \overline{y} \rangle_{H_Y}
    \leq \| \overline{y} - y_\varrho \|_{H_Y}
    \| \overline{y} \|_{H_Y},
  \]
  i.e.,
  \[
    \| y_\varrho - \overline{y} \|_{H_Y} \leq
    \| \overline{y} \|_{H_Y}, \quad
    \| y_\varrho \|_D \leq \varrho^{-1/2} \, \| \overline{y} \|_{H_Y} .
  \]
  When assuming $\overline{y} \in K$ we can consider $y=\overline{y}$
  to obtain
  \[
    \| y_\varrho - \overline{y} \|^2_{H_Y}
    +
    \varrho \, \| y_\varrho - \overline{y} \|_D^2 
    \leq
    \varrho \, \langle D \overline{y} ,
    \overline{y} - y_\varrho \rangle_{Y,Y^*}
    \leq
    \varrho \, \|\overline{y} \|_D \| y_\varrho - \overline{y} \|_D,
  \]
  i.e.,
  \[
    \| y_\varrho - \overline{y} \|_D \leq \| \overline{y} \|_D, \quad
    \| y_\varrho - \overline{y} \|_{H_Y} \leq \varrho^{1/2} \,
    \| \overline{y} \|_D \, .
  \]
  For $y = \overline{y} \in K$ we can write \eqref{Abstract VI} also as
  \[
    \| y_\varrho - \overline{y} \|^2_{H_Y} +
    \varrho \, \| y_\varrho \|_D^2
    \leq
    \varrho \, \langle D y_\varrho , \overline{y} \rangle_{H_Y}
    \leq \varrho \, \| y_\varrho \|_D \| \overline{y} \|_D,
  \]
  i.e.,
  \[
    \| y_\varrho \|_D \leq \| \overline{y} \|_D .
  \]
  Finally, if $D \overline{y} \in H_Y$ for $\overline{y} \in K$, then
  \[
    \| y_\varrho - \overline{y} \|^2_{H_Y}
    +
    \varrho \, \| y_\varrho - \overline{y} \|_D^2 
    \leq
    \varrho \, \langle D \overline{y} ,
    \overline{y} - y_\varrho \rangle_{Y^*,Y} \leq
    \varrho \, \| D \overline{y} \|_{H_Y}
    \| y_\varrho - \overline{y} \|_{H_Y},
  \]
  implying
  \[
    \| y_\varrho - \overline{y} \|_{H_Y} \leq \varrho \,
    \| D \overline{y} \|_{H_Y}, \quad
    \| y_\varrho - \overline{y} \|_D \leq
    \varrho^{1/2} \, \| D \overline{y} \|_{H_Y} .
  \]
\end{proof}

\noindent
Note that the results of Lemma
\ref{Lemma regularization error state constraint} correspond to the
results of Lemma \ref{Lemma:AbstractOCP:RegularizationErrorUnconstrained}
when no constraints are considered.

As in the unconstrained case, let
$Y_h = \mbox{span} \{ \varphi_i \}_{i=1}^M \subset Y$
be a conforming ansatz space, and let $K_h \subset Y_h$ 
be some non-empty, convex and closed set being an appropriate
approximation of $K$. Then we consider the
Galerkin variational inequality of \eqref{Abstract VI} to find
$y_{\varrho h} \in K_h$ such that
\begin{equation}\label{Abstract VI Disk}
  \langle y_{\varrho h} , y_h - y_{\varrho h} \rangle_{H_Y} +
  \varrho \, \langle D y_{\varrho h} , y_h - y_{\varrho h} \rangle_{Y^*,Y} \geq
  \langle \overline{y} , y_h - y_{\varrho h} \rangle_{H_Y}
\end{equation}
is satisfied for all $y_h \in K_h$, which is obviously equivalent to
\[
  \langle \overline{y} - y_{\varrho h}, y_h - y_{\varrho h} \rangle_{H_Y}
  - 
  \varrho \, \langle D y_{\varrho h} , y_h - y_{\varrho h} \rangle_{Y^*,Y}
  \leq 0 \quad \mbox{for all} \; y_h \in K_h .
\]
Following \cite{Falk:1974}
we can prove the following a priori error estimate:

\begin{lemma}
  For $y_\varrho \in K$ and $y_{\varrho h} \in K_h$ being the unique
  solutions of the variational inequalities
  \eqref{Abstract VI} and \eqref{Abstract VI Disk}, respectively,
  there holds the error estimate
  \begin{eqnarray*}
    \| y_\varrho - y_{\varrho h} \|^2_{H_Y} +
    2 \, \varrho \, \| y_\varrho - y_{\varrho h} \|^2_D
    && \\
    && \hspace*{-5cm}
       \leq \, 2 \inf\limits_{y_h \in K_h} \Big[ 
       3 \, \| y_\varrho - y_h \|_{H_Y}^2 +
       \varrho \, \| y_\varrho - y_h \|^2_D \Big] +
       4 \, \| \varrho D y_\varrho + y_\varrho - \overline{y} \|_{H_Y}^2 .
  \end{eqnarray*}
\end{lemma}

\begin{proof}
  For arbitrary $y_h \in K_h$,
  \begin{eqnarray*}
    \| y_\varrho - y_{\varrho h} \|^2_{H_Y} +
    \varrho \, \| y_\varrho - y_{\varrho h} \|^2_D
    && \\
    && \hspace*{-5cm}
       = \, \langle y_\varrho - y_{\varrho h} , y_\varrho - y_{\varrho h}
       \rangle_{H_Y} + \varrho \, \langle D (y_\varrho - y_{\varrho h}),
       y_\varrho - y_{\varrho h} \rangle_{Y^*,Y} \\
    && \hspace*{-5cm}
       = \, \langle y_\varrho - y_{\varrho h} , y_\varrho - y_h
       \rangle_{H_Y} + \varrho \, \langle D (y_\varrho - y_{\varrho h}),
       y_\varrho - y_h \rangle_{Y^*,Y} \\
    && \hspace*{-4cm}
       + \langle y_\varrho - y_{\varrho h} , y_h - y_{\varrho h}
       \rangle_{H_Y} + \varrho \, \langle D (y_\varrho - y_{\varrho h}),
       y_h - y_{\varrho h} \rangle_{Y^*,Y} \\
    && \hspace*{-5cm}
       = \, \langle y_\varrho - y_{\varrho h} , y_\varrho - y_h
       \rangle_{H_Y} + \varrho \, \langle D (y_\varrho - y_{\varrho h}),
       y_\varrho - y_h \rangle_{Y^*,Y} \\
    && \hspace*{-4cm}
       + \langle \overline{y} - y_{\varrho h} , y_h - y_{\varrho h}
       \rangle_{H_Y} - \varrho \, \langle D y_{\varrho h},
       y_h - y_{\varrho h} \rangle_{Y^*,Y} \\
    && \hspace*{-3cm}
       + \langle y_\varrho - \overline{y} , y_h - y_{\varrho h}
       \rangle_{H_Y} + \varrho \, \langle D y_\varrho ,
       y_h - y_{\varrho h} \rangle_{Y^*,Y} \\
    && \hspace*{-5cm}
       \leq \, \langle y_\varrho - y_{\varrho h} , y_\varrho - y_h
       \rangle_{H_Y} + \varrho \, \langle D (y_\varrho - y_{\varrho h}),
       y_\varrho - y_h \rangle_{Y^*,Y} \\
    && \hspace*{-3cm}
       + \langle y_\varrho - \overline{y} , y_h - y_{\varrho h}
       \rangle_{H_Y} + \varrho \, \langle D y_\varrho ,
       y_h - y_{\varrho h} \rangle_{Y^*,Y} \\
    && \hspace*{-5cm}
       = \, \langle y_\varrho - y_{\varrho h} , y_\varrho - y_h
       \rangle_{H_Y} + \varrho \, \langle D (y_\varrho - y_{\varrho h}),
       y_\varrho - y_h \rangle_{Y^*,Y} \\
    && \hspace*{-3cm}
       + \langle \varrho D y_\varrho + y_\varrho - \overline{y} ,
       y_h - y_{\varrho h} \rangle_{Y^*,Y} \\
    && \hspace*{-5cm}
       \leq \, \| y_\varrho - y_{\varrho h} \|_{H_Y}
       \| y_\varrho - y_h \|_{H_Y} + \varrho \,
       \| y_\varrho - y_{\varrho h} \|_D \| y_\varrho - y_h \|_D \\
    && \hspace*{-3cm} +
       \| \varrho D y_\varrho + y_\varrho - \overline{y} \|_{H_Y}
       \| y_h - y_{\varrho h} \|_{H_Y} \, .
  \end{eqnarray*}
  When using Young's inequality, we further have
    \begin{eqnarray*}
    \| y_\varrho - y_{\varrho h} \|^2_{H_Y} +
    \varrho \, \| y_\varrho - y_{\varrho h} \|^2_D
      && \\
       && \hspace*{-5cm}
       \leq \, \frac{1}{4} \, \| y_\varrho - y_{\varrho h} \|_{H_Y}^2 +
       \| y_\varrho - y_h \|_{H_Y}^2 + \frac{1}{2} \, \varrho \,
          \| y_\varrho - y_{\varrho h} \|_D^2 +
          \frac{1}{2} \, \varrho \, \| y_\varrho - y_h \|^2_D \\
    && \hspace*{-3cm} +
       \| \varrho D y_\varrho + y_\varrho - \overline{y} \|_{H_Y}^2 +
       \frac{1}{4} \, \| y_h - y_{\varrho h} \|_{H_Y}^2 \\
       && \hspace*{-5cm}
       \leq \, \frac{1}{4} \, \| y_\varrho - y_{\varrho h} \|_{H_Y}^2 +
       \| y_\varrho - y_h \|_{H_Y}^2 + \frac{1}{2} \, \varrho \,
          \| y_\varrho - y_{\varrho h} \|_D^2 +
          \frac{1}{2} \, \varrho \, \| y_\varrho - y_h \|^2_D \\
    && \hspace*{-3cm} +
       \| \varrho D y_\varrho + y_\varrho - \overline{y} \|_{H_Y}^2 +
       \frac{1}{2} \, \| y_h - y_\varrho \|_{H_Y}^2 +
       \frac{1}{2} \, \| y_\varrho - y_{\varrho h} \|_{H_Y}^2 ,
    \end{eqnarray*}
    i.e.,
       \begin{eqnarray*}
    \frac{1}{4} \, \| y_\varrho - y_{\varrho h} \|^2_{H_Y} +
    \frac{1}{2} \, \varrho \, \| y_\varrho - y_{\varrho h} \|^2_D
      && \\
       && \hspace*{-5cm}
       \leq \, 
       \frac{3}{2} \, \| y_\varrho - y_h \|_{H_Y}^2 +
          \frac{1}{2} \, \varrho \, \| y_\varrho - y_h \|^2_D +
       \| \varrho D y_\varrho + y_\varrho - \overline{y} \|_{H_Y}^2 .
       \end{eqnarray*}
       This gives the assertion.
\end{proof}

\section{Distributed control of the Poisson equation}
\label{Section:Application}
In this section, we will describe the application of the abstract theory
to the solution of distributed control problems for the Poisson
equation, considering the control either in $L^2(\Omega)$ or in
$H^{-1}(\Omega)$, and using either a constant or a variable regularization
parameter $\varrho$.

\subsection{$H^{-1}$ regularization}
\label{Subsection:Application:EnergyRegularization}
First we consider the distributed optimal control problem to minimize
\begin{equation}\label{Eqn:Application:PoissonFunctional Energy}
  {\mathcal{J}}(y_\varrho,u_\varrho) =
  \frac{1}{2} \, \| y_\varrho - \overline{y} \|^2_{L^2(\Omega)} +
  \frac{1}{2} \, \varrho \, \| u_\varrho \|^2_{H^{-1}(\Omega)}
\end{equation}
subject to the Dirichlet boundary value problem for the Poisson equation,
\begin{equation}\label{Eqn:Application:PoissonDBVP}
  - \Delta y_\varrho = u_\varrho \quad \mbox{in} \; \Omega,
  \quad y_\varrho =0 \quad \mbox{on} \; \partial \Omega .
\end{equation}
We assume that $\Omega \subset {\mathbb{R}}^d$, $d=1,2,3$, is a bounded
domain with either smooth boundary $\partial \Omega$, or convex.
The standard variational formulation of the Dirichlet boundary value
problem \eqref{Eqn:Application:PoissonDBVP} is to find
$y_\varrho \in H^1_0(\Omega)$ such that
\begin{equation}\label{Eqn:Application:Poisson DBVP VF}
  \int_\Omega \nabla y_\varrho(x) \cdot \nabla y(x) \, dx =
  \langle u_\varrho , y \rangle_\Omega
\end{equation}
is satisfied for all $v \in H^1_0(\Omega)$. The variational formulation
\eqref{Eqn:Application:Poisson DBVP VF} admits a unique solution
$y_\varrho \in Y := H^1_0(\Omega)$ when assuming
$u_\varrho \in U:= H^{-1}(\Omega) = X^* = [H^1_0(\Omega)]^*$, i.e.,
$X=H^1_0(\Omega)$, and $H_X = H_Y = L^2(\Omega)$.
When using the norm $\| y \|_Y = \| \nabla y \|_{L^2(\Omega)}$ we easily
conclude the abstract assumptions for  
$A = B = - \Delta : H^1_0(\Omega) \to H^{-1}(\Omega)$
with $c_1^A = c_1^B = c_2^A = c_2^B = 1$. Moreover,
$S = B^* A^{-1} B = - \Delta : H^1_0(\Omega) \to H^{-1}(\Omega)$
with $c_1^S = c_2^S = 1$. 
Hence we can rewrite the abstract variational formulation
\eqref{Eqn:AbstractOCP:Abstract VF} to find $y_\varrho \in H^1_0(\Omega)$
such that
\begin{equation}\label{Eqn:Application:VF}
  \langle y_\varrho , y \rangle_{L^2(\Omega)} +
  \varrho \, \langle \nabla y_\varrho , \nabla y \rangle_{L^2(\Omega)}
  =
  \langle \overline{y} , y \rangle_{L^2(\Omega)}
\end{equation}
is satisfied for all $y \in H^1_0(\Omega)$. 
Note that, in fluid mechanics, the variational problem \eqref{Eqn:Application:VF} 
is known as differential filter; see, e.g., \cite{John2016Book}. 
The results of Lemma \ref{Lemma:AbstractOCP:RegularizationErrorUnconstrained} now read
\[
  \| y_\varrho - \overline{y} \|_{L^2(\Omega)} \leq
  \left \{
    \begin{array}{ccl}
      \| \overline{y} \|_{L^2(\Omega)}
      & & \mbox{if} \; \overline{y} \in L^2(\Omega), \\[1mm]
      \varrho^{1/2} \, \| \nabla \overline{y} \|_{L^2(\Omega)}
      & & \mbox{if} \; \overline{y} \in H^1_0(\Omega), \\[1mm]
      \varrho \, \| \Delta \overline{y} \|_{L^2(\Omega)}
      & & \mbox{if} \; \overline{y} \in H^1_0(\Omega,\Delta),
    \end{array} \right.
\]
where we have used
$H^1_0(\Omega,\Delta) := \{ y \in H^1_0(\Omega) : \Delta y \in L^2(\Omega) \}$.
Moreover,
\[
  \| \nabla(y_\varrho - \overline{y}) \|_{L^2(\Omega)} \leq
  \left \{
    \begin{array}{ccl}
      \| \nabla \overline{y} \|_{L^2(\Omega)}
      & & \mbox{if} \; \overline{y} \in H^1_0(\Omega), \\[1mm]
      \varrho^{1/2} \, \| \Delta \overline{y} \|_{L^2(\Omega)}
      & & \mbox{if} \; \overline{y} \in H^1_0(\Omega,\Delta).
    \end{array} \right.
\]
For the discretization of the variational formulation
\eqref{Eqn:Application:VF} we introduce the standard finite
element space $Y_h = \mbox{span} \{ \varphi_i \}_{i=1}^M \subset
Y = H^1_0(\Omega)$ of piecewise linear continuous basis functions
$\varphi_i$ which are defined with respect to an admissible decomposition 
$\mathcal{T}_h$ of $\Omega$ into simplicial shape regular finite elements
$\tau$ of local mesh size $h_\tau$, and $h = \max_{\tau \in \mathcal{T}_h} h_\tau$.
For $y \in Y = H^1_0(\Omega)$ let $P_hy \in Y_h$ be the unique solution
of the variational formulation satisfying
\[
  \int_\Omega \nabla P_h y \cdot \nabla z_h \, dx =
  \int_\Omega \nabla y \cdot \nabla z_h \, dx \quad
  \mbox{for all} \; z_h \in Y_h .
\]
When using standard finite element error estimates we immediately
have
\[
  \| \nabla (y-P_hy) \|_{L^2(\Omega)} \leq \| \nabla y \|_{L^2(\Omega)}
  \quad \mbox{for} \; y \in H^1_0(\Omega),
\]
and
\[
  \| \nabla (y-P_hy) \|_{L^2(\Omega)} \leq c \, h \, |y|_{H^2(\Omega)} \leq
  c \, h \, \| \Delta y \|_{L^2(\Omega)} \quad
  \mbox{for} \; y \in H^1_0(\Omega,\Delta),
\]
where for the last inequality we need that $\Omega$ is either smoothly
bounded or convex, as assumed. In this case, and using the
Aubin--Nitsche trick, we also have the error estimates
\[
  \| y - P_h y \|_{L^2(\Omega)} \leq
  c \, h \, \| \nabla y \|_{L^2(\Omega)}, \quad
  \| y - P_h y \|_{L^2(\Omega)} \leq
  c \, h^2 \, \| \Delta y \|_{L^2(\Omega)}.
\]
Hence we have established the abstract assumptions
\eqref{Eqn:AbstractOCP:Approximation Assumption} and
\eqref{Eqn:AbstractOCP:Approximation Assumption 2}, with $\mbox{$\alpha=1$}$,
implying the optimal choice $\varrho = h^2$.
The finite element variational formulation of
\eqref{Eqn:Application:VF} reads to find $y_{\varrho h} \in Y_h$ such that
\begin{equation}
\label{Eqn:Application:FEM}
  \langle y_{\varrho h} , y_h \rangle_{L^2(\Omega)} +
  \varrho \, \langle \nabla y_{\varrho h} , \nabla y_h \rangle_{L^2(\Omega)}
  =
  \langle \overline{y} , y_h \rangle_{L^2(\Omega)}
\end{equation}
is satisfied for all $y_h \in Y_h$. 
For the choice $\varrho = h^2$, the error
estimates \eqref{Eqn:AbstractOCP:Error target} read
\[
   \| y_{\varrho h} - \overline{y} \|_{L^2(\Omega)} \leq c \, \left \{
      \begin{array}{ccl}
        \| \overline{y} \|_{L^2(\Omega)}
        & & \mbox{for} \; \overline{y} \in L^2(\Omega), \\[2mm]
        h \, \| \nabla \overline{y} \|_{L^2(\Omega)}
        & & \mbox{for} \; \overline{y} \in H^1_0(\Omega), \\[2mm]
        h^2 \, \| \Delta \overline{y} \|_{L^2(\Omega)}
        && \mbox{for} \; \overline{y} \in H^1_0(\Omega,\Delta) =
           H^1_0(\Omega) \cap H^2(\Omega) .
      \end{array} \right.
\]
Finally, for $s \in [0,1]$ we define the interpolation space
$H^s_0(\Omega) = [L^2(\Omega),H^1_0(\Omega)]_{|s}$, 
and when assuming $\overline{y} \in H^s_0(\Omega)$ the error estimate
\eqref{Eqn:AbstractOCP:Final Error 3} reads
\begin{equation}
\label{Eqn:Application:ErrorEstimateH^s}
  \| y_{\varrho h} - \overline{y} \|_{L^2(\Omega)} \, \leq \,
  c \, h^s \, \| \overline{y} \|_{H^s_0(\Omega)} .
\end{equation}
Once the basis is chosen, the finite element scheme \eqref{Eqn:Application:FEM}
is equivalent to the linear system of algebraic equations
\begin{equation}\label{Eqn:Application:Energy LGS}
(\mathbf{M}_h + \varrho \, \mathbf{K}_h) \mathbf{y}_{\varrho h} = \overline{\mathbf{y}}_h,
\end{equation}
where the mass and stiffness matrices are defined via their entries
\[
 M_h[j,i] = \int_\Omega \varphi_i(x) \varphi_j(x) \, dx 
 \quad \mbox{and} \quad
 K_h[j,i] = \int_\Omega \varphi_i(x) \varphi_j(x) \, dx 
\]
for $i,j=1,\ldots,M$,
and the source vector $\overline{\mathbf{y}}_h$ via its coefficients
\[
  \overline{y}_j =
  \int_\Omega \overline{y}(x) \varphi_j(x) \, dx \quad
  \mbox{for} \; j=1,\ldots,M.
\]
\subsection{Variable $H^{-1}$ regularization}
Instead of the variational formulation \eqref{Eqn:Application:VF} with a
constant regularization parameter $\varrho$, we now consider a variational
formulation with a suitable regularization function $\varrho(x)$, $x \in \Omega$. 
For a given decomposition ${\mathcal{T}}_h$ of $\Omega$
into finite elements $\tau$ of local mesh size $h_\tau$, we define
the mesh dependent regularization function
\[
\varrho_h(x) = h_\tau^2 \quad \mbox{for} \; x \in \tau ,
\]
and the mesh dependent norm
\[
  \| y \|^2_{H^1_0(\Omega),\varrho_h} :=
  \int_\Omega \varrho_h(x) \, |\nabla y(x)|^2 \, dx \quad
  \mbox{for} \; y \in H^1_0(\Omega).
\]
Then we consider the variational formulation to find
$y_{\varrho_h} \in H^1_0(\Omega)$ such that
\begin{equation}\label{Eqn:Application:Variable VF}
  \int_\Omega y_{\varrho_h}(x) y(x) \, dx +
  \int_\Omega \varrho_h(x) \, \nabla y_{\varrho_h}(x) \cdot \nabla y(x) \, dx
  =
  \int_\Omega \overline{y}(x) y(x) \, dx
\end{equation}
is satisfied for all $y \in H^1_0(\Omega)$. 

As in Lemma \ref{Lemma:AbstractOCP:RegularizationErrorUnconstrained}
we conclude the regularization error estimates
\[
  \| y_{\varrho_h} - \overline{y} \|_{L^2(\Omega)} \leq
  \| \overline{y} \|_{L^2(\Omega)}, \quad
  \| y_{\varrho_h} \|_{L^2(\Omega)} \leq \| \overline{y} \|_{L^2(\Omega)},
  \quad
  \| y_{\varrho_h} \|_{H^1_0(\Omega),\varrho_h} \leq
    \| \overline{y} \|_{L^2(\Omega)}
\]
for $\overline{y} \in L^2(\Omega)$, and
\[
  \| y_{\varrho_h} - \overline{y} \|_{H^1_0(\Omega),\varrho_h} \leq
  \| \overline{y} \|_{H^1_0(\Omega),\varrho_h}, \quad
  \| y_{\varrho_h} - \overline{y} \|_{L^2(\Omega)} \leq
  \| \overline{y} \|_{H^1_0(\Omega),\varrho_h},
\]
as well as
\[
  \| y_{\varrho_h} \|_{H^1_0(\Omega),\varrho_h} \leq
  \| \overline{y} \|_{H^1_0(\Omega),\varrho_h}
\]
for $\overline{y} \in H^1_0(\Omega)$.

The finite element discretization of \eqref{Eqn:Application:Variable VF}
reads to find $y_{\varrho_h h} \in Y_h$ such that
\begin{equation}\label{Eqn:Application:Variable FEM}
  \int_\Omega y_{\varrho_h h}(x) y_h(x) \, dx +
  \int_\Omega \varrho_h(x) \, \nabla y_{\varrho_h h}(x) \cdot \nabla y_h(x) \, dx
  =
  \int_\Omega \overline{y}(x) y_h(x) \, dx
\end{equation}
is satisfied for all $y_h \in Y_h$.
In this case, Cea's lemma \eqref{Eqn:AbstractOCP:AbstractCea} reads
\begin{eqnarray*}
  && \| y_{\varrho_h} - y_{\varrho_h,h} \|^2_{L^2(\Omega)} +
     \| y_{\varrho_h} - y_{\varrho_h,h} \|^2_{H^1_0(\Omega),\varrho_h} \\
  && \hspace*{3cm}
     \leq \inf\limits_{y_h \in Y_h}
     \Big[
     \| y_{\varrho_h} - y_h \|^2_{L^2(\Omega)} +
     \| y_{\varrho_h} - y_h \|^2_{H^1_0(\Omega),\varrho_h}
     \Big] ,
\end{eqnarray*}
and when choosing $y_h \equiv 0$ this gives
\[
  \| y_{\varrho_h} - y_{\varrho_h,h} \|_{L^2(\Omega)} \leq
  \sqrt{2} \, \| \overline{y} \|_{L^2(\Omega)} \quad
  \mbox{for} \; \overline{y} \in L^2(\Omega) .
\]
Moreover, when considering some quasi-interpolation $y_h = P_h y_{\varrho_h}$
we obtain the error estimate
\[
  \| y_{\varrho_h} - y_{\varrho_hh} \|_{L^2(\Omega)} \leq
  c \, \| \overline{y} \|_{H^1_0(\Omega),\varrho_h} =
  c \, \left(
    \sum\limits_{\tau \subset {\mathcal{T}}_h}
    h_\tau^2 \, \| \nabla \overline{y} \|^2_{L^2(\tau)}
  \right)^{1/2} \quad \mbox{for} \;
  \overline{y} \in H^1_0(\Omega) .
\]
When combining these results with the regularization error estimates,
we finally obtain
\[
  \| y_{\varrho_hh} - \overline{y} \|_{L^2(\Omega)} \leq
  c \, \| \overline{y} \|_{L^2(\Omega)}, \quad
  \| y_{\varrho_hh} - \overline{y} \|_{L^2(\Omega)} \leq
  c \, \left(
    \sum\limits_{\tau \subset {\mathcal{T}}_h}
    h_\tau^2 \, \| \nabla \overline{y} \|^2_{L^2(\tau)}
  \right)^{1/2} .
\]

Instead of \eqref{Eqn:Application:Energy LGS}, we now conclude
the linear system
\begin{equation}
  (\mathbf{M}_h + \mathbf{K}_{\varrho_h h}) \mathbf{y}_{\varrho h} = \overline{\mathbf{y}}_h
\end{equation}
from the finite element scheme \eqref{Eqn:Application:Variable FEM},
where the entries of the diffusion type stiffness matrix
$\mathbf{K}_{\varrho_h h}$
are now given by
\[
  K_{\varrho_hh}[j,i] = \int_\Omega \varrho_h(x) \,
  \nabla \varphi_i(x) \cdot \nabla \varphi_j(x) \, dx, \quad
  i,j=1,\ldots,M.
\]
%

%
%

\subsection{$L^2$ regularization}
Instead of \eqref{Eqn:Application:PoissonFunctional Energy} we now
consider the optimal control problem to minimize
\begin{equation}\label{Eqn:Application:PoissonFunctional L2}
  {\mathcal{J}}(y_\varrho,u_\varrho) =
  \frac{1}{2} \, \| y_\varrho - \overline{y} \|^2_{L^2(\Omega)} +
  \frac{1}{2} \, \varrho \, \| u_\varrho \|^2_{L^2(\Omega)}
\end{equation}
subject to the Dirichlet boundary value problem
\eqref{Eqn:Application:PoissonDBVP} where we now consider the source
term $u_\varrho \in U = X^* = L^2(\Omega)$, i.e., $X=L^2(\Omega)$.
For the solution $y_\varrho$ of \eqref{Eqn:Application:PoissonDBVP}
we therefore have $Y = H^1_0(\Omega,\Delta) := \{ y \in H^1_0(\Omega) :
\Delta y \in L^2(\Omega)\}$, with norm
$\| y \|_{H^1_0(\Omega,\Delta)} = \| \Delta y \|_{L^2(\Omega)}$. 
Using $B = - \Delta : H^1_0(\Omega,\Delta) \to L^2(\Omega)$, 
we therefore have
\[
  \| B y \|_{L^2(\Omega)} = \| \Delta y \|_{L^2(\Omega)} =
  \| y \|_{H^1_0(\Omega,\Delta)} 
\]
for all $ 0 \neq y \in H_0^1(\Omega,\Delta)$,
and
\[
  \| y \|_{Y=H_0^1(\Omega,\Delta)} =
  \| \Delta y \|_{L^2(\Omega)} =
  \frac{\langle - \Delta y , - \Delta y \rangle_{L^2(\Omega)}}
  {\| \Delta y \|_{L^2(\Omega)}} \leq
  \sup\limits_{0 \neq q \in L^2(\Omega)}
  \frac{\langle - \Delta y , q \rangle_{L^2(\Omega)}}{\| q \|_{L^2(\Omega)}},
\]
i.e., $c_1^B = c_2^B = 1$. Moreover, $A = I : L^2(\Omega) \to L^2(\Omega)$,
with $ c_1^A = c_2^A = 1$. 
Thus, we define
$S := B^* B : H^1_0(\Omega,\Delta) \to [H^1_0(\Omega,\Delta)]^*$
with $c_1^S = c_2^S = 1$, and where
$B^* : L^2(\Omega) \to [H^1_0(\Omega,\Delta)]^*$ is the adjoint of
$B : H^1_0(\Omega,\Delta) \to L^2(\Omega)$ satisfying
\[
  \langle B^* q , y \rangle_\Omega
  = \langle q , By \rangle_{L^2(\Omega)} \quad
  \mbox{for all} \; (q,y) \in L^2(\Omega) \times H^1_0(\Omega,\Delta).
\]
Hence, we can rewrite the abstract variational formulation
\eqref{Eqn:AbstractOCP:Abstract VF} to find
$y_\varrho \in H^1_0(\Omega,\Delta)$ such that
\begin{equation}\label{Eqn:Application:VF L2}
  \langle y_\varrho , y \rangle_{L^2(\Omega)} + \varrho \,
  \langle \Delta y_\varrho , \Delta y \rangle_{L^2(\Omega)} \, = \,
  \langle \overline{y} , y \rangle_{L^2(\Omega)}
\end{equation}
is satisfied for all $y \in H^1_0(\Omega,\Delta)$. Note that
\eqref{Eqn:Application:VF L2} is the variational formulation of the
Dirchlet problem for the BiLaplace equation,
\begin{equation}\label{Eqn:Application:L2 BiLaplace}
  \varrho \, \Delta^2 y_\varrho + y_\varrho = \overline{y} \quad
  \mbox{in} \; \Omega, \quad
  y_\varrho = \Delta y_\varrho = 0 \quad
  \mbox{on} \; \partial \Omega .
\end{equation}
The results of Lemma \ref{Lemma:AbstractOCP:RegularizationErrorUnconstrained}
now read
\[
  \| y_\varrho - \overline{y} \|_{L^2(\Omega)} \leq
  \left \{
    \begin{array}{ccl}
      \| \overline{y}  \|_{L^2(\Omega)}
      & & \mbox{if} \; \overline{y} \in L^2(\Omega), \\[1mm]
      \varrho^{1/2} \, \| \Delta \overline{y} \|_{L^2(\Omega)}
      && \mbox{if} \; \overline{y} \in H^1_0(\Omega,\Delta),
    \end{array}
  \right.
\]
and
\[
  \| \Delta (y_\varrho - \overline{y}) \|_{L^2(\Omega)} \leq
  \| \Delta \overline{y} \|_{L^2(\Omega)} \quad \mbox{for} \;
  \overline{y} \in H^1_0(\Omega,\Delta).
\]
For a conforming finite element discretization of the variational
formulation \eqref{Eqn:Application:VF L2} we need to introduce
an ansatz space $Y_h \subset Y = H^1_0(\Omega,\Delta)$. 
At this time,
and for simplicity of the presentation, let us first consider the case
$d=1$
and $\Omega = (0,1)$. In this case, we have $Y=H^1_0(0,1) \cap H^2(0,1)$,
and for a conforming ansatz space we can use the space
$Y_h = S_h^2(0,1) \cap H^1_0(0,1)$ of second order B splines. 
Since, for
$d=1$,
the nodal interpolation operator $I_h : Y \to Y_h$ is
well defined and bounded, we can write the abstract
assumptions \eqref{Eqn:AbstractOCP:Approximation Assumption} as
\[
  \| y - I_h y \|_{L^2(0,1)} \leq c_1 \, h^2 \, \| y'' \|_{L^2(0,1)}, \quad
  \| (y-I_hy)'' \|_{L^2(0,1)}  \leq c_2 \, \| y'' \|_{L^2(0,1)}, 
\]
i.e., $\alpha = 2$, implying the optimal choice $\varrho = h^4$.
The Galerkin finite element formulation of \eqref{Eqn:Application:VF L2}
then reads to find $y_{\varrho h} \in Y_h$ such that
\begin{equation}\label{Eqn:Application:VF L2 FEM}
  \langle y_{\varrho h} , y_h \rangle_{L^2(0,1)} + \varrho \,
  \langle y_{\varrho h}'' , y_h'' \rangle_{L^2(0,1)} 
  \, = \,
  \langle \overline{y} , y_h \rangle_{L^2(0,1)}
\end{equation}
is satisfied for all $y_h \in Y_h$, and, for the error estimate
\eqref{Eqn:AbstractOCP:Error target}, we obtain, for $\varrho=h^4$,
\[
  \| y_{\varrho h} - \overline{y} \|_{L^2(0,1)} \leq c \,
  \left \{
    \begin{array}{ccl}
      \| \overline{y} \|_{L^2(0,1)}
      & & \mbox{for} \; \overline{y} \in L^2(0,1), \\[2mm]
      h^2 \, \| \overline{y}'' \|_{L^2(0,1)}
      & & \mbox{for} \; \overline{y} \in H^1_0(0,1) \cap H^2(0,1). 
    \end{array} \right.    
\]  
Finally, when using some space interpolation arguments, we conclude
the error estimate
\[
  \| y_{\varrho h} - \overline{y} \|_{L^2(0,1)} \leq
  c \, h^{2s} \, \| \overline{y} \|_{
    [L^2(0,1),H^1_0(0,1) \cap H^2(0,1)]_{|s}}
\]
provided that
$\overline{y} \in [L^2(0,1),H^1_0(0,1) \cap H^2(0,1)]_{|s}$
for some $s \in [0,1]$.

Although we can generalize the above approach to domains
$\Omega = (0,1)^d \subset {\mathbb{R}}^d$, $d=2,3$ 
by using tensor product finite element spaces or IgA spaces,
the construction of conforming
finite element spaces $Y_h \subset H^1_0(\Omega,\Delta)$ with respect to
simplicial decompositions of $\Omega$ seems to be more challenging,
e.g., \cite{AinsworthParker:2024}.
Hence, we will describe an alternative non-conforming approach as follows.
For $y_\varrho \in H^1_0(\Omega,\Delta)$ being the unique solution of \eqref{Eqn:Application:VF L2}, 
we define $p_\varrho = \varrho \Delta y_\varrho \in L^2(\Omega)$, and we can
rewrite the Dirichlet boundary value problem for the BiLaplace
equation \eqref{Eqn:Application:L2 BiLaplace} as system,
\[
  - \Delta p_\varrho = y_\varrho - \overline{y}, \quad
  \frac{1}{\varrho} \, p_\varrho - \Delta y_\varrho = 0 \quad \mbox{in} \;
  \Omega, \quad
  y_\varrho = p_\varrho = 0 \; \mbox{on} \; \partial \Omega .
\]
From this system we conclude $p_\varrho,y_\varrho \in H^1_0(\Omega)$,
and in the sequel $y_\varrho \in H^1_0(\Omega,\Delta)$.
The related variational formulation reads to find
$(p_\varrho , y_\varrho) \in H^1_0(\Omega) \times H^1_0(\Omega)$
such that
\begin{equation}\label{Eqn:Application:L2 System}
  \frac{1}{\varrho} \, \langle p_\varrho , q \rangle_{L^2(\Omega)} +
  \langle \nabla y_\varrho , \nabla q \rangle_{L^2(\Omega)} = 0, \quad
  \langle \nabla p_\varrho , \nabla y \rangle_{L^2(\Omega)} =
  \langle y_\varrho - \overline{y}, y \rangle_{L^2(\Omega)}
\end{equation}
is satisfied for all $(q,y) \in H^1_0(\Omega) \times H^1_0(\Omega)$;
cf. also Section~\ref{Section:Introduction}.

For the discretization of \eqref{Eqn:Application:L2 System}, 
we now use the conforming
finite element space $V_h := S_h^1(\Omega) \cap H^1_0(\Omega)
= \mbox{span} \{ \varphi_i \}_{i=1}^M$ of piecewise linear and continuous
basis functions $\varphi_i$, as already used in the case of the $H^{-1}$
regularization. This results in a coupled linear system of
algebraic equations  
\[
  \left(
    \begin{array}{cc}
      \frac{1}{\varrho} \, \mathbf{M}_h & \mathbf{K}_h \\[1mm]
      -\mathbf{K}_h & \mathbf{M}_h
    \end{array}
  \right)
  \left(
    \begin{array}{c}
       \mathbf{p}_{\varrho h} \\[1mm]
       \mathbf{y}_{\varrho h}
    \end{array}
  \right)
  \, = \,
  \left(
    \begin{array}{c}
      \mathbf{0}_h \\[1mm]
      \overline{\mathbf{y}}_h
    \end{array} \right) ,
\]
which is equivalent to the Schur complement system
\[
  ( \mathbf{M}_h + \varrho \mathbf{K}_h \mathbf{M}_h^{-1} \mathbf{K}_h) %
  \mathbf{y}_{\varrho h} = \overline{\mathbf{y}}_h \, .
\]
%
%
\subsection{Variable $L^2$ regularization}

Instead of \eqref{Eqn:Application:VF L2} we now consider a
variational formulation to find $y_{\varrho_h} \in H^1_0(\Omega,\Delta)$
such that
\[
  \int_\Omega y_{\varrho_h}(x) y(x) \, dx +
  \int_\Omega \varrho_h(x) \, \Delta y_{\varrho_h}(x) \,
  \Delta y(x) \, dx \, = \, \int_\Omega \overline{y}(x) \, y(x) \, dx
\]
is satisfied for all $y \in H^1_0(\Omega,\Delta)$, with
the mesh dependent regularization function
\[
\varrho_h(x) = h_\tau^4 \quad \mbox{for} \; x \in \tau .
\]
When introducing
$p_{\varrho_h} = \varrho_h \, \Delta y_{\varrho_h}$, we end up with
a variational system to find
$(p_{\varrho_h},y_{\varrho_h}) \in H^1_0(\Omega) \times H^1_0(\Omega)$
such that
\[
  \int_\Omega \frac{1}{\varrho_h(x)} \, p_{\varrho_h}(x) \, q(x) \, dx +
  \int_\Omega \nabla y_{\varrho_h}(x) \cdot \nabla q(x) \, dx = 0
\]
is satisfied for all $q \in H^1_0(\Omega)$, and
\[
  - \int_\Omega \nabla p_{\varrho_h}(x) \cdot \nabla y(x) \, dx +
  \int_\Omega y_{\varrho_h}(x) y(x) \, dx =
  \int_\Omega \overline{y}(x) y(x) \, dx
\]
is satisfied for all $y \in H^1_0(\Omega)$. The finite element
discretization of this system results in a linear system of
algebraic equations,
\[
  \left(
    \begin{array}{cc}
      \mathbf{M}_{1/\varrho_h,h} & \mathbf{K}_h \\[1mm]
      - \mathbf{K}_h & \mathbf{M}_h
    \end{array}
  \right)
  \left(
    \begin{array}{c}
      \mathbf{p}_{\varrho h} \\[1mm]
      \mathbf{y}_{\varrho h}
    \end{array}
  \right)
  \, = \,
  \left(
    \begin{array}{c}
      \mathbf{0}_h\\[1mm]
      \overline{\mathbf{y}}_h
    \end{array} \right) ,
\]
where the scaled mass matrix $\mathbf{M}_{h,1/\varrho_h}$ is given by its
entries
\[
  M_{1/\varrho_h,h}[j,i] =
  \int_\Omega \frac{1}{\varrho(x)} \, \varphi_i(x) \, \varphi_j(x) \, dx
  \quad \mbox{for} \; i,j=1,\ldots, M.
\]
When eliminating $\underline{p}$, we end up with the
Schur complement system
\[
  ( \mathbf{M}_h + \varrho \mathbf{K}_h \mathbf{M}_{1/\varrho_h,h}^{-1} \mathbf{K}_h) 
  \mathbf{y}_{\varrho_h h} = \overline{\mathbf{y}}_h \, .
\]
%
%
\subsection{Control recovering}
\label{Subsection:Application:EnergyRegularisation:ControlRecovering}
For the finite element approximation of the control $u_\varrho = B y_\varrho$,
we consider the abstract variational
formulation \eqref{Eqn:AbstractOCP:Control VF} with
$U= X^* =H^{-1}(\Omega)$ and $X = H^1_0(\Omega)$. 
In this particular situation, we can choose the finite element space
$X_h = Y_h = \mbox{span} \{ \varphi_k \}_{k=1}^M$
of piecewise linear continuous basis functions $\varphi_k$,
where the assumption \eqref{Eqn:AbstractOCP:Assumption Pi}
coincides with the second assumption
in \eqref{Eqn:AbstractOCP:Approximation Assumption 2} which was
already established.
When an approximate state $y_{\varrho h} \in Y_h \subset H^1_0(\Omega)$
is known, we can compute the related control
$\widetilde{u}_{\varrho h} \in U_h
= \mbox{span} \{ \psi_k \}_{k=1}^M \subset L^2(\Omega) \subset H^{-1}(\Omega)$
as unique solution of the variational formulation
\begin{equation}\label{Eqn:Application:Control Reconstruction}
  \int_\Omega \widetilde{u}_{\varrho h}(x) \phi_h(x) \, dx =
  \int_\Omega \nabla y_{\varrho h}(x) \cdot \nabla y_h(x) \, dx \quad
  \mbox{for all} \; y_h \in Y_h .
\end{equation}
It remains to define $U_h$ in order to satisfy the discrete inf-sup
condition \eqref{Eqn:AbstractOCP:Control inf-sup}, which now reads
\[
  c_S \, \| u_h \|_{H^{-1}(\Omega)}
  \leq \sup\limits_{y_h \in Y_h \subset H^1_0(\Omega)}
  \frac{\langle u_h , y_h \rangle_{L^2(\Omega)}}{\| \nabla y_h
    \|_{L^2(\Omega)}}
  \quad \mbox{for all} \; u_h \in U_h .
\]
A first choice is to consider the control space
$U_h = Y_h \subset H^1_0(\Omega)$ of piecewise linear and continuous basis
functions, i.e., we have to solve the linear system
\eqref{Eqn:AbstractOCP:AlgebraicSystem4Control} with
$\overline{\mathbf M}_h = \mathbf{M}_h$ and ${\mathbf B}_h = \mathbf{K}_h$.
Now the discrete inf-sup condition is equivalent to the stability estimate
\[
\| Q_h y \|_{H^1(\Omega)} \leq \frac{1}{c_S} \,
\| y \|_{H^1(\Omega)}
\]
for the $L^2$ projection
$Q_h : L^2(\Omega) \to Y_h \subset H^1_0(\Omega) \subset L^2(\Omega)$,
see, e.g., \cite{BramblePasciakSteinbach:2002}, which also covers
adaptive meshes.

For the approximate control $\widetilde{u}_{\varrho h}$ we can compute
the related state $\widetilde{y}_\varrho \in H^1_0(\Omega)$ as unique
solution of the variational formulation
\[
  \int_\Omega \nabla \widetilde{y}_\varrho(x) \cdot \nabla y(x) \, dx =
  \int_\Omega \widetilde{u}_{\varrho h} \, y(x) \, dx \quad
  \mbox{for all} \; y \in H^1_0(\Omega),
\]
and from \eqref{Eqn:AbstractOCP:Final Error 3} we conclude the error estimate
\begin{equation}\label{Eqn:Application:Final Error 3}
  \| \widetilde{y}_\varrho - \overline{y} \|_{L^2(\Omega)} \leq
  c \, h^s \, \| \overline{y} \|_{H^s_0(\Omega)} \quad
  \mbox{for} \; \overline{y} \in H^s_0(\Omega) =
  [L^2(\Omega,H^1_0(\Omega))]_{|s}, s \in [0,1],
\end{equation}
when choosing $\varrho = h^2$ in the case of $H^{-1}$ regularization,
and $\varrho = h^4$ in the case of $L^2$ regularization.

Due to the choice $U_h = Y_h \subset H^1_0(\Omega)$ the discrete control
$\widetilde{u}_{\varrho h}$ is much more regular than expected, i.e., it
involves boundary conditions. Although this does not effect the final
error estimate \eqref{Eqn:Application:Final Error 3}, the shape of the
piecewise linear and continuous control $\widetilde{u}_{\varrho h}$ may be
not feasible. As an alternative we aim to construct a piecewise constant
discrete control.

For a given admissible decomposition of $\Omega \subset {\mathbb{R}}^d$ into
shape regular simplicial finite elements $\tau$ we introduce a dual mesh as
follows: For any interior node $x_k \in \Omega$ we define a dual finite
element $\omega_k$ satisfying $\omega_k \cap \omega_j = \emptyset$ for
$x_k \neq x_j$ such that ${\mathcal{T}}_h = \cup_{k=1}^M \overline{\omega}_k$,
see Figure \ref{Fig:mesh and dual mesh} and, e.g., \cite{Steinbach:2002}.
Then we define $U_h = \mbox{span} \{ \psi_k \}_{k=1}^M \subset U=H^{-1}(\Omega)$
as ansatz space of piecewise constant basis functions $\psi_k$ which are one
in $\omega_k$, and zero elsewhere. While the approximation assumption
\eqref{Eqn:AbstractOCP:Assumption Pi} remains unchanged, the discrete
inf-sup condition \eqref{Eqn:AbstractOCP:Control inf-sup} follows
as in \cite{Steinbach:2002}. Moreover, the error estimate
\eqref{Eqn:Application:Final Error 3} remains true. But instead of the
standard mass matrix $\overline{\mathbf{M}}_h = \mathbf{M}_h$, we now have
to use a matrix $\overline{\mathbf{M}}_h = \widetilde{\mathbf{M}}_h$
defined by the entries
\[
  \widetilde{{M}}_h[j,k] = \int_\Omega \varphi_k(x) \, \psi_j(x) \, dx =
  \int_{\omega_j} \varphi_k(x) \, dx, \quad j,k=1,\ldots,M.
\]
Thus, the linear system \eqref{Eqn:AbstractOCP:AlgebraicSystem4Control} now
takes the form $\widetilde{\mathbf{M}}_h \mathbf{u}_{\varrho h} =
\mathbf{K}_h \mathbf{y}_{\varrho h}$. Note, that we additionally have that
\begin{equation*}
  \norm{\nabla y_h}_{L^2(\Omega)} = \sup_{0\neq x_h\in Y_h}
  \frac{\langle \nabla y_h, \nabla x_h \rangle_{L^2(\Omega)}}
  {\norm{\nabla x_h}_{L^2(\Omega)}},
\end{equation*}
and therefore the discrete inf-sup condition \eqref{Eqn:AbstractOCP:discrete inf sup Y} is satisfied. Thus, by Lemma \ref{Lem:AbstractOCP:bound for the cost via state} the cost can be estimated by 
\begin{equation}\label{Eqn:Application:cost bounds laplacian}
  \norm{\nabla y_{\varrho h}}_{L^2(\Omega)} \leq \norm{\widetilde{u}_{\varrho h}}_{H^{-1}(\Omega)}\leq c_S \norm{\nabla y_{\varrho h}}_{L^2(\Omega)}. 
\end{equation}

\begin{remark}
  For $\mathbf{u}= (u_1,\ldots,u_M)^\top \in\mathbb{R}^M$ we compute 
  \begin{equation*}
    (\widetilde{\mathbf{M}}_h\mathbf{u},\mathbf{u})_2 =
    \sum_{i,j=1}^M u_iu_j\int_\Omega \psi_j(x)\varphi_i(x)\, dx =
    \sum_{\ell=1}^N \sum_{i,j=1}^M u_iu_j
    \int_{\tau_\ell} \psi_j(x)\varphi_i(x)\, dx.  
  \end{equation*}
  The local element matrices with entries
  \begin{equation*}
    \widetilde{M}_{\tau_\ell}[i,j] :=
    \int_{\tau_\ell} \psi_j(x)\varphi_i(x)\, dx,\quad \ell=1,\ldots,N,
  \end{equation*}
  can be computed to be, see \cite{Steinbach:2002}, for $d=1$
  \begin{equation*}
    \widetilde{\mathbf{M}}_{\tau_\ell} =
    \frac{|\tau_\ell|}{8}
    \begin{pmatrix}
      3 & 1\\ 1& 3
    \end{pmatrix} \quad \text{with }
    \lambda_{\text{min}}(\widetilde{\mathbf{M}}_{\tau_\ell}) =
    \frac{|\tau_\ell|}{4},\;
    \lambda_{\text{max}}(\widetilde{\mathbf{M}}_{\tau_\ell}) =
    \frac{|\tau_\ell|}{2},
  \end{equation*}
  and for $d=2$
  \begin{equation*}
    \widetilde{\mathbf{M}}_{\tau_\ell} =
    \frac{|\tau_\ell|}{108}
    \begin{pmatrix}
      22 & 7 & 7\\ 7 & 22 & 7\\ 7 & 7 & 22
    \end{pmatrix} \quad \text{with }
    \lambda_{\text{min}}(\widetilde{\mathbf{M}}_{\tau_\ell}) =
    \frac{5}{36}|\tau_\ell|,\;
    \lambda_{\text{max}}(\widetilde{\mathbf{M}}_{\tau_\ell}) =
    \frac{|\tau_\ell|}{3}. 
  \end{equation*}
  Therefore, we see that $\widetilde{\mathbf{M}}_h$ is symmetric and
  positive definite. Moreover, by elemantary computations the spectral
  equivalence inequalities
  \begin{equation}\label{Eqn:lump Masse}
    \underline{c}(d) \, (\text{lump} \,
    (\mathbf{M}_h) \mathbf{u},\mathbf{u})_2 \, \leq \,
    (\widetilde{\mathbf{M}}_h\mathbf{u},\mathbf{u})_2 \, \leq \,
    (\text{lump} \, (\mathbf{M}_h)\mathbf{u},\mathbf{u})_2 
  \end{equation}
  follow, where $\text{lump} \, (\mathbf{M}_h)$ denotes the lumped mass
  matrix and 
  \begin{equation*}
    \underline{c}(d) = \begin{cases}
      \frac{1}{2}, & d=1,\\[2mm]
      \frac{5}{12}, & d=2.
    \end{cases} 
   \end{equation*}
\end{remark}

\begin{figure}[h]
  \begin{subfigure}[c]{0.5\textwidth}
    \centering
    \includegraphics[width=\textwidth]{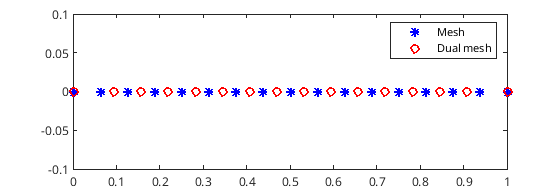}
  \end{subfigure}
  \begin{subfigure}[c]{0.5\textwidth}
    \centering
      \begin{tikzpicture}[scale=0.8]
        \filldraw[red!20] ({7.5/3}, {2.5/3}) -- (1.25,1.25) -- ({2.5/3}, {7.5/3}) -- (1.25,3.75) -- ({7.5/3}, {12.5/3}) -- (3.75,3.75) -- ({12.5/3}, {7.5/3}) -- (3.75,1.25) -- (7.5/3,2.5/3);

        \filldraw[blue!20] (0,0) -- (2.5,0) -- (2.5,2.5/3) -- (1.25,1.25) -- (2.5/3,2.5) -- (0,2.5) -- (0,0);

        \draw (0,0) rectangle (5,5);
        \draw (0,0) -- (5,5);
        \draw (0,5) -- (5,0);
        \filldraw (5,5) circle (2pt);
        \filldraw (0,5) circle (2pt);
        \filldraw (5,0) circle (2pt);
        \filldraw[red] (2.5,2.5) circle (4pt) node[anchor=west] {$x_i\in\Omega$};
        \filldraw[blue] (0,0) circle (4pt) node[anchor=north] {$x_j\in \partial\Omega$};

        \filldraw[blue] (7.5/3,2.5/3) circle (2pt);
        \filldraw[red] (7.5/3,2.5/3) circle (1pt);
        \filldraw[blue] (1.25,1.25) circle (2pt); 
        \filldraw[red] (1.25,1.25) circle (1pt);
        \filldraw[blue] (2.5/3,7.5/3) circle (2pt);
        \filldraw[red] (2.5/3,7.5/3) circle (1pt);
        \filldraw[red] (1.25,3.75) circle (2pt);
        \filldraw[red] (7.5/3,12.5/3) circle (2pt);
        \filldraw[red] (3.75,3.75) circle (2pt);
        \filldraw[red] (12.5/3,7.5/3) circle (2pt);
        \filldraw[red] (3.75,1.25) circle (2pt);
        \filldraw[blue] (0,2.5) circle (2pt);
        \filldraw[blue] (2.5,0) circle (2pt);

        \draw[blue] (7.5/3,2.5/3) -- (1.25,1.25);
        \draw[red,dashed] (7.5/3,2.5/3) -- (1.25,1.25);
        \draw[blue] (1.25,1.25) -- (2.5/3,7.5/3);
        \draw[red,dashed] (1.25,1.25) -- (2.5/3,7.5/3);
        \draw[red] (2.5/3,7.5/3) --  (1.25,3.75);
        \draw[red] (1.25,3.75) -- (7.5/3,12.5/3);
        \draw[red] (7.5/3,12.5/3) -- (3.75,3.75);
        \draw[red] (3.75,3.75) -- (12.5/3,7.5/3);
        \draw[red] (12.5/3,7.5/3) -- (3.75,1.25);
        \draw[red] (3.75,1.25) -- (7.5/3,2.5/3);
        \draw[blue] (0,0) -- (2.5,0); 
        \draw[blue] (2.5,0) -- (7.5/3,2.5/3);
        \draw[blue] (0,0) -- (0,2.5); 
        \draw[blue] (0,2.5) --  (2.5/3,7.5/3);

        \node at (0.3, 1.6) {${\color{blue}P_j}$}; 
        \node at (2.5, 1.6) {${\color{red}P_i}$}; 
      \end{tikzpicture}
    \end{subfigure}    
    \caption{Meshes and dual meshes in 1D (left) and 2D (right).}
    \label{Fig:mesh and dual mesh}
  \end{figure}
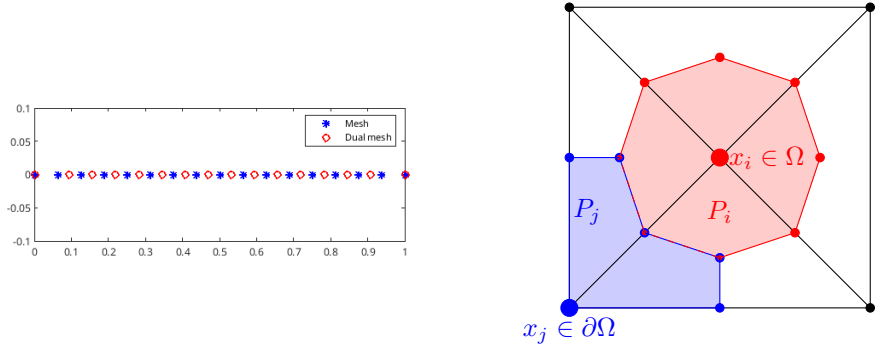

\subsection{Solvers and their use in nested iteration}
\label{Subsection:Application:EnergyRegularisation:NestedIteration}
We start this subsection with the observation that we can write all linear
systems to be solved in a unified manner as
\begin{equation}
\label{Eqn:Application:System:Sy=rhs}
 \mathbf{S}_{\varrho h}\mathbf{y}_{\varrho,h} = \overline{\mathbf{y}}_h
\end{equation}
where $\mathbf{S}_{\varrho h} = \mathbf{M}_h + \mathbf{D}_{\varrho h}$ with
\[
  \mathbf{D}_{\varrho,h}  \, = \,
  \left \{
    \begin{array}{ccl}
      \varrho \, \mathbf{K}_h
        & & \mbox{for $H^{-1}$ regularization}, \varrho = h^2, \\[1mm]
      \mathbf{K}_{\varrho_h,h}
      & & \mbox{for variable $H^{-1}$ regularization},
          \varrho_h(x) = h_\tau^2, x \in \tau, \\[1mm]
      \varrho \, \mathbf{K}_h \mathbf{M}_h^{-1} \mathbf{K}_h
      & & \mbox{for $L^2$ regularization}, \varrho=h^4, \\[1mm]
      \mathbf{K}_h \mathbf{M}_{1/\varrho_h,h}^{-1} \mathbf{K}_h
      && \mbox{for variable $L^2$ regularization},
         \varrho_h(x) = h_\tau^4, x \in \tau .
    \end{array}
  \right.
\]
We note that, for constant regularization functions
$\varrho_h(x) = \varrho$ for all $x \in \Omega$, we obtain
$\mathbf{K}_{\varrho_h h} = \varrho \mathbf{K}_h$ as well as
$\mathbf{M}_{1/\varrho_h h} = \frac{1}{\varrho} \mathbf{M}_h$. Hence, it is
sufficient to consider the variable regularizations only. For the system
matrix $\mathbf{S}_{\varrho,h} = \mathbf{M}_h + \mathbf{D}_{\varrho h}$ of
\eqref{Eqn:Application:System:Sy=rhs}, we can prove the following lemma for
all regularizations discussed above.

\begin{lemma}\label{Lemma:Application:SpectralEquivalenceInequalitiesCSC}
  There hold the spectral equivalence inequalities
  \begin{equation}\label{Eqn:Application:SpectralEquivalenceInequalitiesCSC}
    c_1 \, ( \mathbf{C}_h \mathbf{y}_h , \mathbf{y}_h) \, \leq \,
    (\mathbf{S}_{\varrho h} \mathbf{y}_h , \mathbf{y}_h) \, \leq \,
    c_2 \, (\mathbf{C}_h \mathbf{y}_h, \mathbf{y}_h) 
\end{equation}    
for all $\mathbf{y}_h \in {\mathbb{R}}^M$, cf.
\eqref{Eqn:AbstractOCP:SpectralEquivalenceInequalities}, where the
preconditioner $\mathbf{C}_h = \mbox{lump} \, (\mathbf{M}_h)$ is a simple
diagonal matrix that is obtained from the mass matrix $\mathbf{M}_h$ by
mass lumping, i.e.,
\begin{equation*}
  C_h[j,k] = \mbox{lump}(M_h)[j,k] =
  \delta_{j,k} \sum_{i=1}^M M_h[j,i], \quad j,k = 1,\ldots,M;
\end{equation*}
$c_1 = 1/(d+2)$, and $c_2 = 1 + c^2$ with the $h$-independent
$c^2 \ge \lambda_{\max}(\mathbf{M}_h^{-1}\mathbf{D}_{\varrho h})$
being the maximal eigenvalue of the generalized eigenvalue problem
$\mathbf{D}_{\varrho h} \mathbf{v}_h = \lambda \, \mathbf{M}_h\mathbf{v}_h$ 
or, at least, an upper bound of it.
\end{lemma}

\begin{proof}
The lower estimate follows from the inequalities 
$\mathbf{S}_{\varrho h} = \mathbf{M}_h + \mathbf{D}_{\varrho h} \ge
\mathbf{M}_h \ge (d+2)^{-1} \mbox{lump} \, (\mathbf{M}_h)$, where the last
estimate can be found in \cite[Lemma~1]{LangerLoescherSteinbachYang:2024NLA}. 
The upper estimate can be obtained from 
$\mathbf{S}_{\varrho h} = \mathbf{M}_h + \mathbf{D}_{\varrho h} \le
(1 + c^2) \mathbf{M}_h \le \linebreak
(1 + c^2)\, \mbox{lump} \, (\mathbf{M}_h)$.
The estimate $\mathbf{D}_{\varrho h} \le c^2\, \mathbf{M}_h$ follows from
local inverse inequalities and appropriate choices of the regularization
parameter or function $\varrho$ as given above for different regularizations.
We refer to  \cite{LangerLoescherSteinbachYang:2024NLA} for a detailed proof. 
We mention that $c^2 = \lambda_{\max}(\mathbf{M}_h^{-1}\mathbf{D}_{\varrho h})$
is the best possible constant.
\end{proof}

\begin{remark}
The mass matrices $\mathbf{M}_h$ and $\mathbf{M}_{1/\varrho_h,h}$ in 
$\mathbf{D}_{\varrho h}= \mathbf{K}_h \mathbf{M}_h^{-1} \mathbf{K}_h$ and
$\mathbf{D}_{\varrho h} = \mathbf{K}_h \mathbf{M}_{1/\varrho_h,h}^{-1}
\mathbf{K}_h$, respectively, can be replaced by the corresponding lumped
versions $\mbox{lump}(\mathbf{M}_h)$ and
$\mbox{lump}(\mathbf{M}_{1/\varrho_h,h})$ without affecting the discretization
error and the spectral equivalence inequalities; see
\cite{LangerLoescherSteinbachYang:2024NLA}. We note that this replacement of
the mass matrix by their lumped versions makes the matrix-vector
multiplication $\mathbf{S}_{\varrho h} * \mathbf{y}_h$ fast. More precisely,
$\mathbf{S}_{\varrho h} * \mathbf{y}_h$ can be performed in optimal complexity
${\mathcal{O}}(h^{-d})$. This is important when solving the
system \eqref{Eqn:Application:System:Sy=rhs} by pcg as we do in the nested
iteration procedure presented in Algorithm~\ref{Algorithm1}.
\end{remark}

In order to  recover the control $\widetilde{u}_{\varrho h} \leftrightarrow
\widetilde{\mathbf u}_{\varrho h} \in \mathbb{R}^M$, we have to solve the
system \eqref{Eqn:AbstractOCP:AlgebraicSystem4Control} with
$\mathbf{B}_h = \mathbf{K}_h$ and $\overline{\mathbf{M}}_h = \mathbf{M}_h$
or $\overline{\mathbf{M}}_h = \widetilde{\mathbf{M}}_h$. In both cases,
$\overline{\mathbf{C}}_h = \mbox{lump}(\mathbf{M}_h)$
is an asymptotically optimal preconditioner, see \eqref{Eqn:lump Masse}.

Let us now specify the nested iteration and, in particular,
Algorithm~\ref{Algorithm1}, descried in
Subsection~\ref{Subsection:Application:EnergyRegularisation:NestedIteration}
for abstract optimal control problems, in the special case of distributed
control of the Poission equation with energy regularization 
presented in Subsection~\ref{Subsection:Application:EnergyRegularization}.
Let us assume that $\mathcal{T}_\ell = \mathcal{T}_{h_\ell}$ is a sequence 
of uniformly (or adaptively) refined simpicial, shape regular meshes with the 
mesh sizes $h = h_\ell$, $\ell = 1,\ldots,L$, and $Y_\ell$, $X_\ell$, $U_\ell$
are corresponding finite element spaces as described in
Subsections~\ref{Subsection:Application:EnergyRegularization} and
\ref{Subsection:Application:EnergyRegularisation:ControlRecovering}. We
recall that here $X_\ell = Y_\ell = \mbox{span} \{ \varphi_i \}_{i=1}^{M=N}
\subset X=Y = H^1_0(\Omega)$.
Line 11: $\mathbf{y}_\ell \leftarrow \mathbf{S}_\ell^{-1}\mathbf{y}_\ell$ 
in Algorithm~\ref{Algorithm1} now means that the system
\eqref{Eqn:AbstractOCP:AlgebraicSystem4State} is solved by the pcg
iteration with the preconditioner $\mathbf{C}_h = \mbox{lump}(\mathbf{M}_h)$
and the initial guess
$\mathbf{y}_\ell^0 = \mathbf{I}_{\ell-1}^\ell \mathbf{y}_{\ell-1}^n$ 
that is simply interpolated from the last iterate on the coarser mesh
$\mathcal{T}_{\ell-1}$. It is clear that we need a constant number $n$ of
nested iterations on all levels $\ell = 2,\ldots,L$ in order to match the
discretization error \eqref{Eqn:Application:ErrorEstimateH^s}. The coarse
mesh system $\mathbf{S}_1 \mathbf{y}_1 = \overline{\mathbf{y}}_1$
(line 7 in in Algorithm~\ref{Algorithm1}) is usually solved by some sparse
direct method \cite{Davis2006Book}, but it can be solved by pcg with the
same preconditioner and the initial guess $\mathbf{y}_1^0 = \mathbf{0}_1$. 
This immediately yields that we need $\ln h_1^{-1}$ pcg iteration in order
to match the discretization error estimate
\eqref{Eqn:Application:ErrorEstimateH^s} for $h=h_1$.

\subsection{State constraints}
\label{Subsection:Application:EnergyRegularisation:Constraints}
We now consider the minimization of
\eqref{Eqn:Application:PoissonFunctional Energy} subject to the
Poisson equation \eqref{Eqn:Application:PoissonDBVP} with
constraints on the state
$y_\varrho \in K_s := \{ y \in H^1_0(\Omega) : g_- \leq y \leq g_+ \;
\mbox{a.e. in} \; \Omega \}$,
where $g_\pm \in H^1_0(\Omega,\Delta)$ are given barrier functions, and
where we assume $g_- \leq g_+$ and $0 \in K_s$ to be satisfied.
The solution $y_\varrho \in K_s$ of this minimization problem
is then characterized as the unique solution of the variational
inequality
\begin{equation}
  \langle y_\varrho , y-y_\varrho \rangle_{L^2(\Omega)} + \varrho \,
  \langle \nabla y_\varrho , \nabla (y-y_\varrho) \rangle_{L^2(\Omega)}
  \geq \langle \overline{y}, y-y_\varrho \rangle_{L^2(\Omega)}
  \quad \mbox{for all} \; y \in K_s .
\end{equation}
This variational inequality completely corresponds to \eqref{Abstract VI}
as considered in the abstract setting. Hence, all results as given in
Subsection \ref{Subsection:AbstractOCP:Constraints} remain true.
Instead of the linear system \eqref{Eqn:AbstractOCP:AlgebraicSystem4State}
we now have to solve a discrete variational inequality to find
${\mathbf y}_{\varrho h} \in {\mathbb{R}}^M \leftrightarrow
y_{\varrho h} \in K_{s,h}$ such that
\begin{equation}
\label{Eqn:AbstractOCP:DiscreteVariationalInequality4StateConstraints}
  (({\mathbf M}_h + \varrho \, {\mathbf K}_h) {\mathbf y}_{\varrho h} -
  \overline{\mathbf y}_h, {\mathbf{y}-{\mathbf{y}_{\varrho h}}}) \geq 0
\end{equation}
is satisfied for all ${\mathbf y} \in {\mathbb{R}}^M \leftrightarrow
y_h \in K_{s,h}$. We define the discrete Lagrange multiplier
$\mbox{\boldmath $\lambda$} :=
({\mathbf M}_h + \varrho \, {\mathbf K}_h) {\mathbf y}_{\varrho h} -
\overline{\mathbf y}_h$, and the index set of the active nodes,
$I_{s,\pm} := \{ k:=1,\ldots,M: y_k = g_{\pm,k} := g_\pm(x_k)\}$. With this
we then conclude the discrete complementarity conditions
\[
  \lambda_k = 0, \;
  g_{-,k} < y_k < g_{+,k} \; \mbox{for} \; k \not\in I_{s,\pm}, \;
  \lambda_k \leq 0 \; \mbox{for} \; k \in I_{s,+}, \;
  \lambda_k \geq 0 \; \mbox{for} \; k \in I_{s,-},
\]
which are equivalent to
\[
  \lambda_k = \min \{ 0 , \lambda_k + c (g_{+,k}-y_k) \} +
  \max \{ 0 , \lambda_k + c (g_{-,k}-y_k) \}, \quad c > 0.
\]
Hence we have to solve a system
${\mathbf F}({\mathbf y}_{\varrho h},\mbox{\boldmath $\lambda$})=
{\mathbf 0}$ of (non)linear equations
\begin{eqnarray*}
  {\mathbf F}_1({\mathbf y}_{\varrho h},\mbox{\boldmath $\lambda$})
  & = &
  ({\mathbf M}_h + \varrho \, {\mathbf K}_h) {\mathbf y}_{\varrho h} -
  \overline{\mathbf y}_h - \mbox{\boldmath $\lambda$} =
  {\mathbf 0}, \\
  {\mathbf F}_2({\mathbf y}_{\varrho h},\mbox{\boldmath $\lambda$})
  & = &
  \mbox{\boldmath $\lambda$} -
  \min \{ 0 , \mbox{\boldmath $\lambda$} +
  c ( {\mathbf g}_+ - {\mathbf y}) \} +
    \max \{ 0 , \mbox{\boldmath $\lambda$} + c
    ({\mathbf{g}_-- {\mathbf y}) \}
   = {\mathbf 0}},
\end{eqnarray*}
where the latter have to be considered componentwise. For any given
$\mbox{\boldmath $\lambda$}$ the system
$ {\mathbf F}_1({\mathbf y}_{\varrho h},\mbox{\boldmath $\lambda$}) =
{\mathbf 0}$  reads
\[
 ({\mathbf M}_h + \varrho \, {\mathbf K}_h) {\mathbf y}_{\varrho h} =
  \overline{\mathbf y}_h + \mbox{\boldmath $\lambda$} 
\]
which can be solved as in the unconstrained case, and it remains to
solve the nonlinear system
\begin{equation}\label{Eqn:Nonlinear state constraints}
  {\mathbf F}_2(({\mathbf M}_h + \varrho \, {\mathbf K}_h)^{-1}
  (\overline{\mathbf y}_h + \mbox{\boldmath $\lambda$}),
  \mbox{\boldmath $\lambda$})
  =
  {\mathbf{0}} .
\end{equation}
For the solution of \eqref{Eqn:Nonlinear state constraints} we can apply
a semi-smooth Newton method which is equivalent to a primal-dual active
set strategy, see, e.g., \cite{ChenNashedQi:2001,HintermuellerItoKunisch2003SJO,
  HintermuellerUlbrich:2004,ItoKunisch:2008}, and
\cite{GanglLoescherSteinbach2025CAMWA}. Instead of solving the nonlinear
system \eqref{Eqn:Nonlinear state constraints} we can solve the
variational inequality 
\eqref{Eqn:AbstractOCP:DiscreteVariationalInequality4StateConstraints}
by using multigrid methods, see \cite{GraeserKornhuber2009JCM} for
an overview of related methods. This will be a topic of future research.
When considering control constraints we replace
$K_s$ by 
\begin{eqnarray*}
  K_c & := & \Big \{ y \in H^1_0(\Omega) : \langle f_- ,
             \phi \rangle_{L^2(\Omega)}
             \leq \langle \nabla y , \nabla \phi \rangle_{L^2(\Omega)} \leq
             \langle f_+ , \phi \rangle_{L^2(\Omega)} \\
      && \hspace*{3.5cm} \mbox{for all} \; \phi \in H^1_0(\Omega) \;
         \mbox{with} \; \phi \geq 0 \; \mbox{a.e. in} \; \Omega \Big \},
\end{eqnarray*}
where we assume $f_\pm \in L^2(\Omega)$. For a more detailed discussion
we refer to \cite{GanglLoescherSteinbach2025CAMWA},
see also \cite{GongTan2025JSC}.

\subsection{Numerical results}\label{Subsection:NumericalResults}
We first reconsider the 1d examples from the introduction. 
Since we can analytically solve all of these 1d OCPs, we can easily verify 
the numerical results for both the $L^2$ and the $H^{-1}$ regularization 
with respect to the accuracy of the approximation of the computed finite 
element state to the target and the approximation of the cost of the control.
Furthermore, we numerically study three multi-dimensional examples with targets
possessing different features. The first two examples are taken from 
\cite{ClarsonKunisch:2011ESAIM:COCV}, where beside the standard $L^2$
regularization also other regularizations including measure and BV
regularizations are studied both theoretically and numerically. These two
benchmark examples from \cite{ClarsonKunisch:2011ESAIM:COCV} are given in
the two-dimensional (2d) computational domain $\Omega = (-1,1)^2$. Here we
also consider the three-dimensional counterparts given in $\Omega = (-1,1)^3$.
Finally, we numerically study a three-dimensional (3d) example with a more
complicated discontinuous target that was already used in our paper 
\cite{LangerLoescherSteinbachYang:2024CAMWA} for numerical tests. In this
example, the target is zero with exception of several small inclusions which
are nothing but hot spots. For the three multi-dimensional examples, 
we always us the $H^{-1}$ regularization, which is sometimes also called
energy regularization, as described in
Subsection~\ref{Subsection:Application:EnergyRegularization}.
The finite element discretization, the control recovering, the solvers,
and the nested iteration procedure also follow the description as given 
in Subsections~\ref{Subsection:Application:EnergyRegularization},
\ref{Subsection:Application:EnergyRegularisation:ControlRecovering},
and \ref{Subsection:Application:EnergyRegularisation:NestedIteration}.
In particular, we solve the system~\eqref{Eqn:Application:System:Sy=rhs}
by pcg preconditioned by the lumped mass matrix
$\mathbf{C}_h = \mbox{lump}(\mathbf{M}_h)$. In the non-nested version, the
pcg iterations are stopped as soon as the 
$\mathbf{S}_{\varrho h}^T \mathbf{C}_h^{-1} \mathbf{S}_{\varrho h}$ energy norm
of the initial iteration error 
$\mathbf{e}_h^0 = \mathbf{u}_{\varrho h} - \mathbf{u}_{\varrho h}^0$
is reduced by a factor of $10^6$. In terms of the residual 
$\mathbf{r}_h^n = \mathbf{S}_{\varrho h} \mathbf{e}_h^n$,
the stopping criterion can be written in the form
\begin{equation*}
  (\mathbf{C}_h^{-1}\mathbf{r}_h^n,\mathbf{r}_h^n)^{1/2}
  \, \le \, 10^{-6} \, (\mathbf{C}_h^{-1}\mathbf{r}_h^0,\mathbf{r}_h^0)^{1/2}. 
\end{equation*}
We always use a zero initial guess in the nonnested iterations while, in the
nested iteration procedure, the initial guess is interpolated from the coarser
mesh and the iteration is stopped when the discretization error is reached.

\subsubsection{Numerical justification of the theoretical results in 1d}
\label{Subsubsection:NumericalResults:1d}
We reconsider the examples from the introduction, especially the smooth
target \eqref{Eqn:Introduction:target-u1} and the discontinuous target
\eqref{Eqn:Introduction:target-u3} for which we computed the exact solutions
$y_{1,\varrho}$ and $y_{3,\varrho}$ for \eqref{Eqn:Application:VF} and
\eqref{Eqn:Application:L2 BiLaplace} depending on $\varrho>0$ explicitely.

Let us consider piecewise linear finite elements, defined on the decomposition
of $(0,1)$ into equidistant nodes $x_k = k/N$, $k=0,\ldots,N$, for some
$N\in\mathbb{N}$. Given the mesh size $h=1/N$, the optimal choice of the
regularization parameter is $\varrho = \varrho_{H^{-1}} = h^2$ and
$\varrho = \varrho_{L^2}=h^4$ in the case of the $H^{-1}$ regularization and
$L^2$ regularization, respectively. Now let us fix $\varrho=2^{-\ell}$ for
some $\ell\in\mathbb{N}$ and choose two decompositions of mesh sizes
$h_{H^{-1}}$ and $h_{L^2}$, such that $\varrho = h_{H^{-1}}^2 = h_{L^2}^4$.
On these meshes we compute the corresponding states $y_{i,\varrho h_{H^{-1}}}$
and $y_{i, \varrho h_{L^2}}$, $i=1,3$ of the $H^{-1}$ regularization and the
$L^2$ regularization, solving \eqref{Eqn:Application:FEM} and
\eqref{Eqn:Application:L2 System}, respectively. The states are plotted in
Figure~\ref{Fig:states-and-controls}. We clearly observe, that the
reconstructed states approximate the exact states very well. We note that
the larger distance of the target to the state reconstructed using the $L^2$
regularization with respect to the state reconstructed by the $H^{-1}$
regularization does not stem from the state being computed on a coarser mesh,
but is rather inherited from the continuous problem. This is further supported
by Figure~\ref{Fig:error-discrete-and-continuous}, where the exact errors
$\norm{\overline{y}-y_\varrho}_{L^2(0,1)}$ are plotted against the finite
element errors $\norm{\overline{y}-y_{\varrho h}}_{L^2(0,1)}$.  

In a post processing step, we now compute the reconstruction of the control
$\widetilde{u}_{\varrho h}\in X_h$ on the primal mesh and
$\widetilde{u}_{\varrho h}^d\in U_h^d$ on the
dual mesh, solving \eqref{Eqn:Application:Control Reconstruction}. The results
are presented in Figure~\ref{Fig:states-and-controls}. 
In the one dimensional case $\Omega = (0,1)$, we can compute the exact cost,
by introducing $\widetilde{y}_\varrho\in H^1_0(0,1)$ given as 
\begin{equation*}
  \widetilde{y}_\varrho(x) = \int_0^1 G(x,y)\widetilde{u}_{\varrho h}(y)\, dy, 
\end{equation*}
where
\begin{equation*}
  G(x,y) = \begin{cases}
    y \, (1-x),& y\in(0,x)\\
    x \, (1-y), & y\in(x,1),
  \end{cases}
\end{equation*}
denotes the Greens function, i.e., $-\widetilde{y}_\varrho'' =
\widetilde{u}_{\varrho h}$
in $(0,1)$. Then we compute 
\begin{align*}
  \norm{\widetilde{u}_{\varrho h}}_{H^{-1}(0,1)}
  & = \sup_{0\neq v\in H^1_0(0,1)}
    \frac{\langle \widetilde{u}_{\varrho h},v\rangle_{L^2(0,1)} }{\norm{v'}_{L^2(0,1)}}
    = \sup_{0\neq v\in H^1_0(0,1)}\frac{\langle -\widetilde{y}_\varrho'',
    v\rangle_{L^2(0,1)} }{\norm{v'}_{L^2(0,1)}}\\
  & = \sup_{0\neq v\in H^1_0(0,1)}
    \frac{\langle \widetilde{y}_\varrho',v'\rangle_{L^2(0,1)}}
    {\norm{v'}_{L^2(0,1)}}
    = \norm{\widetilde{y}_\varrho'}_{L^2(0,1)} =
    \sqrt{\langle \widetilde{u}_{\varrho h},
    \widetilde{y}_\varrho\rangle_{L^2(0,1)} }.  
\end{align*} 
The cost of the different control reconstructions is compared with the exact
cost $\norm{u_{\varrho}}_{H^{-1}(0,1)}$ and $\norm{u_\varrho}_{L^2(0,1)}$ in 
Figures~\ref{Fig:cost-discrete-and-continuous-H} and
\ref{Fig:cost-discrete-and-continuous-L}, respectively.
We note that all the computations align very well and fit the exact cost.
Although it might seem that the $L^2$ regularization comes with lower cost,
we again stress that for fixed $\varrho >0$ the error
$\norm{\overline{y}-y_{\varrho}}_{L^2(0,1)}$ for the $L^2$ regularization is
larger compared to the $H^{-1}$ regularization. Thus, to achieve the same
level of accuracy one needs to consider a smaller regularization parameter
$\varrho$ leading to the same cost for the $L^2$ regularization. Thus, for
the implementation the $H^{-1}$ regularization is beneficial, as the
regularization parameter does not need to be chosen too small. Hence, if
one is interested in the $L^2$ cost, we propose to compute the state via the
$H^{-1}$ regularization, then reconstruct the control and compute the $L^2$
cost and check if it is still below the threshold given by the application.   

\begin{figure}[htbp!]
  \centering
  \begin{subfigure}{0.45\textwidth}
    \includegraphics[width=\textwidth]
    {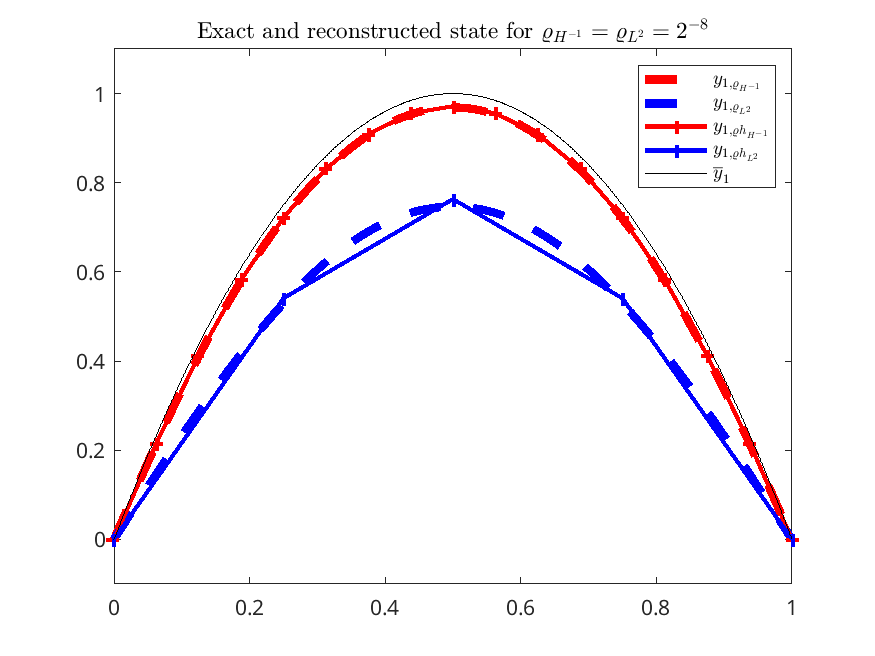}
    \caption{$\overline{y}_1$, $y_{1,\varrho}$ and $y_{1,\varrho h}$}
  \end{subfigure}
  \hfill
  \begin{subfigure}{0.45\textwidth}
    \includegraphics[width=\textwidth]
    {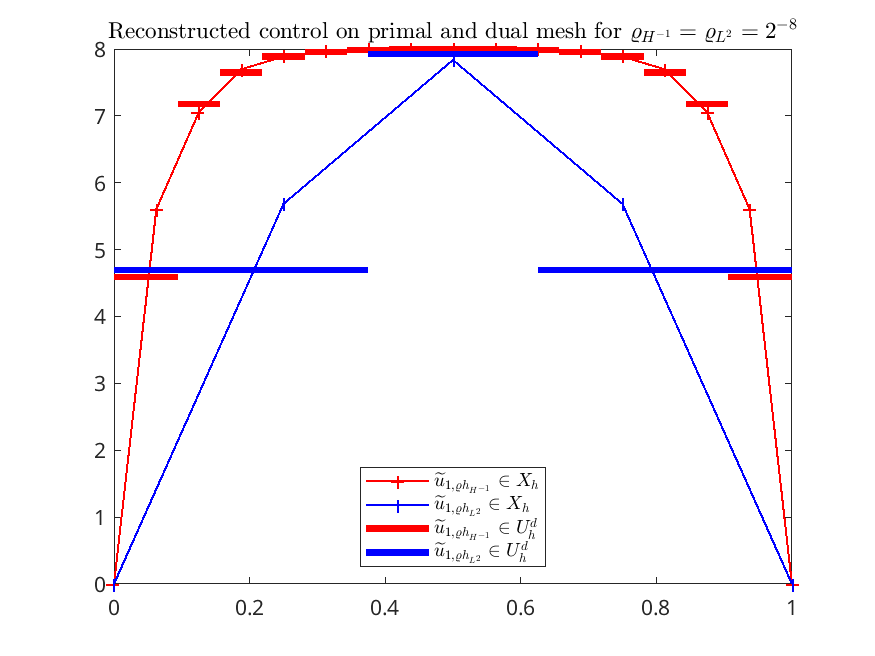}
    \caption{$\widetilde{u}_{1,\varrho h}$}
  \end{subfigure}\\
  \begin{subfigure}{0.45\textwidth}
    \includegraphics[width=\textwidth]
    {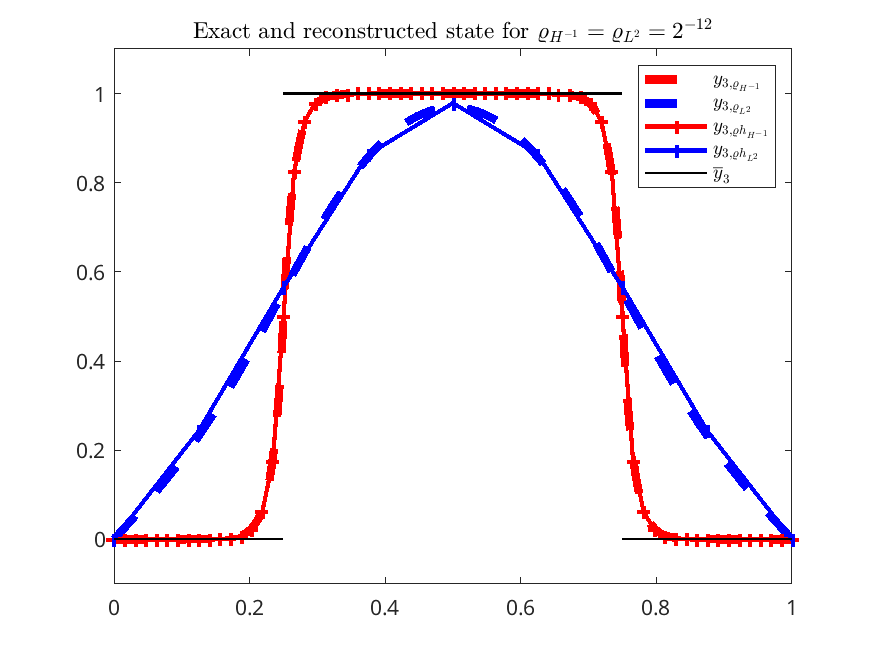}
    \caption{$\overline{y}_3$, $y_{3,\varrho}$ and $y_{3,\varrho h}$}
  \end{subfigure}
  \hfill
  \begin{subfigure}{0.45\textwidth}
    \includegraphics[width=\textwidth]
    {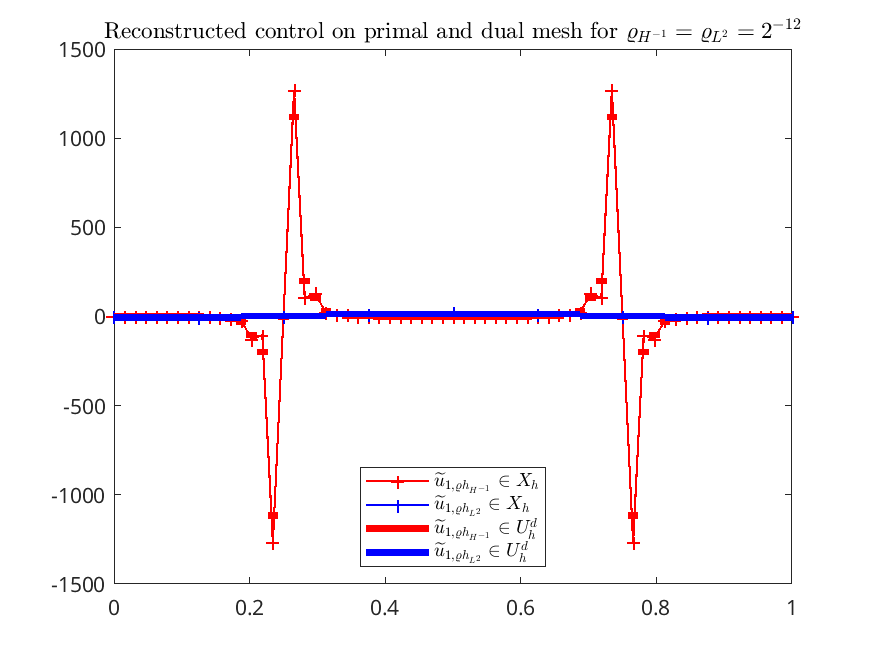}
    \caption{$\widetilde{u}_{3,\varrho h}$}
  \end{subfigure}
  \caption{Targets $\overline{y}_1$, $\overline{y}_3$, (exact) reconstructed
    states using the $H^{-1}$ and $L^2$ regularization $y_{i,\varrho}$ and
    $y_{i,\varrho h}$, respectively. And reconstruction of the controls
    $\widetilde{u}_{i,\varrho h}$ on the primal and dual mesh.}
  \label{Fig:states-and-controls} 
\end{figure}

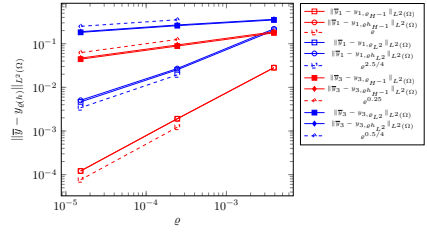
\begin{figure}[htbp!]
  \centering
    \begin{tikzpicture}[scale=0.45]
      \begin{axis}[
        xmode = log,
        ymode = log,
        xlabel=$\varrho$,
        ylabel=$\| \overline y - y_{\varrho (h)}\|_{L^2(\Omega)}$,
        legend style={font=\tiny}, legend pos = outer north east]
        \addplot[mark = square,red] table [col sep=
        &, y=error_ex, x=rho]{tables/table_cost_energy_target1.dat};
        \addlegendentry{$\|\overline{y}_1-y_{1,\varrho_{H^{-1}}}\|_{L^2(\Omega)}$}
        \addplot[mark = o,red] table [col sep=
        &, y=error, x=rho]{tables/table_cost_energy_target1.dat};
        \addlegendentry{$\|\overline{y}_1-y_{1,\varrho h_{H^{-1}}}\|_{L^2(\Omega)}$}
        \addplot[
        domain = 2^(-12):2^(-16),
        samples = 2,
        dashed,
        thin,
        red,
        mark = square] {5*x};
        \addlegendentry{$\varrho$}
        \addplot[mark = square,blue] table [col sep=
        &, y=error_ex, x=rho]{tables/table_cost_L2_target1.dat};
        \addlegendentry{$\|\overline{y}_1-y_{1,\varrho_{L^2}}\|_{L^2(\Omega)}$}
        \addplot[mark = o,blue] table [col sep=
        &, y=error, x=rho]{tables/table_cost_L2_target1.dat};
        \addlegendentry{$\|\overline{y}_1-y_{1,\varrho h_{L^2}}\|_{L^2(\Omega)}$}
        \addplot[
        domain = 2^(-12):2^(-16),
        samples = 2,
        dashed,
        thin,
        blue,
        mark = square] {3.5*x^(2.5/4)};
        \addlegendentry{$\varrho^{2.5/4}$}
        \addplot[mark = square*,red] table [col sep=
        &, y=error_ex, x=rho]{tables/table_cost_energy_target3.dat};
        \addlegendentry{$\|\overline{y}_3-y_{3,\varrho_{H^{-1}}}\|_{L^2(\Omega)}$}
        \addplot[mark = diamond*,red] table [col sep=
        &, y=error, x=rho]{tables/table_cost_energy_target3.dat};
        \addlegendentry{$\|\overline{y}_3-y_{3,\varrho h_{H^{-1}}}\|_{L^2(\Omega)}$}
        \addplot[
        domain = 2^(-12):2^(-16),
        samples = 2,
        dashed,
        thin,
        red,
        mark = diamond] {x^(0.25)};
        \addlegendentry{$\varrho^{0.25}$}
        \addplot[mark = square*,blue] table [col sep=
        &, y=error_ex, x=rho]{tables/table_cost_L2_target3.dat};
        \addlegendentry{$\|\overline{y}_3-y_{3,\varrho_{L^2} }\|_{L^2(\Omega)}$}
        \addplot[mark = diamond*,blue] table [col sep=
        &, y=error, x=rho]{tables/table_cost_L2_target3.dat};
        \addlegendentry{$\|\overline{y}_3-y_{3,\varrho h_{L^2}}\|_{L^2(\Omega)}$}
        \addplot[
        domain = 2^(-12):2^(-16),
        samples = 2,
        dashed,
        thin,
        blue,
        mark = diamond] {1*x^(0.25/2)};
        \addlegendentry{$\varrho^{0.5/4}$}  
      \end{axis}
    \end{tikzpicture}
    \caption{Errors $\norm{\overline{y}-y_\varrho}_{L^2(\Omega)}$ and
      $\norm{\overline{y}-y_{\varrho h}}_{L^2(\Omega)}$ for the different
      targets $\overline{y}_1$ and $\overline{y}_3$ and for the $H^{-1}$
      and the $L^2$ regularization. }
      \label{Fig:error-discrete-and-continuous}
    \end{figure}
    
\begin{figure}[htbp!]
  \centering
  \begin{subfigure}[t]{0.45\textwidth}
    \centering
    \begin{tikzpicture}[scale=0.45]
      \begin{axis}[
        xmode = log,
        ymode = log,
        xlabel=$\varrho$,
        ylabel=$\| \widetilde{u}_{\varrho (h)}\|_{H^{-1}(\Omega)}$,
        legend style={font=\tiny}, legend pos = outer north east]
        \addplot[mark = o,red] table [col sep=
        &, y=costenergy_H1ex, x=rho]{tables/table_cost_energy_target1.dat};
        \addlegendentry{$\|u_{1,\varrho_{H^{-1}}}\|_{H^{-1}(\Omega)}$}
        \addplot[mark = o,blue] table [col sep=
        &, y=costL2_H1ex, x=rho]{tables/table_cost_L2_target1.dat};
        \addlegendentry{$\|u_{1,\varrho_{L^2}}\|_{H^{-1}(\Omega)}$}
        \addplot[mark = square,red] table [col sep=
        &, y=costenergy_primal_H1, x=rho]{tables/table_cost_energy_target1.dat};
        \addlegendentry{$\|\widetilde{u}_{1,\varrho h_{H^{-1}}}\|_{H^{-1}(\Omega)}$}
        \addplot[mark = square,blue] table [col sep=
        &, y=costL2_primal_H1, x=rho]{tables/table_cost_L2_target1.dat};
        \addlegendentry{$\|\widetilde{u}_{1,\varrho h_{L^2}}\|_{H^{-1}(\Omega)}$}
        \addplot[mark = square*,red] table [col sep=
        &, y=costenergy_dual_H1, x=rho]{tables/table_cost_energy_target1.dat};
        \addlegendentry{$\|\widetilde{u}_{1,\varrho h_{H^{-1}}}^d\|_{H^{-1}(\Omega)}$}
        \addplot[mark = square*,blue] table [col sep=
        &, y=costL2_dual_H1, x=rho]{tables/table_cost_L2_target1.dat};
        \addlegendentry{$\|\widetilde{u}_{1,\varrho h_{L^2}}^d\|_{H^{-1}(\Omega)}$}
        \addplot[mark = diamond*,brown] table [col sep=
        &, y=yrhoH1, x=rho]{tables/table_cost_energy_target1.dat};
        \addlegendentry{$\|\nabla y_{\varrho h_{H^{-1}}}\|_{L^2(\Omega)}$}
        \end{axis}
    \end{tikzpicture}
    \caption{Cost $\|u_{1,\varrho}\|_{H^{-1}(\Omega)}$ and
      $\|\widetilde{u}_{1,\varrho h}\|_{H^{-1}(\Omega)}$}
  \end{subfigure}
  \hfill
  \begin{subfigure}[t]{0.45\textwidth}
    \centering
    \begin{tikzpicture}[scale=0.45]
      \begin{axis}[
        xmode = log,
        ymode = log,
        xlabel=$\varrho$,
        ylabel=$\| u_{\varrho (h)}\|_{H^{-1}(\Omega)}$,
        legend style={font=\tiny}, legend pos = outer north east]
        \addplot[mark = o,red] table [col sep=
        &, y=costenergy_H1ex, x=rho]{tables/table_cost_energy_target3.dat};
        \addlegendentry{$\|u_{3,\varrho_{H^{-1}}}\|_{H^{-1}(\Omega)}$}
        \addplot[mark = o,blue] table [col sep=
        &, y=costL2_H1ex, x=rho]{tables/table_cost_L2_target3.dat};
        \addlegendentry{$\|u_{3,\varrho_{L^2}}\|_{H^{-1}(\Omega)}$}
        \addplot[mark = square,red] table [col sep=
        &, y=costenergy_primal_H1, x=rho]{tables/table_cost_energy_target3.dat};
        \addlegendentry{$\|\widetilde{u}_{3,\varrho h_{H^{-1}}}\|_{H^{-1}(\Omega)}$}
        \addplot[mark = square,blue] table [col sep=
        &, y=costL2_primal_H1, x=rho]{tables/table_cost_L2_target3.dat};
        \addlegendentry{$\|\widetilde{u}_{3,\varrho h_{L^2}}\|_{H^{-1}(\Omega)}$}
        \addplot[mark = square*,red] table [col sep=
        &, y=costenergy_dual_H1, x=rho]{tables/table_cost_energy_target3.dat};
        \addlegendentry{$\|\widetilde{u}_{3,\varrho h_{H^{-1}}}^d\|_{H^{-1}(\Omega)}$}
        \addplot[mark = square*,blue] table [col sep=
        &, y=costL2_dual_H1, x=rho]{tables/table_cost_L2_target3.dat};
        \addlegendentry{$\|\widetilde{u}_{3,\varrho h_{L^2}}^d\|_{H^{-1}(\Omega)}$}
        \addplot[mark = diamond*,brown] table [col sep=
        &, y=yrhoH1, x=rho]{tables/table_cost_energy_target3.dat};
        \addlegendentry{$\|\nabla y_{\varrho h_{H^{-1}}}\|_{L^2(\Omega)}$}
        \addplot[
        domain = 2^(-12):2^(-16),
        samples = 2,
        dashed,
        thin,
        red,
        mark = diamond*,] {0.9*x^(-0.25)};
        \addlegendentry{$\varrho^{(0.5-1)/2}$}
        \addplot[
        domain = 2^(-12):2^(-16),
        samples = 2,
        dashed,
        thin,
        blue,
        mark = diamond*,] {0.9*x^(-0.125)};
        \addlegendentry{$\varrho^{(0.5-1)/4}$}
        \end{axis}
    \end{tikzpicture}
    \caption{Cost $\|u_{3,\varrho}\|_{H^{-1}(\Omega)}$ and
      $\|\widetilde{u}_{3,\varrho h}\|_{H^{-1}(\Omega)}$}
  \end{subfigure}
  \caption{Cost $\norm{u_\varrho}_{H^{-1}(\Omega)}$ and cost of the
    reconstructed control $\norm{\widetilde{u}_{\varrho h}}_{H^{-1}(\Omega)}$ for the
    targets $\overline{y}_1$ and $\overline{y}_3$ when choosing
    $\varrho = \varrho_{H^{-1}}=h_{H^{-1}}^2 = \varrho_{L^2}=h_{L^2}^4$. }
        \label{Fig:cost-discrete-and-continuous-H}
\end{figure}

\begin{figure}[htbp!]
  \centering
  \begin{subfigure}[t]{0.45\textwidth}
    \centering
    \begin{tikzpicture}[scale=0.45]
      \begin{axis}[
        xmode = log,
        ymode = log,
        xlabel=$\varrho$,
        ylabel=$\| u_{\varrho (h)}\|_{L^2(\Omega)}$,
        legend style={font=\tiny}, legend pos = outer north east]
        \addplot[mark = o,red] table [col sep=
        &, y=costenergy_L2ex, x=rho]{tables/table_cost_energy_target1.dat};
        \addlegendentry{$\|u_{1,\varrho_{H^{-1}}}\|_{L^2(\Omega)}$}
        \addplot[mark = o,blue] table [col sep=
        &, y=costL2_L2ex, x=rho]{tables/table_cost_L2_target1.dat};
        \addlegendentry{$\|u_{1,\varrho_{L^2}}\|_{L^2(\Omega)}$}
        \addplot[mark = square,red] table [col sep=
        &, y=costenergy_primal_L2, x=rho]{tables/table_cost_energy_target1.dat};
        \addlegendentry{$\|\widetilde{u}_{1,\varrho h_{H^{-1}}}\|_{L^2(\Omega)}$}
        \addplot[mark = square,blue] table [col sep=
        &, y=costL2_primal_L2, x=rho]{tables/table_cost_L2_target1.dat};
        \addlegendentry{$\|\widetilde{u}_{1,\varrho h_{L^2}}\|_{L^2(\Omega)}$}
        \addplot[mark = square*,red] table [col sep=
        &, y=costenergy_dual_L2, x=rho]{tables/table_cost_energy_target1.dat};
        \addlegendentry{$\|\widetilde{u}_{1,\varrho h_{H^{-1}}}^d\|_{L^2(\Omega)}$}
        \addplot[mark = square*,blue] table [col sep=
        &, y=costL2_dual_L2, x=rho]{tables/table_cost_L2_target1.dat};
        \addlegendentry{$\|\widetilde{u}_{1,\varrho h_{L^2}}^d\|_{L^2(\Omega)}$}
        \end{axis}
    \end{tikzpicture}
    \caption{Cost $\|u_{1,\varrho}\|_{L^2(\Omega)}$ and
      $\|\widetilde{u}_{1,\varrho h}\|_{L^2(\Omega)}$}
  \end{subfigure}
  \hfill
  \begin{subfigure}[t]{0.45\textwidth}
    \centering
    \begin{tikzpicture}[scale=0.45]
      \begin{axis}[
        xmode = log,
        ymode = log,
        xlabel=$\varrho$,
        ylabel=$\| u_{\varrho (h)}\|_{L^2(\Omega)}$,
        legend style={font=\tiny}, legend pos = outer north east]
        \addplot[mark = o,red] table [col sep=
        &, y=costenergy_L2ex, x=rho]{tables/table_cost_energy_target3.dat};
        \addlegendentry{$\|u_{3,\varrho_{H^{-1}}}\|_{L^2(\Omega)}$}
        \addplot[mark = o,blue] table [col sep=
        &, y=costL2_L2ex, x=rho]{tables/table_cost_L2_target3.dat};
        \addlegendentry{$\|u_{3,\varrho_{L^2}}\|_{L^2(\Omega)}$}
        \addplot[mark = square,red] table [col sep=
        &, y=costenergy_primal_L2, x=rho]{tables/table_cost_energy_target3.dat};
        \addlegendentry{$\|\widetilde{u}_{3,\varrho h_{H^{-1}}}\|_{L^2(\Omega)}$}
        \addplot[mark = square,blue] table [col sep=
        &, y=costL2_primal_L2, x=rho]{tables/table_cost_L2_target3.dat};
        \addlegendentry{$\|\widetilde{u}_{3,\varrho h_{L^2}}\|_{L^2(\Omega)}$}
        \addplot[mark = square*,red] table [col sep=
        &, y=costenergy_dual_L2, x=rho]{tables/table_cost_energy_target3.dat};
        \addlegendentry{$\|\widetilde{u}_{3,\varrho h_{H^{-1}}}^d\|_{L^2(\Omega)}$}
        \addplot[mark = square*,blue] table [col sep=
        &, y=costL2_dual_L2, x=rho]{tables/table_cost_L2_target3.dat};
        \addlegendentry{$\|\widetilde{u}_{3,\varrho h_{L^2}}^d\|_{L^2(\Omega)}$}
        \addplot[
        domain = 2^(-12):2^(-16),
        samples = 2,
        dashed,
        thin,
        red,
        mark = diamond*,] {x^(-0.75)};
        \addlegendentry{$\varrho^{0.5/2-1}$}
        \addplot[
        domain = 2^(-12):2^(-16),
        samples = 2,
        dashed,
        thin,
        blue,
        mark = diamond*,] {0.65*x^(-0.375)};
        \addlegendentry{$\varrho^{\tfrac{1}{2}\left(\tfrac{0.5}{2}-1\right)}$}
        \end{axis}
    \end{tikzpicture}
    \caption{Cost $\|u_{3,\varrho}\|_{L^2(\Omega)}$ and
      $\|\widetilde{u}_{3,\varrho h}\|_{L^2(\Omega)}$}
  \end{subfigure}
  \caption{Cost $\norm{u_\varrho}_{L^2(\Omega)}$ and cost of the reconstructed
    control $\norm{\widetilde{u}_{\varrho h}}_{L^2(\Omega)}$ for the targets
    $\overline{y}_1$ and $\overline{y}_3$ when choosing
    $\varrho = \varrho_{H^{-1}}=h_{H^{-1}}^2 = \varrho_{L^2}=h_{L^2}^4$.  }
        \label{Fig:cost-discrete-and-continuous-L}
\end{figure}
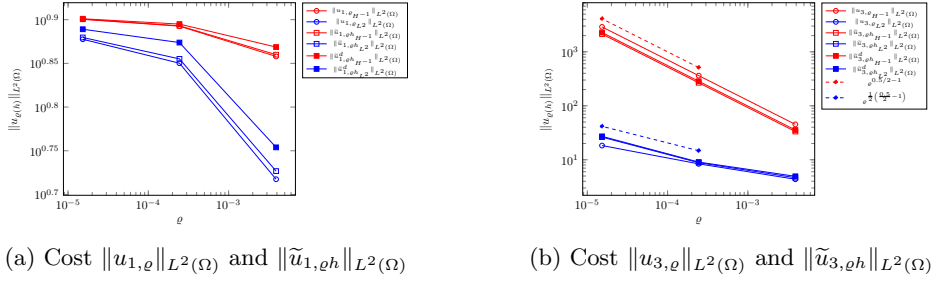

\subsubsection{Peak}
\label{Subsubsection:NumericalResults:Peak}
We now consider the smooth target
\begin{equation*}
\overline{y}(x) = e^{-50[(x_1 - 0.2)^2 + (x_2 + 0.1)^2]}
\end{equation*}
with $x=(x_1,x_2) \in \Omega = (-1,1)^2 \subset \mathbb{R}^2$;
see \cite{ClarsonKunisch:2011ESAIM:COCV}. We note that this target does not
vanish on the boundary $\partial \Omega$ of $\Omega$. 
Therefore, it does not belong to the state space $Y=H^1_0(\Omega)$.
The violation of the homogeneous boundary conditions may cause 
boundary layers which affect the convergence of the finite 
element approximation to the state. Furthermore, we consider a 3d version
of this example with the target function
\begin{equation*}
\overline{y}(x) = e^{-50[(x_1 - 0.2)^2 + (x_2 + 0.1)^2 + (x_3+0.3)^2]}
\end{equation*}
with $x=(x_1,x_2,x_3) \in \Omega = (-1,1)^3 \subset \mathbb{R}^3$.
Table~\ref{Tab:NumericalResults:solver_comparison_of_nonest_nested_peak_uniform}
presents the numerical results for the nonnested and nested iteration regimes.
We obtain the full experimental order of convergence (eoc) as we would get 
for smooth targets satisfying homogeneous Dirichlet conditions. The nested
iteration procedure produces approximations with the same accuracy 
as the nonnested iteration but several times faster. 

\begin{table}[ht]
  {\small
    \begin{tabular}{|l|r|l|r|l|l|r|r|}
      \hline
      \multirow{2}{*}{$\ell$}&\multirow{2}{*}{\#Dofs}&
      \multicolumn{3}{c|}{Non-nested}&\multicolumn{3}{c|}{Nested} \\  \cline{3-8}
      &&error&eoc &its (time) &error&eoc & its (time)\\
      \hline
      $1$&$4,913$&$3.34$e$-2$&$-$&$10$ ($2.1$e$-3$ s)&$3.34$e$-2$& $-$ &$10$($2.1$e$-3$ s)\\
      $2$&$35,937$&$1.25$e$-2$&$1.41$&$11$ ($3.8$e$-3$ s)&$1.28$e$-2$ &$1.39$ & $2$ ($8.5$e$-4$ s)\\
      $3$&$274,625$&$3.48$e$-3$&$1.85$&$11$ ($5.5$e$-3$ s)&$3.74$e$-3$&$1.77$ & $2$ ($1.2$e$-3$ s)\\
      $4$&$2,146,689$&$8.87$e$-4$&$1.97$&$11$ ($2.9$e$-2$ s)&$9.91$e$-4$&$1.92$ & $2$ ($6.5$e$-3$ s)\\
      $5$&$16,974,593$&$2.22$e$-4$&$2.00$&$11$ ($2.1$e$-1$ s)&$2.54$e$-4$&$1.96$ & $2$ ($4.7$e$-2$ s)\\
      $6$&$135,005,697$&$5.56$e$-5$&$2.00$&$11$ ($1.5$e$-0$ s)&$6.43$e$-5$&$1.98$ & $2$ ($3.7$e$-1$ s)\\
      \hline
    \end{tabular}
    \caption{Peak ($d=3$): Comparison of Nonnested and Nested iterations: $L^2$ error, 
      experimental order of convergence eoc,
      number its of pcg iterations, and computational time (time) in seconds 
      on uniform mesh refinements,
      using $256$ cores. 
    } 
    \label{Tab:NumericalResults:solver_comparison_of_nonest_nested_peak_uniform}
  }
\end{table}

\subsubsection{Pedestal}
\label{Subsubsection:NumericalResults:Pedestal}
The next example is also inspired by a 2d example used in the numerical
experiments presented in \cite{ClarsonKunisch:2011ESAIM:COCV}. The target
\[
\overline{y}(x) =
\begin{cases}
	1 & \mbox{if}\;  x \in (-1/2,1/2)^d, \\
	0 & \mbox{else},
\end{cases}
\]
is nothing than a pedestal with a plateau of the high $1$. This target
$\overline{y}$ is discontinuous, and, therefore, it does not belong 
the state space $Y=H^1_0(\Omega)$, but to the spaces $H^s(\Omega)$ 
with $s < 1/2$. Thus, we can only expect reduced convergence rates 
in the case of uniform mesh refinement. Table
\ref{Tab:NumericalResults:solver_comparison_of_nonest_nested_pedestal_uniform}
presents some numerical results for the three-dimensional case, i.e., $d=3$.
We again compare the nonnested and nested iteration regimes. Since the target
does only belong to $H^s(\Omega)$ with $s < 0.5$, we only see the reduced eoc
of about $0.5$. The nested iteration procedure again produces approximations
with the same accuracy as the nonnested iteration but several times faster. 

\begin{table}[ht]
  {\small
    \begin{tabular}{|l|r|l|r|l|l|r|r|}
      \hline
      \multirow{2}{*}{$\ell$}&\multirow{2}{*}{\#Dofs}&
      \multicolumn{3}{c|}{Non-nested}&\multicolumn{3}{c|}{Nested} \\  \cline{3-8}
      &&error&eoc &Its (Time) &error&eoc & Its (Time)\\
      \hline
      $1$&$4,913$&$3.66$e$-1$&$-$&$10$ ($2.2$e$-3$ s)&$3.66$e$-0$& $-$ &$10$($2.2$e$-3$ s)\\
      $2$&$35,937$&$2.67$e$-1$&$0.45$&$11$ ($3.2$e$-3$ s)&$2.73$e$-1$ &$0.43$ & $1$ ($5.2$e$-4$ s)\\
      $3$&$274,625$&$1.87$e$-1$&$0.52$&$11$ ($5.5$e$-3$ s)&$1.93$e$-1$&$0.50$ & $1$ ($7.5$e$-4$ s)\\
      $4$&$2,146,689$&$1.31$e$-1$&$0.51$&$11$ ($2.8$e$-2$ s)&$1.35$e$-1$&$0.52$ & $1$ ($4.5$e$-3$ s)\\
      $5$&$16,974,593$&$9.24$e$-2$&$0.51$&$11$ ($2.1$e$-1$ s)&$9.43$e$-2$&$0.52$ & $1$ ($3.5$e$-2$ s)\\
      $6$&$135,005,697$&$6.52$e$-2$&$0.50$&$11$ ($1.5$e$-0$ s)&$6.62$e$-2$&$0.51$ & $1$ ($2.5$e$-1$ s)\\
      \hline
    \end{tabular}
    \caption{Pedestal ($d=3$): Comparison of Non-nested and Nested iterations: $L^2$ error, 
      experimental order of convergence eoc,
      number its of pcg iterations, and computational time (time) in seconds 
      on uniform mesh refinements, using $256$ cores.  
    } 
    \label{Tab:NumericalResults:solver_comparison_of_nonest_nested_pedestal_uniform}
  }
\end{table}

\subsubsection{Inclusions}\label{Subsubsection:NumericalResults:Inclusions}
Finally, we consider the  target 
\[
\overline{y}(x) =
\begin{cases}
  1 & \mbox{if}\; (x_1-0.5)^2+(x_2-0.5)^2+(x_3-0.5)^2\leq 0.05^2, \\
  2 & \mbox{if}\; (x_1-0.5)^2+(x_2-0.25)^2+(x_3-0.75)^2\leq 0.0625^2,\\
  3 & \mbox{if}\; (x_1-0.5)^2+(x_2-0.75)^2+(x_3-0.75)^2\leq 0.0625^2,\\
  4 & \mbox{if}\; (x_1-0.5)^2+(x_2-0.75)^2+(x_3-0.25)^2\leq 0.075^2,\\
  5 & \mbox{if}\; x_1\in[0.25, 0.75] \;\mbox{and}\; x_2\in [0.45,0.5]\;\mbox{and}\; x_3\in[0.125, 0.375],\\
  6 & \mbox{if}\; (x_1-0.5)^2+(x_2-0.25)^2+(x_3-0.25)^2\leq 0.0625^2,\\
  0 & \mbox{else},
\end{cases}
\]
with  piecewise constant, positive values inside small inclusions in the
3d domain $\Omega=(0,1)^3$. Outside of these hot spots the target is zero.
Again, we expect interface boundary layers and reduced convergence 
rate in the case of uniform mesh refinement.
Table~\ref{Tab:NumericalResults:solver_comparison_of_nonest_nested_inclusions_uniform} 
provides the numerical results for non-nested and nested iterations. 
First we observe that the eoc is about $0.5$ that perfectly corresponds
to the regularity of the target. The nested iteration procedure again
produces approximations with the same accuracy as the non-nested iteration
but several times faster. More precisely, at the finest refinement
level $\ell = 6$ with $135,005,697$ unknowns ({\#Dofs}),
the nested iteration reaches the same accuracy (error) as the nonnested
iteration within $0.18$ seconds in comparison with $1.50$ seconds needed for
the nonnested iteration.
We note that we stopped the non-nested iteration at the relative accuracy
$10^{-6}$. This can be relaxed, and the relative accuracy can be adapted
to the discretization error.
Figure~\ref{Fig:NumericalResults:ComputedState} shows the computed finite
element approximation to the state at level $\ell= 4$.

\begin{table}[ht]
  {\small
    \begin{tabular}{|l|r|l|r|l|l|r|r|}
      \hline
      \multirow{2}{*}{$\ell$}&\multirow{2}{*}{\#Dofs}&
      \multicolumn{3}{c|}{Non-nested}&\multicolumn{3}{c|}{Nested} \\  \cline{3-8}
      &&error&eoc &Its (Time) &error&eoc & Its (Time)\\
      \hline
      $1$&$4,913$&$3.50$e$-1$&$-$&$22$ ($5.3$e$-3$ s)&$3.50$e$-1$& $-$ &$22$ ($5.1$e$-3$ s)\\
      $2$&$35,937$&$3.26$e$-1$&$0.10$&$25$ ($6.8$e$-3$ s)&$3.31$e$-1$ &$0.08$ & $1$ ($6.4$e$-4$ s)\\
      $3$&$274,625$&$2.35$e$-1$&$0.47$&$24$ ($9.3$e$-3$ s)&$2.35$e$-1$&$0.50$ & $2$ ($1.0$e$-3$ s)\\
      $4$&$2,146,689$&$1.60$e$-1$&$0.55$&$24$ ($2.2$e$-2$ s)&$1.61$e$-1$&$0.55$ & $2$ ($2.6$e$-3$ s)\\
      $5$&$16,974,593$&$1.13$e$-1$&$0.51$&$24$ ($2.2$e$-1$ s)&$1.12$e$-1$&$0.52$ & $2$ ($2.7$e$-2$ s)\\
      $6$&$135,005,697$&$7.95$e$-2$&$0.50$&$24$ ($1.5$e$-0$ s)&$7.95$e$-2$&$0.50$ & $2$ ($1.8$e$-1$ s)\\
      \hline
    \end{tabular}
    \caption{Inclusions ($d=3$): Comparison of non-nested and nested iterations: $L^2$ error, 
      experimental order of convergence eoc,
      number its of pcg iterations, and computational time (time) in seconds 
      on uniform mesh refinements, using $512$ cores. 
    } 
    \label{Tab:NumericalResults:solver_comparison_of_nonest_nested_inclusions_uniform}
  }
\end{table}

\begin{figure}[ht]
  \begin{center}
    \includegraphics[scale=.12]{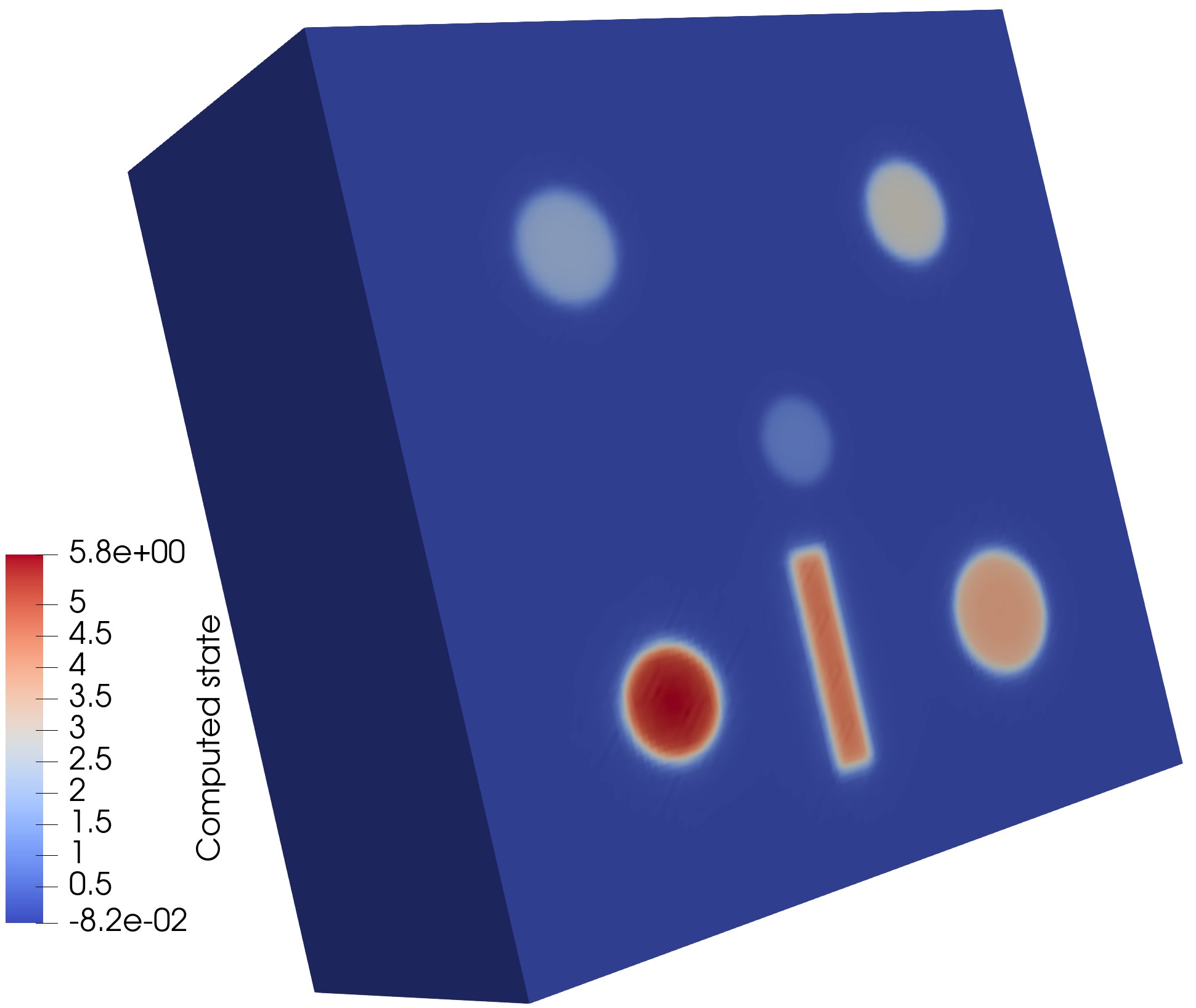}
  \end{center}
  \caption{Computed state solution on the uniform refined mesh
    with $2,146,689$ Dofs at the cut $x_3=1/2$.}
  \label{Fig:NumericalResults:ComputedState}
\end{figure}

\section{An overview on other applications}
\label{Section:OtherApplications}
In this section, we are going to discuss some selected further applications
of the abstract theory as presented in Section \ref{Section:AbstractOCP}.
%
%
\subsection{Dirichlet boundary control of the Laplace equation}
As a first example, we consider the Dirichlet boundary control problem
to minimize
\begin{equation}\label{Eqn:Dirichlet OCP Functional}
  {\mathcal{J}}(y_\varrho,u_\varrho) =
  \frac{1}{2} \, \| y_\varrho - \overline{y} \|^2_{L^2(\Omega)} +
  \frac{1}{2} \, \varrho \, | u_\varrho |^2_{H^{1/2}(\Gamma)}
\end{equation}
subject to the Dirichlet boundary value problem for the Laplace
equation
\begin{equation}\label{Eqn:Dirichlet OCP PDE}
  - \Delta y_\varrho = 0 \quad \mbox{in} \; \Omega, \quad
  y_\varrho = u_\varrho \quad \mbox{on} \; \Gamma := \partial \Omega,
\end{equation}
where the control $u_\varrho$ is now nothing but the Dirichlet data of the
state $y_\varrho$ on the boundary $\Gamma$ of $\Omega$. In this case, we
again have $H_Y = L^2(\Omega)$, but the state space $Y$ now is the
space of all harmonic functions in $H^1(\Omega)$, i.e.,
$Y := \{ y \in H^1(\Omega) :
\langle \nabla y , \nabla v \rangle_{L^2(\Omega)} = 0 \; \forall
v \in H^1_0(\Omega)\}$. The state to control map
$u_\varrho = \gamma_0^{\mbox{\footnotesize int}}y_\varrho$ is then given by
the interior Dirichlet trace operator
$\gamma_0^{\mbox{\footnotesize int}} : H^1(\Omega) \to H^{1/2}(\Gamma)$,
i.e.,
$ B = \gamma_0^{\mbox{\footnotesize int}} : Y \to U = H^{1/2}(\Gamma)$.
Moreover, we introduce
$X = H_*^{-1/2}(\Gamma) := \{ \psi \in H^{-1/2}(\Gamma) :
\langle \psi , 1 \rangle_\Gamma = 0\}$.
A semi-norm in the control space $U = H^{1/2}(\Gamma)$
is induced by the Steklov--Poincar\'e operator
$S : H^{1/2}(\Gamma) \to H^{-1/2}_*(\Gamma)$ which is defined via
\[
  |u_\varrho|_{H^{1/2}(\Gamma)}^2 = \langle S u_\varrho , u_\varrho \rangle_\Gamma =
  \int_\Gamma \frac{\partial}{\partial n_x}y_\varrho(x) \, y_\varrho(x) \, dx =
  \int_\Omega |\nabla y_\varrho(x)|^2 \, dx,
\]
where $y_\varrho \in Y$ is the harmonic extension of $u_\varrho \in U$.
With this we can write the minimization problem
\eqref{Eqn:Dirichlet OCP Functional}-\eqref{Eqn:Dirichlet OCP PDE}
as in \eqref{Eqn:AbstractOCP:AbstractFunctionalD} to minimize
\begin{equation}\label{Eqn:Dirichlet OCP reduced Functional}
  \check{\mathcal{J}}(y_\varrho) =
  \frac{1}{2} \, \| y_\varrho - \overline{y} \|^2_{L^2(\Omega)} +
  \frac{1}{2} \, \varrho \, \| \nabla y_\varrho \|^2_{L^2(\Omega)}
\end{equation}
on the space $Y$ of harmonic functions.
The minimizer $y_\varrho \in Y$ of the reduced cost functional
\eqref{Eqn:Dirichlet OCP Functional} is then charcterized as the unique
solution of the gradient equation in variational form, satisfying
\begin{equation}\label{Eqn:Dirichlet OCP gradient equation}
  \langle y_\varrho , y \rangle_{L^2(\Omega)} + \varrho \,
  \langle \nabla y_\varrho , \nabla y \rangle_{L^2(\Omega)} =
  \langle \overline{y} , y \rangle_{L^2(\Omega)} \quad
  \mbox{for all} \; y \in Y .
\end{equation}
Moreover, in this particular case, we can formulate state and control
constraints at once by defining
$Y_{s/c} := \{ y \in Y : g_- \leq y \leq g_+ \}$, where $g_\pm \in Y$ are
given (constant) barrier functions. For $y_\varrho \in Y_{s/c}$, the
minimizer of \eqref{Eqn:Dirichlet OCP reduced Functional} is then
determined as the unique solution of the variational inequality
\eqref{Abstract VI} satisfying
\begin{equation}\label{Eqn:Dirichlet OCP VI}
 \langle y_\varrho , y-y_\varrho \rangle_{L^2(\Omega)} + \varrho \,
  \langle \nabla y_\varrho , \nabla (y-y_\varrho) \rangle_{L^2(\Omega)} \geq
  \langle \overline{y} , y - y_\varrho \rangle_{L^2(\Omega)} \quad
  \mbox{for all} \; y \in Y_{s/c} .
\end{equation}
It is obvious that all regularization error estimates as given
in the abstract setting remain true.
In order to incorporate the constraints in the definition of the state
space $Y$, instead of \eqref{Eqn:Dirichlet OCP gradient equation} we
can introduce a Lagrange multiplier $p_\varrho \in H^1_0(\Omega)$ and
solve a saddle point variational formulation for
$(y_\varrho , p_\varrho) \in H^1(\Omega) \times H^1_0(\Omega)$ satisfying
\begin{equation}\label{Eqn:Dirichlet OCP Mixed}
  \begin{array}{lcl}
  \langle y_\varrho , y \rangle_{L^2(\Omega)} + \varrho \,
  \langle \nabla y_\varrho , \nabla y \rangle_{L^2(\Omega)}
  + \langle \nabla p_\varrho , \nabla y \rangle_{L^2(\Omega)}
  & = & \langle \overline{y} , y \rangle_{L^2(\Omega)}, \\[1mm]
    \langle \nabla y_\varrho , \nabla q \rangle_{L^2(\Omega)}
    & = & 0
  \end{array}
\end{equation}
for all $(y,q) \in H^1(\Omega) \times H^1_0(\Omega)$. Finite element
error estimates for the numerical solution of \eqref{Eqn:Dirichlet OCP Mixed}
follow when using standard arguments.

\begin{remark}
  The Dirichlet boundary control problem
  \eqref{Eqn:Dirichlet OCP Functional}-\eqref{Eqn:Dirichlet OCP PDE}
  in the control space $U=H^{1/2}(\Gamma)$ was first considered in
  \cite{OfPhanSteinbach:2015}, see also
  \cite{ChowdhuryGudiNandakumaran:2017,GongLiuTanYan:2019,Winkler:2020};
  for the consideration of control or state constraints, see
  \cite{GanglLoescherSteinbach2025CAMWA,GongTan2025JSC,GudiSau:2020}.
  Further extensions include Dirichlet control for Stokes flow
  \cite{GongMateosSinglerZhang:2022}, or for parabolic evolution
  equations \cite{GudiMallikSau:2022}.
\end{remark}

%
%
\subsection{Distributed control of parabolic evolution equations}
The abstract theory as given in Section \ref{Section:AbstractOCP}
is not restricted to elliptic state equations, 
but can also be applied to time-dependent PDEs.
As an example for an parabolic evolution equation 
we consider the minimization of
\[
  {\mathcal{J}}(y_\varrho,u_\varrho) =
  \frac{1}{2} \, \| y_\varrho - \overline{y} \|_{L^2(Q)}^2 +
  \frac{1}{2} \, \| u_\varrho \|^2_{L^2(0,T;H^{-1}(\Omega))}
\]
subject to the 
initial-boundary value
problem for the heat equation
\[
  \partial_t y_\varrho - \Delta_x y_\varrho = u_\varrho \quad
  \mbox{in} \, Q, \quad y_\varrho = 0 \; \mbox{on} \; \Sigma, \quad
  y_\varrho = 0 \; \mbox{on} \; \Sigma_0,
\]
where, for a given time horizon $T>0$,
$Q:=\Omega \times (0,T)$ is the space-time cylinder  
with the lateral boundary $\Sigma = \partial \Omega \times (0,T)$ 
and the bottom $\Sigma_0 = \Omega \times \{0\}$. 
In this case, we have
$H_X = H_Y = L^2(Q)$, as well as
$X = L^2(0,T;H^1_0(\Omega))$ and $U = X^* = L^2(0,T;H^{-1}(\Omega))$.
Moreover, the related state space is defined as
$Y = \{ y \in X : \partial_t y \in X^*, y(0)=0 \}$. Within this setting
we have $B = \partial_t - \Delta_x : Y \to X^*$, and
$A = - \Delta_x : X \to X^*$. The reduced optimality system then reads
to find $(p_\varrho,y_\varrho) \in X \times Y$ such that
\begin{eqnarray*}
  \frac{1}{\varrho} \, \langle \nabla_x p_\varrho ,
  \nabla_x q \rangle_{L^2(\Omega)} + \langle \partial_t y_\varrho , q \rangle_Q
  + \langle \nabla_x y_\varrho , \nabla_x y \rangle_{L^2(\Omega)}
  & = & 0, \\
  - \langle p_\varrho , \partial_t y \rangle_Q -
  \langle \nabla_x p_\varrho , \nabla_x y \rangle_{L^2(\Omega)} +
  \langle y_\varrho , y \rangle_{L^2(\Omega)}
  & = & \langle \overline{y} , y \rangle_{L^2(\Omega)}
\end{eqnarray*}
is satisfied for all $(q,y) \in X \times Y$. This variational formulation
and its space-time finite element discretization was analysed in
\cite{LangerSteinbachYang:2024ACOM}, for other approaches, see, e.g.,
\cite{BeranekReinholdUrban:2023,
  GongHinzeZhou:2012,
  LSTY:GunzburgerKunoth:2011a,
  LangerSteinbachTroeltzschYang:2021SINUM,
  LangerSteinbachTroeltzschYang:2021SISC,
  DMBV2008a,
  DanielsHinzeVierling:2015}.
While all of these approaches require the solution of a coupled
forward-backward system, a new approach related to the abstract
variational formulation \eqref{Eqn:AbstractOCP:Abstract VF} was
recently considered in \cite{LoescherReicheltSteinbach:2024},
where $D$ is induced by the norm of the anisotropic Sobolev space
$H^{1,1/2}_{0;0,}(Q) := L^2(0,T;H^1_0(\Omega)) \cap
H^{1/2}_{0,}(0,T;L^2(\Omega))$.
%
%
\subsection{Distributed control of hyperbolic evolution equations}
As an example for a hyperbolic evolution equation as state equation we
consider the minimization of
\[
  {\mathcal{J}}(y_\varrho,u_\varrho) =
  \frac{1}{2} \, \| y_\varrho - \overline{y} \|_{L^2(Q)}^2 +
  \frac{1}{2} \, \| u_\varrho \|^2_{L^2(0,T;H^{-1}(\Omega))}
\]
subject to the 
initial-boundary value problem  for the wave equation
\begin{equation}\label{Eqn:DBVP Wave}
  \partial_{tt} y_\varrho - \Delta_x y_\varrho = u_\varrho \quad
  \mbox{in} \, Q, \quad y_\varrho = 0 \; \mbox{on} \; \Sigma, \quad
  y_\varrho = \partial_t y_\varrho = 0\; \mbox{on} \; \Sigma_0.
\end{equation}
When assuming $u_\varrho \in L^2(Q)$, the initial boundary value problem
\eqref{Eqn:DBVP Wave} admits a unique solution
$y_\varrho \in H^{1,1}_{0;0,}(Q) := L^2(0,T;H^1_0(\Omega)) \cap
H^1_{0,}(0,T;L^2(\Omega))$; see, e.g.,~\cite{Ladyzhenskaya:1985,SteinbachZank:2020}.
Hoeever, this does not define an isomorphism. 
When using
$X = H^{1,1}_{0;,0}(Q) := L^2(0,T;H^1_0(\Omega)) \cap
H^1_{,0}(0,T;L^2(\Omega))$, i.e., $U = [H^{1,1}_{0;,0}(Q)]^*$, and in
order to ensure that the wave operator
$ B := \Box := \partial_{tt} - \Delta_x : Y \to X^*$ is
an isomorphism, we need to define the state space accordingly.
Therefore, and following \cite{SteinbachZank:2022}, we introduce
\[
  Y := {\mathcal{H}}_{0;0,}(Q) :=
  \overline{H^{1,1}_{0;0,}(Q)}^{\| \cdot \|_{{\mathcal{H}}(Q)}}, \quad
  \| u \|_{{\mathcal{H}}(Q)} := \sqrt{\| u \|^2_{L^2(Q)} +
  \| \Box \widetilde{u} \|^2_{[H^1_0(Q_-)]^*}},
\]
where $\widetilde{u}$ is the zero extension of $u \in L^2(Q)$ to
$Q_- := \Omega \times (-T,T)$. In this setting we can define
$A : X \to X^*$, satisfying
\[
  \langle A p,q \rangle_Q :=
  \langle \partial_t p , \partial_t q \rangle_{L^2(Q)} +
  \langle \nabla_x p , \nabla_x q \rangle_{L^2(Q)} \quad
  \mbox{for all} \; p,q \in X .
\]
Then we can write the abstract gradient equation
\eqref{Eqn:AbstractOCP:Abstract gradient equation}
as variational problem to find
$(y_\varrho,p_\varrho) \in {\mathcal{H}}_{0;0,}(Q) \times
H^{1,1}_{0;,0}(Q)$ such that
\begin{equation}\label{Eqn:OCP Wave}
  \varrho^{-1} \, \langle A p_\varrho , q \rangle_Q +
  \langle \Box \widetilde{y}_\varrho , {\mathcal{E}} q \rangle_{Q_-} = 0,
  \quad
  - \langle \Box \widetilde{y} , {\mathcal{E}} p_\varrho \rangle_{Q_-} +
  \langle y_\varrho , y \rangle_{L^2(Q)} =
  \langle \overline{y} , y \rangle_{L^2(Q)}
\end{equation}
is satisfied for all 
$(y_,q) \in {\mathcal{H}}_{0;0,}(Q) \times H^{1,1}_{0;,0}(Q)$,
where ${\mathcal{E}} : H^{1,1}_{0;,0}(Q) \to H^1_0(Q_-)$ is a suitable
extension operator, e.g., reflection in time with respect to $t=0$.
In any case, for a space-time finite element Galerkin 
discretization
of \eqref{Eqn:OCP Wave}, 
we introduce the standard finite
element spaces $Y_h = S_h^1(Q) \cap H^{1,1}_{0;0,}(Q)$
and $X_h = S_h^1(\Omega) \cap H^{1,1}_{0;,0}(Q)$ of piecewise linear
continuous basis functions. For $(y_h,q_h) \in Y_h \times X_h$ and
following \cite[Lemma 3.5]{SteinbachZank:2022} we have
\[
  \langle \Box y_h , {\mathcal{E}} q_h \rangle_{Q_-}
  = - \langle \partial_t y_h , \partial_t q_h \rangle_{L^2(Q)} +
  \langle \nabla_x y_h , \nabla_x q_h \rangle_{L^2(Q)} .
\]
Hence we have to find
$(y_{\varrho h}, p_{\varrho h}) \in Y_h \times X_h$ such that
\begin{eqnarray*}
  \langle \partial_t p_{\varrho h} , \partial_t q_h \rangle_{L^2(Q)} +
  \langle \nabla_x p_{\varrho h} , \nabla_x q_h \rangle_{L^2(Q)}
  \hspace*{2.5cm}
  && \\[1mm]
     - \langle \partial_t y_{\varrho h} , \partial_t q_h \rangle_{L^2(Q)} +     
  \langle \nabla_x y_{\varrho h} , \nabla_x q_h \rangle_{L^2(Q)}
  & = & 0, \\[2mm]
  \langle \partial_t y_h , \partial_t p_{\varrho h} \rangle_{L^2(Q)} -     
  \langle \nabla_x y_h , \nabla_x p_{\varrho h} \rangle_{L^2(Q)} +
  \langle y_{\varrho h} , y_h \rangle_{L^2(Q)}
  & = & \langle \overline{y} , y_h \rangle_{L^2(Q)}
\end{eqnarray*}
is satisfied for all $(y_h,q_h) \in Y_h \times X_h$. For a more detailed
analysis of this approach, see \cite{LoescherSteinbach:2022SINUM};
for other approaches we refer to, e.g., 
\cite{GugatKeimerLeugering:2009,KroenerKunischVexler:2011,
  MontanerMuench:2019,PeraltaKunisch:2022,Zuazua:2005}.
  
%
%
%
\section{Conclusions and outlook}\label{Section:ConclusionOutlook}
We have first considered abstract tracking-type OCPs subject to some state 
equation $By=u$ where we think of linear elliptic, parabolic, and
hyperbolic PDEs or PDE systems. If the state operator $B$ is an isomorphism
between the state space $Y$ and the control space $U$, then the OCP can be
reduced to a state-based optimality condition that is nothing but a state-based
operator equation, or to a state-based variational inequality in the case of
additional abstract constraints imposed on the state or the
control. We have started with the investigation of the case without any box
constraints. For this case, we have provided estimates of the error 
$\|y_\varrho - \overline{y} \|_{H_Y}$ between the optimal state $y_\varrho$
and the desired state $\overline{y}$ in the tracking norm $\| \cdot \|_{H_Y}$
in terms of the regularization parameter $\varrho$ and the regularity of
the desired state $\overline{y}$. After the Galerkin discretization of the
state-based operator equation, we have analysed the error 
$\|{y}_{\varrho h} - \overline{y}\|_{H_Y}$ between the Galerkin solution 
${y}_{\varrho h}$ and the desired state $\overline{y}$. We have observed that
the finite element mesh size $h$ is strongly related to the
regularization parameter $\varrho$ in order to obtain the asymptotically
optimal convergence rate. The optimal control $\widetilde{u}_{\varrho h}$ 
can be computed from the Galerkin state solution ${y}_{\varrho h}$ in a
postprocessing procedure. Moreover, we have presented efficient (parallel)
iterative solvers that can be used in a smart nested iteration process 
where we can control the accuracy of the computed state and the energy cost
of the control. Furthermore, we have shown that one can easily add and
analyse constraints for the state or the control.

In the second part of the paper, we have applied the abstract theoretical
framework to the distributed control of Poisson's equation as blueprint
for other applications. We have presented numerical results for 1d, 2d, and
3d benchmarks with different features concerning the regularity of the target. 
These numerical results not only illustrate our theoretical error estimates
quantitatively, but also demonstrate that the iterative solvers are very
efficient, in particular, in a nested iteration setting on parallel computers. 

Finally, we have briefly discussed some other applications of the abstract 
framework to PDE constrained OCP like the  Dirichlet boundary control of
the Laplace equation and OCPs subject to parabolic or hyperbolic state
equations where space-time finite elements are used for their discretization.

Further extensions include problem classes for other boundary control problems
such as Neumann or Robin type problems; OCPs with either partial observations
or partial controls, or the optimal control of non-linear state equations
such as the Navier--Stokes system.

%
%
%
\section*{Acknowledgments}
We would like to thank 
the Johann Radon Institute for Computational and Applied Mathematics (RICAM) 
for providing computing resources on the high-performance computing cluster
Radon1\footnote{https://www.oeaw.ac.at/ricam/hpc}.
Further, the financial support for the fourth author by the
Austrian Federal Ministry for Digital and Economic Affairs, the National
Foundation for Research, Technology and Development and the Christian
Doppler Research Association is gratefully acknowledged.

	


\bibliography{LLSY2024MCRF}
\bibliographystyle{abbrv} 
  

\end{document}